\documentclass[12pt]{amsart} 
\usepackage[utf8]{inputenc}
\usepackage[T1]{fontenc}	

\usepackage{cite}

\usepackage{lmodern}			

\usepackage{graphicx}	
\usepackage[plainpages=false,colorlinks,linkcolor=bleuFonce,citecolor=rougeFonce,urlcolor=vertFonce,breaklinks]{hyperref}

\usepackage{placeins}
\usepackage{float} 

\usepackage{tikz}
\usepackage{makecell}
\usepackage{mathtools}
\usetikzlibrary{patterns}
\usetikzlibrary{decorations.pathreplacing}

\usepackage[normalem]{ulem}

\allowdisplaybreaks[4]

\usepackage{color}
\definecolor{vertFonce}	{rgb}{0,0.5,0}
\definecolor{numLignes}	{rgb}{0.17,0.57,0.7}	
\definecolor{gris}		{rgb}{0.5,0.5,0.5}
\definecolor{grisFonce}	{rgb}{0.2,0.2,0.2}
\definecolor{orange}	{rgb}{1,0.65,0.31}		
\definecolor{orangeFonce}{rgb}{1,0.4,0}
\definecolor{bleuFonce}	{rgb}{0,0,0.4}
\definecolor{rougeFonce}{rgb}{0.3,0,0}
\definecolor{rougeWord}	{rgb}{0.5,0,0}
\definecolor{vertClair}	{rgb}{0.8,1,0.8}
\definecolor{rougeClair}{rgb}{1,0.5,0.5}
\definecolor{violet}	{rgb}{0.5,0,0.5}

\usepackage{pict2e}
\usepackage{multido}
\usepackage{upgreek}

\usepackage{amsfonts,amssymb,amsthm,amsmath}
\usepackage{dsfont}				
\usepackage{mathrsfs}
\usepackage[mathscr]{euscript}
\usepackage{yfonts}
\usepackage{cancel}
\usepackage{enumerate}

\usepackage{geometry}
\newtheorem{theorem}{Theorem}[section]
\newtheorem{lem}[theorem]{Lemma}

\newtheorem{prop}[theorem]{Proposition}

\newtheorem{remark}{Remark}[section]

\usepackage{stmaryrd}
\SetSymbolFont{stmry}{bold}{U}{stmry}{m}{n}
\newcommand		{\subsetArrow}	{\mathrel{\ooalign{$\subset$\cr%
\hidewidth\raise-.087ex\hbox{$_\shortrightarrow\mkern-1.5mu$}\cr}}}
\newcommand		{\subsetarrow}	{\mathrel{\ooalign{$\subset$\cr%
\hidewidth\raise-1.45ex\hbox{$\vec{}\mkern6mu$}\cr}}}

\newcommand		{\N}		{\mathbb N}			
\newcommand		{\RR}		{\mathbb R}			
\newcommand		{\R}		{\RR}

\newcommand		{\Rd}		{\R^3}
\newcommand		{\RRd}		{\R^6}

\newcommand		{\CC}		{\mathbb C}			
\newcommand		{\cH}		{\mathcal H}		
\newcommand		{\cD}		{\mathcal D}		
\newcommand		{\cM}		{\mathcal M}		
\newcommand		{\cN}		{\mathcal N}		
\newcommand		{\cK}		{\mathcal K}		

\newcommand		{\cA}		{\mathcal A}
\newcommand		{\cB}		{\mathcal B}
\newcommand		{\cE}		{\mathcal E}
\newcommand		{\cG}		{\mathcal G}
\newcommand		{\cI}		{\mathcal I}

\newcommand		{\cQ}		{\mathcal Q}
\newcommand		{\cR}		{\mathcal R}
\newcommand		{\cS}		{\mathcal S}

\newcommand		{\cU}		{\mathcal U}
\newcommand		{\cV}		{\mathcal V}
\newcommand		{\cL}		{\mathcal L}		

\newcommand		\sfA		{\mathsf A}
\newcommand		\sfB		{\mathsf B}
\newcommand		\sfC		{\mathsf C}
\newcommand		\sfD		{\mathsf D}

\newcommand		\sfJ		{\mathsf J}

\newcommand		\sfR		{\mathsf R}			

\newcommand		\sfX		{\mathsf X}			

\newcommand		\sfc		{\mathsf c}

\newcommand		{\lt}			{\left}				%
\newcommand		{\rt}			{\right}			%
\renewcommand	{\(}			{\lt(}
\renewcommand	{\)}			{\rt)}
\newcommand		{\bangle}[1]	{\lt\langle #1\rt\rangle}
\newcommand		{\weight}[1]	{\bangle{#1}}	

\newcommand		{\inprod}[2]	{\bangle{#1, #2}}
\newcommand		{\com}[1]		{\lt[{#1}\rt]}		

\newcommand		{\n}[1]			{\lt\lvert #1 \rt\rvert}
\newcommand		{\norm}[1]		{\big\lVert #1 \big\rVert}		
\newcommand		{\nrm}[1]		{\lt\lVert #1\rt\rVert}
\newcommand		{\Nrm}[2]		{\nrm{#1}_{#2}}

\renewcommand		{\d}		{\mathrm{d}}		
\newcommand			{\dd}		{\,\d}				
\newcommand			{\bd}		{\partial}			
\newcommand			{\dpt}		{\partial_t}

\newcommand			{\grad}		{\nabla}
\newcommand			{\lapl}		{\Delta}

\newcommand			{\conj}[1]	{\overline{#1}}		

\DeclareMathOperator{\cF}		{\mathcal{F}}		
\DeclareMathOperator{\re}		{Re}				
\DeclareMathOperator{\im}		{Im}				
\DeclareMathOperator{\tr}		{Tr}				
\DeclareMathOperator{\diag}		{diag}

\DeclareMathOperator{\Nor}		{Nor}

\newcommand			{\ch}		{\operatorname{ch}}
\newcommand			{\sh}		{\operatorname{sh}}

\renewcommand	{\Im}[1]		{\im\!\( #1 \)}		
\newcommand		{\Tr}[1]		{\tr\!\( #1 \)}		

\newcommand		{\nor}[1]		{\Nor\!\lt\{ #1 \rt\}}

\newcommand		{\intd}			{\int_{\Rd}}
\newcommand		{\intdd}		{\int_{\RRd}}

\newcommand		{\ii}			{\mathrm{i}}	
\newcommand		{\jj}			{\mathrm{j}}	
\newcommand		{\init}			{\mathrm{in}}

\newcommand		{\cC}			{\mathcal{C}}

\usepackage{braket}

\newcommand{\vect}[1]{\boldsymbol{\mathbf{#1}}}

\DeclareMathOperator{\dG}	{\d\varGamma}			

\newcommand		{\h}		{\mathfrak{h}}		

\newcommand		\sfK		{\mathsf K}

\newcommand		\sfS		{\mathsf S}

\newcommand     {\ecS}      {\mathscr{S}}

\newcommand{\id}{\mathds{1}}

\newcommand {\fluc}     {\mathrm{fluc}}
\newcommand {\hfluc}    {\cH_{\fluc}}

\newcommand {\asc}      {\mathfrak{a}}
\newcommand{\ol}{\overline}

\newcommand{\nv}[1]{\Vert #1\Vert }

\newcommand{\bbn}[1]{\Big\Vert #1 \Big \Vert }
\newcommand{\lr}[1]{\left\{ #1 \right\} }
\newcommand{\lrc}[1]{\left[ #1 \right] }
\newcommand{\lrs}[1]{\left( #1 \right) }

\newcommand{\lra}[1]{\langle #1 \rangle }

\newcommand{\bbabs}[1]{\Big | #1 \Big|}
\newcommand{\wq}{\infty}
\newcommand{\pa}{\partial}

\newcommand{\la}{\lambda}

\newcommand{\be}{\beta}

\newcommand{\ve}{\varepsilon}

\newcommand{\vp}{\varphi}
\newcommand{\ka}{\kappa}
\newcommand{\bu}{\vect{u}}

\newtheorem{lemma}[theorem]{Lemma}

\numberwithin{equation}{section}

\begin{document}


\title[Derivation of the Compressible Euler Equations]{Derivation of the Compressible Euler Equations from the dynamics of interacting Bose Gas in the Hard-core Limit Regime}


\author[J. Chong]{Jacky Chong}
\address{School of Mathematical Sciences, Peking University,
Beijing, China}
\email{jwchong@math.pku.edu.cn}

\author[S. Shen]{Shunlin Shen}
\address{School of Mathematical Sciences, University of Science and Technology of China, Hefei 230026,
Anhui Province, China}
\email{slshen@ustc.edu.cn}

\author[Z. Zhang]{Zhifei Zhang}
\address{School of Mathematical Sciences, Peking University,
Beijing, China}
\email{zfzhang@math.pku.edu.cn}


\begin{abstract}
We investigate the dynamics of short-range interacting Bose gases with varying degrees of diluteness and interaction strength. By applying a combined mean-field and semiclassical space-time rescaling to the dynamics in both the Gross--Pitaevskii and hard-core limit regimes, we prove that the local one-particle mass, momentum, and energy densities of the many-body system can be quantitatively approximated by solutions to the compressible Euler system in the strong sense, up to the first blow-up time of the fluid description, as the number of particles tends to infinity. In the hard-core limit regime, two novel results are presented. First, we rigorously prove, for the first time, that the internal energy of the fluid takes the form $4\pi \mathfrak{c}_{0}\rho^{2}$ (equivalently, pressure $P=2\pi \mathfrak{c}_{0}\rho^{2}$), arising solely from the kinetic energy density of the many-body system, rather than the interaction energy density, marking a fundamental difference from the Gross--Pitaevskii and other mean-field regimes. Second, the newly discovered coupling constant $\mathfrak{c}_{0}$ is the electrostatic capacity of the interaction potential, corresponding to the scattering length of the hard-core potential. Furthermore, in other limiting regimes, including those beyond the Gross--Pitaevskii regime, we find that the limiting equation is described by an eikonal system, offering a rigorous first-principle justification for using the ``geometric optics approximation'' to describe the dynamics of ultracold Bose gases.
\end{abstract}
\maketitle

\tableofcontents

\section{Introduction}

Consider a system of $N$ identical interacting bosons in $\R^3$ confined by a trapping potential with characteristic length $L_0$ that depends on $N$.
The associated $N$-body Hamiltonian is given by
\begin{align}\label{def:N-body_Hamiltonian}
    H^{\mathrm{trap}}_N = \sum^N_{j=1}\(-\frac{1}{2}\lapl_{x_j}+W_{\mathrm{trap}}(x_j)\)+\lambda\sum^N_{i<j} v(x_i-x_j)
\end{align}
where $\lambda>0$ is a coupling constant that may depend on $N$, and $v:\R^3\rightarrow [0, \infty]$ is a positive radial function with compact support $\{x\in \R^3\mid \n{x}\le R_0\}$, where $R_0$ is of order one.

We then consider the dynamics of such a system. More precisely, at initial time we switch off the trapping confinement and consider the evolution of the system, whose dynamics are governed by the $N$-body linear
Schr\"odinger equation
\begin{align}\label{eq:N-body_Schrodinger}
    &i\bd_t\psi_N = \(\sum^N_{j=1}-\frac{1}{2}\lapl_{x_j}+\lambda\sum^N_{i<j} v(x_i-x_j)\)\psi_N\\
    &\text{with } \quad \psi_N(0)=\psi_{N}^{\init} \in \cS(\R^{3 N})\ , \notag
\end{align}
where $\Nrm{\psi_N^{\init}}{L^2(\R^{3 N})}^2=1$. Here, we work in units where both the mass of each particle is equal to one and that the Planck's constant $\hbar$ is also one.

\subsection{Scalings and models}\label{section:meand_field_and_semiclassical_scaling} To obtain an effective
description of the many-body system at some macroscopic scale, we need to consider different space-time scales of the system in hope to find a scale that would yield a meaningful effective obervation of the many-body dynamics.

We begin by recasting Equation \eqref{eq:N-body_Schrodinger} in terms of dimensionless variables. Let $d=3$, $L_0>0$ be the characteristic length scale as defined above, and $T$ be some characteristic time scale. Then  we introduce the dimensionless variables
\begin{align*}
    \hat x = x/L_0, \qquad \hat t = t/T,
\end{align*}
and the rescaled interaction potential
\begin{align*}
    \widehat v(\hat z)=\frac{\lambda N\, T^2}{L_0^2}v(L_0\, \hat z)\ .
\end{align*}
To rewrite the system in terms of the dimensionless parameter
\begin{align*}
    \varepsilon = T/L_0^2
\end{align*}
and the new rescaled wave function
\begin{align*}
    \widehat \psi_N(\hat t, \hat x_1,\ldots, \hat x_N) := L_0^{dN/2}\psi_N(t, x_1, \ldots, x_N)\ ,
\end{align*}
we  multiply Equation \eqref{eq:N-body_Schrodinger}  by $T^2/L^2$ which then  yields the dimensionless equation
\begin{align*}
    i\varepsilon\,\bd_{\hat t}\widehat \psi_N = \(\sum^N_{j=1}-\frac{\varepsilon^2}{2}\lapl_{\hat x_j}+ \frac{1}{N}\sum^N_{i<j} \widehat v (\hat x_i-\hat x_j)\)\widehat\psi_N\ .
\end{align*}
Furthermore, fix some real parameters $\kappa$ and $\beta$, and consider the scales where
\begin{align}\label{eq:scaling_condition}
    \lambda NT^2/L_0^{2+d}=  \lambda N\varepsilon^2/L_0^{d-2}=\lambda N^{1-\beta} \varepsilon^{2(1-\kappa)}=:\widetilde\lambda(N)\ .
\end{align}
Equivalently, we could express $L_0, T$ of the scaling condition \eqref{eq:scaling_condition} in terms of $N, \varepsilon$
\begin{align}\label{def:characteristic_length}
    L_0 = (N^{\beta}\varepsilon^{2\kappa})^{\frac{1}{d-2}}  \quad \text{ and } \quad  T= (N^{\beta}\varepsilon^{2\kappa+\frac{d}{2}-1})^{\frac{2}{d-2}}
\end{align}
and write
\begin{align*}
    \widehat v (\hat z)= \widetilde\lambda(N)\(N^{\beta}\varepsilon^{2\kappa}\)^{\frac{d}{d-2}} v((N^{\beta}\varepsilon^{2\kappa})^{\frac{1}{d-2}}\, \hat z)\ .
\end{align*}
In this choice of scales, we see the $N^{-1}$ prefactor in front of the interaction potential, which is usually referred to as the mean-field scaling. Furthermore, we choose scales where the dimensionless parameter $\varepsilon>0$ becomes negligible\footnote{$A(N)\ll B(N)$ means $\lim_{N\rightarrow \infty}\frac{A(N)}{B(N)}\rightarrow 0$.}, i.e., $T\ll N^{\frac{2\beta}{d-2}}$. We shall refer to this set of scales as the semiclassical scalings. Lastly, we also restrict ourselves to models where $N^{\beta}\varepsilon^{2\kappa+\frac{d}{2}-1}\gg 1$ so that $(\widehat x, \widehat t\,)$ are indeed macroscopic variables.

In light of the above scaling argument, let us introduce the following dimensionless $N$-body Schr\"odinger equation
\begin{align}\label{eq:dimensionless_N-body}
    i\varepsilon\,\bd_{t}\psi_N^\varepsilon = \(\sum^N_{j=1}-\frac{\varepsilon^2}{2}\lapl_{x_j}+ \sum^N_{i<j}  U^\varepsilon_N \(x_i- x_j\)\)\psi_N =: \widehat H_{N}\psi_N^\varepsilon
\end{align}
with the interaction potential
\begin{align}\label{def:rescaled_interaction_potential}
   U_N^\varepsilon(x):= \tfrac{1}{N}V_{N}^\varepsilon(x)\quad \text{ where } \quad V_{N}^{\varepsilon}(x):=\lambda(N) (N^{\beta}\varepsilon^{2\kappa})^3v(N^{\beta}\varepsilon^{2\kappa}\, x)
\end{align}
where $\beta> 0$ and $0\le \kappa\le 1$ are fixed parameters and the coupling $\lambda$ could depend on $N$ and $\varepsilon$. Again, we impose the mild restriction of $N^{\beta}\varepsilon^{\frac{d-1}{2}+2\kappa}\gg 1$ on the parameter space of $(N, \varepsilon)$ for now. Also, for the simplicity of notations, we will still use $H_N$ to denote the rescaled Hamiltonian $\widehat H_N$.  The introduction of $\beta, \kappa$ allows us to interpolate between different physical models.  We shall explore the physical meaning of these  parameters below.

\subsection{One-particle conservation laws}\label{section:one-particle conservation laws}  In this work, we adopt a statistical approach to the study of the quantum system and consider the corresponding density operator $\gamma_N^\varepsilon= \ket{\psi_N^\varepsilon}\!\!\bra{\psi_N^\varepsilon}$ acting on $L^2(\R^{3N})$ with the integral kernel
\begin{align}
    \gamma_N^\varepsilon(t, X_N; X_N') := \psi_N^\varepsilon(t, X_N)\,\conj{\psi_N^\varepsilon(t, X_N')}
\end{align}
where $X_N:= (x_1, \ldots, x_N)\in \R^{3N}$.
Then its dynamics is governed by the Liouville--von Neumann equation
\begin{align}\label{eq:Liouville_vonNeumann}
    i\varepsilon\, \bd_t \gamma_N^\varepsilon  = [H_N, \gamma_N^\varepsilon]
\end{align}
where  $[\sfA, \sfB]:=\sfA\sfB - \sfB\sfA$ denotes the usual operator commutator. With the obervation that $\tr_{L^2(\R^{3N})}(\gamma_N^\varepsilon)=1$, we define the $k$-particle marginal density operator $\gamma_{N:k}^\varepsilon$ to be the operator acting on $L^2(\R^{3k})$ with the Schwartz kernel
\begin{align*}
    \gamma_{N:k}^\varepsilon(t, X_k; X_k'):= \int_{\R^{3(N-k)}}\psi_N^\varepsilon(t, X_k, Y_{N-k})\conj{\psi_N^\varepsilon(t, X_k', Y_{N-k})}\dd Y_{N-k}
\end{align*}
and its corresponding $k$-particle density and probability density functions
\begin{align*}
    \varrho_{N:k}^\varepsilon(t, X_k) := \binom{N}{k}\gamma_{N:k}^\varepsilon(t, X_k, X_k)\quad \text{ and } \quad \rho_{N:k}^\varepsilon(t, X_k):= \gamma_{N:k}^\varepsilon(t, X_k, X_k)\ .
\end{align*}
In particular, the sequence of $k$-particle density matrices satisfy the BBGKY hierarchy equations
\begin{multline}
    i\varepsilon\,\bd_t\gamma^\varepsilon_{N:k} = \com{\sum^k_{\jj=1}-\tfrac{\varepsilon^2}{2}\lapl_{\jj}+\frac{1}{N}\sum_{\ii<\jj}^kV^{\varepsilon}_{N, \ii\jj}\, , \gamma^\varepsilon_{N:k}}
    +\frac{N-k}{N}\tr_{k+1}\(\com{\sum^k_{\jj=1} V^{\varepsilon}_{N, \jj k+1}\, ,\gamma_{N:k+1}^\varepsilon}\)
\end{multline}
for $k<N$ and $\gamma^\varepsilon_{N:N}:=\gamma^\varepsilon_N$ satisfies Equation \eqref{eq:Liouville_vonNeumann}. Notice for convenience, we have used the same notation for both the density operator and its integral kernel.

Since the dynamics of the system of identical bosons conserves mass, momentum, and energy, it is also natural to consider the one-particle formulation of the conservation laws.  Define the local one-particle mass, momentum, and energy densities
\begin{align*}
    \rho^\varepsilon_{N:1}(x):=&\, \diag(\gamma^\varepsilon_{N:1})(x)= \gamma_{N:1}^\varepsilon(x, x)\ ,\\
    \vect{J}_{N:1}^\varepsilon(x):=&\, \im\diag(\varepsilon\grad\gamma^\varepsilon_{N:1})(x)\\
    =&\, \textstyle\frac{1}{2i}\intd \(\varepsilon\grad_x \gamma^\varepsilon_{N:1}(x, x')-\varepsilon\grad_{x'} \gamma^\varepsilon_{N:1}(x, x')\)\delta(x-x')\dd x'\ ,\notag\\
    E^\varepsilon_N(x):=&\,  \textstyle \tfrac12\diag(\varepsilon\grad_{x}\cdot \varepsilon\grad_{x'}\gamma^\varepsilon_{N:1})(x) +\frac{N-1}{2N}\intd V^\varepsilon_N(x-y)\gamma_{N:2}^\varepsilon(x, y; x, y)\dd y\ .
\end{align*}
Furthermore, we also define the following quantities
\begin{align*}
    \boldsymbol{\sigma}^\varepsilon_{N:1}(x):=&\, \tfrac12 \diag\((\varepsilon\grad_{x}\otimes\varepsilon\grad_{x'})(\gamma^\varepsilon_{N:1}+(\gamma^\varepsilon_{N:1})^\top)\)\ ,\\
    P_{N}^\varepsilon(x):=&\,\textstyle -\tfrac{\varepsilon^2}{4}\lapl \rho^\varepsilon_{N:1}(x)+\frac{N-1}{2N}\intd V^\varepsilon_N(x-y)\gamma_{N:2}^\varepsilon(x, y; x, y)\dd y\ ,
\end{align*}
and
\begin{align*}
    \vect{l}^\varepsilon_{N:1} := \textstyle-\frac{N-1}{2N} \intd V^\varepsilon_N(x-y)\(\grad_x -\grad_y\)\gamma^\varepsilon_{N:2}(x, y; x, y)\dd y\ ,
\end{align*}
where  $\gamma^\top$ denotes the transpose operator. Then, by following the computation in \cite{chong2021global}, we see that $\rho^\varepsilon_{N:1}$ and $\vect{J}^\varepsilon_{N:1}$ satisfy the balance equations
\begin{align}
    &\bd_t \rho^\varepsilon_{N:1} + \grad\cdot \vect{J}^\varepsilon_{N:1}=0\ ,\\
    &\bd_t \vect{J}^\varepsilon_{N:1}+ \grad\cdot\(\boldsymbol{\sigma}^\varepsilon_{N:1}+P^\varepsilon_{N}\vect{I}\)+\vect{l}^\varepsilon_{N:1} = 0\ .
\end{align}
Formally, we see that $\vect{l}^\varepsilon_{N:1}\rightarrow 0$ as $N$ tends to infinity. Hence the limiting system is expected to converge to some Euler-type system. In fact, one of the goals in this work is to show rigorously that
\begin{align}
    \rho_{N:1}^\varepsilon \xrightarrow[\frac1N+\varepsilon \rightarrow 0]{} \rho\quad \text{ and } \quad \vect{J}_{N:1}^\varepsilon \xrightarrow[\frac1N+\varepsilon \rightarrow 0]{} \vect{J}
\end{align}
in some suitable topology, where $(\rho, \vect{J})$ solves the compressible Euler equation
\begin{equation}\label{eq:euler_equation_momentum_form}
    \begin{cases}
        \bd_t \rho +\grad\cdot \vect{J} = 0\\
        \bd_t \vect{J} +\grad\cdot \(\frac{\vect{J}\otimes \vect{J}}{\rho}\)+\grad P(\rho)=0
    \end{cases}\ ,
\end{equation}
for some pressure $P(\rho)=\frac12 c\rho^2$ where $c$ depends on the limiting regime. Moreover, this motivates us to define the microscopic pressure of the Bose gas by $P_N^\varepsilon$.

 \subsection{The compressible Euler equations and eikonal system}
 In semiclassical analysis, Madelung's quantum hydrodynamic picture of wave mechanics  \cite{madelung1927quantum} (similar to the formulation in the previous section) establishes a connection between linear Schr\"{o}dinger-type wave equations with classical fluid equations. Based on this formulation, there have been many developments on the studies of the asymptotic behavior of the wave function as the Planck's constant tends to zero. In fact, this approach has been extended to nonlinear Schr\"odinger-type wave equations. The compressible Euler equations~\eqref{eq:euler_equation_momentum_form} is one such a system of equations that is obtained as the leading order asymptotic behavior of the cubic nonlinear Schr\"odinger equation in the semiclassical expansion.  For a survey related to the rigorous study of the semiclassical limit of nonlinear quantum hydrodynamic equations, see \cite{carles2021semi,zhang2008wigner} and the references within. According to different physical regimes, the limiting behavior of the wave function are described by different classical field equations.

The existence of classical solutions to System~\eqref{eq:euler_equation_momentum_form} can in general only hold locally in time because of physical phenomena such as the focusing and breaking of waves, and the development of shock waves (see, e.g., \cite{dafermos2016hyperbolic}).  The local-in-time existence of smooth solutions has been established by the standard hyperbolic PDE theory, see \cite{majda1984compressible, makino1986local,makino1986compressible}. More specifically,
we write the compressible Euler equation~\eqref{eq:euler_equation_momentum_form}
in its velocity form
\begin{equation}\label{eq:euler_equation_velocity_form}
    \begin{cases}
        \partial _{t}\rho +\nabla \cdot \( \rho\, \bu\) =0 \\
        \partial _{t}\bu+(\bu\cdot \nabla )\bu+c\nabla_{x}\rho =0
    \end{cases}\ ,
\end{equation}
where $\rho(t,x)$ is the density function and $\bu(t,x)$ is the velocity function. Then
System~\eqref{eq:euler_equation_velocity_form} has a unique solution $(\rho, \bu)$ for any time $T_{0}<T_\ast<\infty$ such that
\begin{equation}\label{equ:euler-poisson equation, regularity condition}
	\left\{
	\begin{aligned}
        &\rho\in C([0,T_{0}];H^{s}(\mathbb{R}^{3})),\quad \bu\in C([0,T_{0}];H^{s-1}(\mathbb{R}^{3}; \R^3))\\
        &\rho(t,x)\geq 0,\quad \textstyle\intd \rho(t,x)\dd x=1
	\end{aligned}
	\right. \ ,
\end{equation}
provided $s$ is sufficiently large. Here $T_\ast$ denotes the blow-up time of the solution $(\rho, \bu)$, which depends on the initial data $(\rho^\init, \bu^\init)$.

Another important limiting description is given by the following coupled transport and Hamilton--Jacobi equations of geometric optics
 \begin{equation}\label{eq:geometric_optics}
    \left\{
    \begin{aligned}
        &\pa_{t}a+\nabla \varphi_{\mathrm{eik}}\cdot \nabla a+\tfrac{1}{2}a\, \Delta \varphi_{\mathrm{eik}}=-i c|a|^{2}a \\
        &\pa_{t}\varphi_{\mathrm{eik}}+\tfrac{1}{2}|\nabla \varphi_{\mathrm{eik}}|^{2}=0
    \end{aligned}
    \right.\ ,
 \end{equation}
where $a(t,x)$ is the (complex) amplitude function and $\varphi_{\mathrm{eik}}(t,x)$ is the phase function.
The latter equation is also referred to as the time-dependent eikonal equation and the local well-posedness of System~\eqref{eq:geometric_optics} can follow from the energy method, such as \cite{majda1984compressible}. The system gives a classical corpuscular interpretaton of wave propagation.  We shall refer to this hyperbolic system~\eqref{eq:geometric_optics} as the eikonal system.

These two systems are important in many areas of pure and applied mathematics, and usually derived based on Newtonian mechanical principle. Nevertheless, they do have a fundamental origin from quantum many-body dynamics.
A key goal of statistical mechanics is to understand their emergence from
 microscopic theories in appropriate scaling regimes.

\subsection{Physical models and motivations} The model
\begin{align}\label{def:ESY_Hamiltonian}
    H_N=\sum^N_{j=1}-\frac{1}{2}\lapl_{x_j}+ \frac{1}{N}\sum^N_{i<j}  N^{3\beta} v \(N^\beta(x_i- x_j)\)
\end{align}
with the parameter $0\le \beta\le 1$ was first introduced in \cite{erdos2007derivation} to interpolated between different physical models of Bose gases. Notice the Hamiltonian \eqref{def:ESY_Hamiltonian} corresponds to the space-time rescaling of Hamiltonian \eqref{def:N-body_Hamiltonian} with $\varepsilon =1, \lambda = N^{\beta-1}$, and $L = N^\beta$. One can immediately see that Model \eqref{def:ESY_Hamiltonian} corresponds to three different physical situations: (i) $\beta=0$ corresponds to the mean-field interaction potential case, (ii) $0<\beta<1$ is interpreted as the weakly interacting Bose gas case, and (iii) $\beta=1$ corresponds to the Gross--Pitaevskii scaling regime. Let us give a brief overview of the different existing models/scaling regimes and compare them with the above proposed Model \eqref{eq:dimensionless_N-body}. The discussion will also allow us to give meaning to the case $\beta>1$ which was not covered by the original models of \cite{erdos2007derivation}.

The largest lengthscale of the problem is the size of the system, which is usually comparable with the variational scale of the particle density $\varrho(x)$ of the system, is set to be of order one. Consequently, this means that the size of $\varrho$ is of order $N$ since $N = \int \varrho$ and the average interparticle distance is given by $N^{-\frac13}$. Moreover, the interaction potential $U_N(x):=N^{3\beta-1}v(N^\beta x)$ also defines two lengthscales: the range of the interaction potential, denoted by $R_U$, and the ($s$-wave) scattering length associated with the interaction potential, denoted by $\asc_U$. The significant of the scattering length is that it determines the effective lengthscale for which one could observe two particle correlations. Another useful gas parameter is given by $\varrho\, R_U^3$ which measures the number of particles within the effective interaction range of a particle. When $\varrho\,R_U^3\ll 1$, we call this the dilute gas regime.

To understand the physical motivation for the model, we need to first review the notion of the scattering length. Suppose $U\ge 0$ is a radially symmetric function whose support is contained in $\{x \in \R^3\mid \n{x}\le R_U\}$. Consider the positive radial solution (called the $s$-wave)to the zero-energy scattering problem
\begin{align}\label{eq:zero-energy_scattering}
    \(-\lapl + U\) f = 0
\end{align}
satisfying the boundary condition $f(x)\rightarrow 1$ as $\n{x}\rightarrow \infty$. See \cite[Appendix C]{lieb2005mathematics} for a proof of existence of a unique positive radial solution to Problem \eqref{eq:zero-energy_scattering}.   Then the scattering length $\asc_U$ is defined by
\begin{align}\label{def:scattering_length}
    \asc_U   := \frac{1}{4\pi}\intd U(x)f(x)\dd x\ .
\end{align}
Moreover, it is easy to see that for $\n{x}>R_U$, we have that
\begin{align}
    f(x)= 1- \frac{\asc_U}{\n{x}}
\end{align}
which also means that if
\begin{align}
    w(x):= 1-f(x)
\end{align}
then we have that
\begin{align}
    \lim_{\n{x}\rightarrow \infty} \n{x}w(x) = \asc_U\ .
\end{align}
Furthermore, we have the obvious upper bound $\asc_U\le R_U$, which follows from the positivity of $f(x)$, and the upper bound (see \cite{spruch1959upper} or Remark 4 in \cite[Appendix C]{lieb2005mathematics})
\begin{align}\label{est:upper_bound_on_scattering_length}
    \asc_U \le \frac{1}{4\pi}\intd U(x)\dd x\ .
\end{align}
Notice, for each potential $U$, there are two important lengthscales to keep in mind: the interaction range of $U$ and the scattering length of $U$.

In the case $\beta =0$, the interaction range is the size of the system and the scattering length associated to $U_N$ is of the order $N^{-1}$, i.e. $\asc_N\ll R_{U_N}\sim 1$ where $\asc_N$ is the scattering length of $U_N$. In this case, it has been proved that the effective two-body scattering process is given by $U_N\ast\varrho$, i.e. of mean-field type which also corresponds to the derivation of the Hartree equation. In the case $0<\beta<1$, the range of the interaction is given by $R_{U_N}= \mathcal{O}(N^{-\beta})$ which is much smaller than the variational scale of the density $\varrho$ and the scattering length is again of order $N^{-1}$, that is, $\asc_N\ll R_{U_N}\ll 1$. Notice $\beta<\frac13$ corresponds to a dense gas model and $\beta>\frac13$ corresponds to a dilute gas model. In either cases, the effective two-body scattering process is given by its Born approximation and the corresponding effective potential is given by $(\int U_N)\varrho$. Lastly, the case $\beta = 1$ corresponds to the scenario where $\asc_{N}\simeq R_{U_N}=\mathcal{O}(N^{-1})$, which means the effective two-body scattering process is exactly the full two-body scattering process. Here, the effective potential is given by $4\pi\asc_0 \varrho/N$, where $\asc_0$ is the scattering length of $v$. Moreover, computing the gas parameter,we see that
\begin{align}
    \varrho\,\asc_N^3 = \mathcal{O}\(N^{-2}\)
\end{align}
which is an extremely diluted regime. For this reason,  the GP regime is also known as an ultradilute regime in the literature.

\subsubsection{Hard-core limit regime} Another interesting model we consider is the hard-sphere model, which corresponds to the Hamiltonian \eqref{def:N-body_Hamiltonian} with the hard-core interaction potential
\begin{align*}
    v_{\rm hc}(x) =
\begin{cases}
    \infty &  \text{if } \n{x} < R_0 \\
    0 & \text{if } \n{x} \geq R_0
\end{cases}
\ .
\end{align*}
More precisely, for simplicity, we set the interaction range to $R_0 = N^{-1}$. It follows that $\varrho \asc_0^3 \simeq N R_0^3 \ll 1$, since $\asc_0 = R_0$ for the hard-core potential. Hence, we are in the ultradilute gas regime; in fact, this is simply the Gross--Pitaevskii regime but with a potential that does not belong to $L^1$ in the large $N$ limit. However, for technical reasons, we do not work directly with the hard-core potential. Instead, we consider a limiting family of Hamiltonians. This perspective, that the dynamics of the hard-core model should be well-approximated by the dynamics of a smooth family of potentials, was noted in \cite[Remark 1.0.1]{GSRT13} for classical systems, and our work addresses the quantum analog of this claim. Mathematically, it is an interesting and challenging problem to rigorously understand in what sense the hard-core model can be approximated by models with smooth interactions for large particle systems.

More specifically, we consider the interaction potential $\lambda R_0^{-2} v(\n{x}/R_0)$, where $v: \mathbb{R}^3 \rightarrow [0, \infty)$ is a positive, radially symmetric, continuous function with support on $\{x \mid \n{x} \leq 1\}$, and $\lambda = \lambda(N) \nearrow \infty$ as $N \rightarrow \infty$. Rewriting the Hamiltonian in its dimensionless form with the characteristic length of the interaction range, we have
\begin{align*}
        H_N = \sum_{j=1}^N -\frac{1}{2} \Delta_{x_j} + \frac{1}{N} \sum_{i<j}^N \lambda(N) N^3 v \left( N (x_i - x_j) \right)\ .
\end{align*}
We call this the hard-core limit regime. In practice, we shall consider $\lambda(N) = (\ln N)^\upalpha$ for some $0 < \upalpha < 1$.

\subsubsection{The hyper-ultradilute limit regime: $\beta>1$ Case} As indicated in the above discussion, the case $\beta>1$ can be viewed as a rescaling of Model \eqref{def:N-body_Hamiltonian} with $\lambda = N^{\beta-1}$ and $L_0=N^\beta$. Computing the gas parameter of the rescaled model, we see that $\varrho\,R_{U_N}^3= \mathcal{O}(N^{1-3\beta})\ll \mathcal{O}(N^{-2})$; for this reason,  we refer to the case $\beta>1$ as the hyper-ultradilute regime.

Moreover, unlike the hard-core limit case, here we rescale the spatial variable by the characteristic length $L_0=N^\beta$, where  $\widetilde\lambda=1$ and $\varepsilon=1$, as defined in Equations \eqref{def:characteristic_length}. In particular, the potential belongs in $L^1$ uniformly in $N$, which is not the case in the hard-core limit regime.

An important feature of this scaling is that the corresponding scattering length is comparable to the range of the interaction potential, which is much smaller than $N^{-1}$, that is, $R_{U_N}\sim \asc_N=\mathcal{O}(N^{-\beta}) \ll N^{-1}$. As a result, the corresponding effective potential includes a smallness factor $N^{1-\beta}$, which arises due to the small scattering length. In lay terms, the hyper-ultradiluteness of the gas introduces this smallness factor when analyzing the corresponding scattering process.

\subsubsection{Beyond Gross--Pitaevskii regime} In the case $0<\varepsilon<1$ with $\varepsilon\rightarrow 0$ as $N\rightarrow \infty$ and $L_0= N\varepsilon^{2}$ (similarly for $\kappa \in [0, 1)$) in Model~\eqref{def:N-body_Hamiltonian}, we see that the gas parameter yields
\begin{align*}
    \varrho\,\asc_U^3 =\mathcal{O}(N^{-2}\varepsilon^{-6})\gg N^{-2}\ ,
\end{align*}
which corresponds to a model less dilute than the Gross--Pitaevskii regime. In the studies of the corresponding time-independent problems, these models are the beyond Gross--Pitaevskii regimes (cf. \cite{adhikari2021bose}). To studied the dynamics, it is natural to consider a semiclassical limit problem. But unlike the time-independent problem, we are not able to consider very large $\varepsilon$. The techniques used in this work does not allow for a polynomial dependence of $\varepsilon$ on $N^{-1}$.  Hence, it is a mathematical interesting and physically important open problem to consider large values of $\varepsilon$. However, this will not be a focus of our paper.

\bigskip

In this paper, we focus on
\begin{itemize}
\item the hard-core (HC) limit regime:
\begin{align*}
    \text{$\beta=1$, $\kappa=1$, and $\lambda = (\ln N)^\upalpha$ for $\upalpha \in (0, 1)$\ .}
\end{align*}
\item the Gross--Pitaevskii (GP) limit regime:
\begin{align*}
    \text{$\beta=1$, $\kappa=1$, and $\lambda =1$\ .}
\end{align*}
\item the beyond Gross--Pitaevskii (BGP) limit regime:
\begin{align*}
    \text{$\beta=1$, $\kappa\in(0,1)$, $\la=(\ln N)^{\upalpha}$ for
$\upalpha\in[0,1)$\ .}
\end{align*}

\item the semiclassical Gross--Pitaevskii (SGP) limit regime:
\begin{align*}
    \text{$\beta=1$, $\kappa = 0$, $\la=(\ln N)^{\upalpha}$ for
$\upalpha\in[0,1)$\ .}
\end{align*}
\item the hyper-ultradilute (HD) limit regime:
\begin{align*}
    \text{$\beta>1$, $\kappa \in [0, 1]$, $\la=(\ln N)^{\upalpha}$ for
$\upalpha\in[0,1)$\ .}
\end{align*}
\end{itemize}
Note that the GP limit regime is, in fact, a beyond GP-type limit regime as discussed above; however, since we employed a different method in our analysis, it was more convenient to single out this case and assign it a different name.

\subsection{Bosonic Fock space} We adopt a grand canonical (second quantization) perspective when studying the dynamics of the many-body boson system as opposed to the canonical picture, fixed number of particles, discussed above. The purpose of this section is to introduce the necessary mathematical definitions needed to state the main results of our work; more details on the method of second quantization can be found in Section~\ref{section:bogoliubov_approximation}.

Let $\h = L^2(\R^3, \CC)$ denote the one-particle state space. We define the bosonic Fock space over $\h$, denoted by $\cF_s$, to be the closure of the algebraic direct sum pre-Hilbert space
\begin{align}
    \cF_s^{ \mathrm{alg}}:= \CC\oplus \bigoplus^\infty_{n=1} \bigotimes^n_{s}\h
\end{align}
with respect to the norm $\Nrm{\cdot}{}=\Nrm{\cdot}{\cF}$ induced by the Fock inner product
\begin{align}
    \inprod{\Psi}{\Phi} :=  \conj{\psi_0}\, \phi^{\init} + \sum^\infty_{n=1}\int_{\R^{3n}}\conj{\psi_{n}(x_1,\ldots, x_{n})}\,\phi_n(x_1,\ldots, x_{n})\dd x_1\cdots \d x_n
\end{align}
defined for all vectors $\Psi =(\psi_0, \psi_1, \ldots), \Phi = (\phi^{\init}, \phi_1, \ldots )\in \cF^{\mathrm{alg}}_{s}$. Here $\otimes_s$ is the symmetric tensor product of Hilbert spaces and $\bigotimes^n_{s}\h$ denotes the subspace of $\h^{\otimes n}= L^2(\R^{3n}, \CC)$ generated by simple symmetric tensors. Elements of $\cF_s$ are called (pure) state vectors. An important state vector in $\cF_s$ is the vacuum state
\begin{align}
    \Omega := (1, 0, 0, \ldots)
\end{align}
which describe a state with no particle. More generally, we define the number operator $\cN$ acting on $\cF_s$ by
\begin{align}
    \cN\Psi = (n\, \psi_n)_{n\in \N}\ .
\end{align}
Then the expected number of particles in state $\Psi$ is given by
\begin{align*}
   \inprod{\Psi}{\cN\, \Psi} = \sum^\infty_{n=1} n \Nrm{\psi_n}{L^2(\R^{3n})}^2\ .
\end{align*}

The Fock Hamiltonian $\cH_{N}$, dependent on $\varepsilon$ and $N$, is defined to be the operator acting on $\cF_s$ whose action on the $n$th sector of $\cF_s$ is given by
\begin{align}
    (\cH_N\Psi)_n = \(\sum^n_{j=1}-\frac{\varepsilon^2}{2}\lapl_{x_j}+ \frac{1}{N}\sum^n_{i<j}  V_N^\varepsilon \(x_i- x_j\)\)\psi_n\ ,
\end{align}
with the obvious modification when $n=0, 1$. Here, the interaction potential
\begin{align*}
V_{N}^{\ve}(x)=\lambda(N) (N^{\beta}\varepsilon^{2\kappa})^3v(N^{\beta}\varepsilon^{2\kappa}\, x),
 \end{align*}
 which is given by the scaling analysis in \eqref{def:rescaled_interaction_potential}. Then, we consider the Cauchy problem in Fock space described by the equation
\begin{align}\label{def:Fock_space_cauchy_problem}
    i\varepsilon\,\dpt \Psi = \cH_N \Psi \quad \text{ with } \quad \Psi(0)=\Psi^\init\ ,
\end{align}
where $\Nrm{\Psi^\init}{}=1$ and $\inprod{\Psi^\init}{\cN\, \Psi^\init}=\mathcal{O}(N)$. In fact, we consider a class of initial states called the Bogoliubov states, which we shall postponed its definition till Section~\ref{section:bogoliubov_approximation}.

As in the canonical picture, we consider the statistical formulation. Let $\Psi_N^\varepsilon = (\psi_{N, k}^\varepsilon)_{k\in \N}$ denotes the solution to Problem~\eqref{def:Fock_space_cauchy_problem} and define the corresponding density operator $\Gamma_N^\varepsilon :=\ket{\Psi_N^\varepsilon}\!\!\bra{\Psi_N^\varepsilon}$ acting on $\cF_s$. We define the corresponding one-particle reduced density operator acting on $\h$ by the kernel
\begin{align}\label{def:one-particle_reduced_density}
    \Gamma_{N:1}^\varepsilon(x, x')
    := \frac{1}{\inprod{\Psi_N^\varepsilon}{\cN\,\Psi_N^\varepsilon}}\sum^\infty_{n=1} n \int_{\R^{3(n-1)}}\psi^\varepsilon_{N, n}(x , X_{n-1})\conj{\psi^\varepsilon_{N, n}(x', X_{n-1})}\dd X_{n-1}\ .
\end{align}
and the two-particle reduced density operator
\begin{align}\label{def:two-particle_reduced_density}
    &\Gamma_{N:2}^\varepsilon(x_1, x_2, x_1', x_2')\\
    :=& \frac{1}{\inprod{\Psi_N^\varepsilon}{\cN(\cN-1)\,\Psi_N^\varepsilon}}\sum^\infty_{n=2} \binom{n}{2}  \int_{\R^{3(n-2)}}\psi^\varepsilon_{N, n}(x_1, x_2, X_{n-1})\conj{\psi^\varepsilon_{N, n}(x_1', x_2', X_{n-1})}\dd X_{n-1}\notag.
\end{align}
As in Section~\ref{section:one-particle conservation laws}, from the definitions~\eqref{def:one-particle_reduced_density}--\eqref{def:two-particle_reduced_density} we define the local one-particle mass, momentum, and energy densities as follows
\begin{align*}
    \rho^\varepsilon_{N:1}(x):=&\, \diag(\Gamma^\varepsilon_{N:1})(x)\ ,\\
    \vect{J}_{N:1}^\varepsilon(x):=&\, \im\diag(\varepsilon\grad\Gamma^\varepsilon_{N:1})(x)\ ,\\
    E^\varepsilon_{N}(x):=&\, E^\varepsilon_{N, \mathrm{kin.}}(x)+E^\varepsilon_{N, \mathrm{int.}}(x)\ , \\
    E^\varepsilon_{N, \mathrm{kin.}}(x):=&\,  \textstyle \tfrac12\diag(\varepsilon\grad_{x}\cdot \varepsilon\grad_{x'}\Gamma^\varepsilon_{N:1})(x)\ ,\\
    E^\varepsilon_{N, \mathrm{int.}}(x):=&\,  \textstyle\frac{\inprod{\Psi_N^\varepsilon}{\cN(\cN-1)\,\Psi_N^\varepsilon}}{2N\inprod{\Psi_N^\varepsilon}{\cN\,\Psi_N^\varepsilon}}\intd V^\varepsilon_N(x-y)\Gamma_{N:2}^\varepsilon(x, y; x, y)\dd y\ .
\end{align*}

\subsection{Literature overview}

\subsubsection{Effective description of interacting Bose gas}
The first rigorous estimate of the ground state energy of a system of interacting dilute (hard-core) Bose gas in the thermodynamic limit was provided by Dyson in \cite{dyson1957ground}. Dyson proved the correct leading upper bound, $4\pi\asc_0\varrho$, which rigorously made appear the scattering length, but his lower bound was off by a factor of about 14. The problem was eventually resolved by Lieb and Yngvason in \cite{lieb1998ground}, and, more recently, the next-order correction was obtained by Fournais and Solovej in \cite{fournais2020energy,fournais2023energy}. The work of Lieb and Yngvason inspired numerous subsequent studies, including the rigorous proof that there is $100\%$ Bose--Einstein condensate (BEC) in the ground state in the GP limit, and showed that the Gross--Pitaevskii theory correctly describes the ground state properties of the Hamiltonian~\eqref{def:ESY_Hamiltonian} with $\beta=1$ (see \cite{lieb2002proof,lieb2000bosons}). However, to prove BEC in the thermodynamic limit remains a major open problem. In recent years, a substantial amount of effort has been made to go beyond the Gross--Pitaevskii scaling in hopes of reaching the ``thermodynamic length scales'' (see \cite{adhikari2021bose,brennecke2022excitation,Fournais2020LengthSF}).

In the time-dependent setting, the validity of the Gross--Pitaevskii theory was established in a series of fundamental papers by Erd{\"o}s, Schlein, and Yau in \cite{erdos2006derivation, erdos2007derivation, erdos2009rigorous, erdos2010derivation} where they rigorously derived the 3D time-dependent Gross--Pitaevskii equation from the dynamics of a repulsive quantum many-body system in the GP limit, which motivated a large amount of works on the studies of the quantum BBGKY and Gross--Pitaevskii hierarchy such as \cite{CHPS15,CP11,CH16correlation,CH16on,CH19,HS16,klainerman2008uniqueness,KSS11,Soh15}. In parallel, with a slightly different and more quantitative approach, Pickl in \cite{pickl2015derivation} also provided a derivation of the Gross--Pitaevskii equation.

Subsequently, a quantitative approach to the derivation of the Gross--Pitaevskii equation in the canonical and grand canonical formalism by studying the quantum fluctuations about the effective dynamics was given in \cite{ benedikter2015quantitative,boccato2017quantum, caraci2024quantum}, which were in turn based on the earlier works \cite{ginibre1979classical,grillakis2010second,  hepp1974classical,lewin2015bogoliubov,rodnianski2009quantum}. The main idea behind the approach is to study the quantum fluctuations about the effective dynamics via Bogoliubov theory, which has been proven quite powerful in obtaining strong quantitative estimates for the error between the many-body dynamics and the effective mean-field dynamics in the subcritical regime, $0\le \beta<1$, (see, e.g. \cite{brennecke2019fluctuations, brennecke2019GPdynamics,chen2012second,chong2022dynamical,grillakis2013pair, kuz2017exact,lewin2015fluctuations, nam2017bogoliubov, nam2017note}). In fact, in certain regimes, one could construction approximation which is optimal and to arbitrary precision in the decay of $N$ (see \cite{bossmann2022beyond,chen2011rate,erdos2009quantum}). However, unlike the time-independent counterpart, the hard-core potential was not considered in the above works.

\subsubsection{Semiclassical limit of the Gross--Pitaevskii equation} Long before being rigorously justified from first principles, the Gross--Pitaevskii equation (or nonlinear Schrödinger (NLS) equation) was already a widely accepted phenomenological model for BEC and nonlinear optical phenomena. Hence, to understand the behavior of its solutions, it was natural to analyze high-frequency regimes using semiclassical techniques such as the WKB (Wentzel--Kramers--Brillouin) approximation. These approaches allow one to neglect the wave-like behaviors of the solution and replace them with a more classical mechanical approach (see, e.g., \cite{lifshitz1980statistical}).

The rigorous justification that the leading order behavior of the WKB expansion of the solutions to the semiclassical cubic NLS equation in short time is given by the compressible Euler system can be traced back to the works of Gérard in \cite{gerard1993remarques}, assuming analytic initial data, and Grenier in \cite{Gre98}, with just Sobolev initial data. In the 1D defocusing cubic NLS case, Jin, Levermore, and McLaughlin in \cite{JLM99} established the semiclassical limit of the NLS equation to compressible Euler equations for all time by exploiting the complete integrability structure of the 1D cubic NLS equation. See also \cite{carles2021semi} for the case of more general nonlinearities.

Taking a different approach through the method of modulated energy, Lin and Zhang investigated the semiclassical limit of the Gross--Pitaevskii equation in 2D exterior domains and obtained the compressible Euler equations as the limiting system. Notably, the concept of using modulated energy, while assuming and employing the regularity of the limiting solution, is similar to the relative entropy method introduced by Yau in \cite{yau1991relative} for establishing the stability of hydrodynamic limits. It also resembles Brenier's modulated entropy approach in \cite{brenier2000convergence} for deriving kinetic-to-fluid limits, which was itself based on Lions' notion of dissipative solutions in fluid mechanics, as outlined in \cite{lions1996mathematical}. Currently, the method of modulated energy has been fully developed and successfully employed in many areas of mathematics, notably in mean-field limit problems, as seen in the works of Serfaty \cite{Ser17,serfaty2020mean}.

\subsubsection{Many-body dynamics to Euler-type equations}
Perhaps more importantly, we  aim to understand how nonlinear equations of classical physics emerge as descriptions of microscopic systems in specific asymptotic regimes.  Starting from classical particle systems, one strategy to derive macroscopic fluids equations involves first passing to the Boltzmann equations and then taking the hydrodynamic limit to obtain the fluid equations. See the monographs \cite{CIP94,GSRT13,SR09} and the references within for further details.

In quantum mechanics, the derivation of the compressible Euler equations from quantum fermionic many-body dynamics was first established by Nachtergaele and Yau \cite{nachtergaele2002derivation,nachtergaele2003derivation} under specific technical conditions. For an introduction and survey of the scaling limit problem, see \cite{yau1998scaling}. Deriving nonlinear equations of classical physics from quantum systems, particularly for singular interaction potentials, is far from trivial. For the Coulomb potential, Golse and Paul in \cite{golse2022mean} obtained the first result on deriving the pressureless Euler--Poisson equations from quantum many-body dynamics in the mean-field and semiclassical regimes. Their approach is based on Serfaty's inequality \cite{serfaty2020mean}, a powerful tool in the classical mean-field limit that also plays a key role in the quantum mean-field limit. For example, see \cite{Ros21}, which combined the mean-field, semiclassical, and quasi-neutral limits to derive the incompressible Euler equation with a Coulomb interaction potential.

For a singular $\delta$-type potential, the second and third authors, along with Chen and Wu,  derived the compressible Euler equations from the quantum bosonic many-body dynamics in \cite{chen2023derivation} using a different scheme that combines the BBGKY hierarchy method and the modulated energy method. This scheme was later adopted in \cite{CSZ23} to obtain the quantitative convergence rate to the pressureless Euler--Poisson equation under a specific restriction between $N$ and $\hbar$. In \cite{CSZ23mean}, the second and third authors, along with Chen, developed a new approach to establish the quantum functional inequality for the $\delta$-type potential, resulting in the derivation of the full Euler--Poisson equation with pressure for $\beta\in (0,1)$ in Model~\eqref{def:ESY_Hamiltonian}.

\subsection{Main results}

Let us start by stating the main result for the GP and HC limit regimes.
\begin{theorem}[GP/HC limit regime]\label{thm:main_result_GP-HC}
    Fix $\beta=\kappa=1$ and $\lambda(N)=(\ln N)^\upalpha$ for $\upalpha \in [0, 1)$ fixed. Let $v$ be a bounded nonnegative radially symmetric function supported in $\conj{B_{R_0}(0)}=\conj{\{\n{x}< R_0\}}$ and is strictly positive in $B_{R_0}(0)$. Suppose $\Psi_{N, t}$ is a solution to the Cauchy problem~\eqref{def:Fock_space_cauchy_problem} with initial data $\Psi^{\init}=e^{-\cA(\sqrt{N}\phi^{\init})} e^{-\cB(k_0)}\Xi_0$ being a generalized Bogoliubov state defined by Expression~\eqref{eq:Fock_space_cauchy_problem_data}  satisfying the conditions:
    \begin{enumerate}[$(1)$]
        \item  $\phi^\init $ is an $L^2$ normalized wave function and satisfies the uniform $H^{4}$ bound
        \begin{align*}
        \norm{\weight{\varepsilon\grad}^4\phi^{\init}}_{L^{2}_x}\leq C\ .
        \end{align*}
        \item $\Xi_0 \in \cF_s$ with $\Nrm{\Xi_0}{}=1$ satisfying the following bounds: there exists $D>0$, independent of $N$ and $\varepsilon$, such that we have the bounds
        \begin{align*}
            \inprod{\Xi_0}{\cN\,\Xi_0},\ \frac{1}{N}\inprod{\Xi_0}{\cN^2\,\Xi_0},\ \inprod{\Xi_0}{\cH_N\,\Xi_0} \le D\ .
        \end{align*}
    \end{enumerate}
    Let $(\rho^\init, \bu^\init)\in H^5(\R^3)\times H^4(\R^3;\R^3)$ where $\rho^\init\ge 0$ with $\intd\rho^\init \dd x = 1$. Assume the modulated energy at the initial time tends to zero as $\varepsilon \rightarrow 0$, that this,
    \begin{align*}
        \cM[\phi^\init, \bu^\init, \rho^\init]=\cM^\init:=&\, \frac12\intd \n{(i\varepsilon\grad + \bu^\init)\phi^\init}^2\dd x+\cM_{\mathrm{pot.}}[\phi^\init, \rho^\init]\xrightarrow[\varepsilon \rightarrow 0]{} 0\ ,
    \end{align*}
    where
    \begin{align*}
        \cM_{\mathrm{pot.}}[\phi^\init, \rho^\init]:=&\,  \frac12\intd (K\ast\rho^\init_0)\rho^\init_0\dd x
         +
        \left\{
        \begin{aligned}
             \textstyle  4\pi \mathfrak{a}_{0},\,\, &\text{GP regime} \\
             \textstyle  4\pi \mathfrak{c}_{0},\,\, &\text{HC regime}
        \end{aligned}
        \right\}\times
         \frac12\intd [(\rho^\init)^2-2\,\rho^\init\rho_0^\init]\dd x\ .
    \end{align*}
        Here, $\asc_0$ is the scattering length associated with $v$ and $\mathfrak{c}_0:=R_0$ is the electrostatic capacity of $v$. Suppose $(\rho, \bu)$,  defined on an order one interval $[0, T_0]$ and  satisfies
        \begin{equation*}
            \left\{
            \begin{aligned}
                &\rho\in C([0,T_{0}];H^{5}(\R^3)),\quad \bu\in C([0,T_{0}];H^{4}(\R^3; \R^3)) \\
                &\rho(t,x)\geq 0,\quad \textstyle\intd \rho(t,x)\dd x=1
            \end{aligned}
            \right. \ ,
        \end{equation*}
     is a solution to the compressible Euler system~\eqref{eq:euler_equation_velocity_form}
     \begin{equation*}
        \begin{cases}
            \partial _{t}\rho +\nabla \cdot \( \rho\, \bu\) =0 \\
            \partial _{t}\bu+(\bu\cdot \nabla )\bu+\left\{
            \begin{aligned}
                 \textstyle  4\pi \mathfrak{a}_{0},\,\, &\text{GP regime} \\
                 \textstyle  4\pi \mathfrak{c}_{0},\,\, &\text{HC regime}
            \end{aligned}
            \right\}\times \nabla_{x}\rho=0
        \end{cases}\ ,
    \end{equation*}
    with initial datum $(\rho^\init, \bu^\init)$.

     Under the restriction that\footnote{Here, the implicit constant in Condition~\eqref{equ:restriction,theorem,gp,hc} could depend on the fixed parameters such as the time $T_{0}$, the energy bounds, and the Sobolev norms of $(\rho^{\mathrm{in}}$, $\bu^{\mathrm{in}}$), but is independent of $(N,\ve)$. The 100 in Condition~\eqref{equ:restriction,theorem,gp,hc} is inessential and could be replaced by another large constant.}
    \begin{align}\label{equ:restriction,theorem,gp,hc}
        N\gtrsim \exp(\varepsilon^{-100/(1-\upalpha)})\ ,
    \end{align}
    then we have the following quantitative estimates:
    \begin{enumerate}[$(i)$]
        \item For the GP regime:
        $$\beta=\kappa=1,\ \la(N)=1\ ,$$
         we have the following quantitative convergence of mass, momentum, and energy densities
         \begin{equation}
            \left\{
            \begin{aligned}
                &\Nrm{\rho_{N:1}^{\varepsilon}-\rho}{L^{\infty}_t([0, T_{0}])L^{2}_x}^{2}\lesssim \cM^\init+\varepsilon^{2}+\tfrac{1}{\ln N}\ ,\\[3pt]
                &\Nrm{\vect{J}_{N:1}^{\varepsilon}-\rho\,\bu}{L^{\infty}_t([0, T_{0}])L^{1}_x}^{2}\lesssim \cM^\init+\varepsilon^{2}+\tfrac{1}{\ln N}\ ,\\[3pt]
                &\Nrm{E_{N,\mathrm{kin.}}^{\varepsilon}-\tfrac{1}{2}\rho\n{\bu}^2-4\pi \mathfrak{b}_0 \rho^{2}}{L^{\infty}_t([0,T_{0}])L^{1}_x}^{2}
                \lesssim \cM^\init+\varepsilon^{2}+\tfrac{1}{\ln N}\ ,\\
                &\Nrm{E_{N,\mathrm{int.}}^{\varepsilon}-4\pi (\mathfrak{a}_{0}-\mathfrak{b}_{0})\rho^{2}}{L_{t}^{\infty}([0, T_{0}])L_{x}^{1}}^{2}\lesssim
                \cM^\init+\varepsilon^{2}+\tfrac{1}{\ln N}\ ,
            \end{aligned}
            \right.
         \end{equation}
       where
          \begin{align}
               \mathfrak{b}_0= \frac{1}{4\pi}\intd \n{\grad f_0}^2\dd x = \asc_0 - \frac{1}{4\pi}\intd v (f_0)^2\dd x\ .
           \end{align}
       \item For the HC regime:
       $$\beta=\kappa=1,\ \la(N)=(\ln N)^{\upalpha} \text{ with } \upalpha\in(0,1)\ ,$$
        we have the following quantitative convergence of mass, momentum, and energy densities
         \begin{equation}\label{equ:semiclassical limit,convergence rate,k=1}
       \left\{
       \begin{aligned}
       &\Nrm{\rho_{N:1}^{\varepsilon}-\rho}{L^{\infty}_t([0, T_{0}])L^{2}_x}^{2}\lesssim \cM^\init+\varepsilon^{2}+\eta(\tfrac{1}{\lambda(N)})+\tfrac{1}{\ln N}\ ,\\[3pt]
       &\Nrm{\vect{J}_{N:1}^{\varepsilon}-\rho\,\bu}{L^{\infty}_t([0, T_{0}])L^{1}_x}^{2}\lesssim \cM^\init+\varepsilon^{2}+\eta(\tfrac{1}{\lambda(N)})+\tfrac{1}{\ln N}\ ,\\[3pt]
       &\Nrm{E_{N,\mathrm{kin.}}^{\varepsilon}-\tfrac{1}{2}\rho\n{\bu}^2-4\pi \mathfrak{c}_0 \rho^{2}}{L^{\infty}_t([0,T_{0}])L^{1}_x}^{2}
       \lesssim \cM^\init+\varepsilon^{2}+\eta(\tfrac{1}{\lambda(N)})+\tfrac{1}{\ln N}\ ,\\
       &\nv{E_{N,\mathrm{int.}}^{\varepsilon}}_{L_{t}^{\infty}([0, T_{0}])L_{x}^{1}}^{2}\lesssim
       \eta(\tfrac{1}{\lambda(N)})+\tfrac{1}{\ln N}\ ,
       \end{aligned}
       \right.
       \end{equation}
       where $\eta(\mu)\to 0$ as $\mu\to 0$, which is defined by Expression~\eqref{def:convergence_rate_of_scattering_length} is simply the convergence rate of $\asc_0^{\mu}$ to $\mathfrak{c}_0$.
       \end{enumerate}
\end{theorem}

\begin{remark}
    Recall the total energy density corresponding to Euler equations is given by $E=\tfrac12 \rho\n{\bu}^2+\rho e$ where $\rho e$ is the internal energy density and $e$ is the specific internal energy which is proportional to the density, i.e. $e = c\rho$. We can restate the convergence of the energy density functions as follow
    \begin{align*}
        &\lim_{(N,\ve)\rightarrow (\infty,0)}\Nrm{E_{N,\, \mathrm{kin.}}^\varepsilon-\rho\, (\tfrac12\!\n{\bu}^2+\uptheta e)}{L^\infty_t([0, T_0])L^1_x} = 0\ ,\\
         &\lim_{(N,\ve)\rightarrow (\infty,0)}\Nrm{E_{N,\, \mathrm{int.}}^\varepsilon- (1-\uptheta)\rho e}{L^\infty_t([0, T_0])L^1_x}= 0\ ,
    \end{align*}
    where $\uptheta \in (0, 1]$ given by
    \begin{align*}
        \uptheta :=
        \begin{cases}
            \intd \n{\grad f_0}^2\dd x/\asc_0  & \text{GP regime}\ , \\[3pt]
            1  & \text{HC regime}\ .
        \end{cases}
    \end{align*}
    Hence, Theorem $\ref{thm:main_result_GP-HC}$ demonstrates rigorously how the internal energies arise from the microscopic dynamics. In particular, in the hard-core limit, the internal energy depends only on the microscopic kinetic energy of the quantum system.
\end{remark}

\begin{remark}
    As in Lemma $\ref{lem:bogoliubov_states}$, the generalized Bogoliubov state\footnote{
    This choice of the localization parameter $l$ permits a wide variety of initial data. Specifically, we set $l=\ve^{4}$ to streamline our analysis. See more explanations in Section \ref{subsection:the_neumann_problem}.} $\Psi^{\init}=e^{-\cA(\sqrt{N}\phi^{\init})} e^{-\cB(k_0)}\Xi_0$ concentrates around $N$ particles with a deviation of the order $N^{1/2+}$. This state naturally emerges as an approximation to the ground state of a trapped Bose gas when the system's chemical potential is tuned to admit an expected $N$ particles. Furthermore, as shown in Lemma~\ref{lem:bogoliubov_states} and the result of \cite{lieb2000bosons}, if $\phi^\init$ is the normalized minimizer of the Gross--Pitaevskii functional, it follows that $\Psi^\init$ has, to leading order, the energy of the ground state.  Thus, $\Psi^{\init}$ models the state prepared in experiments by cooling the trapped Bose gas to very low temperatures called a ultracold Bose gas.
\end{remark}

\begin{remark}\label{remark:localization}
    We localize the correlation structure to a ball of size $\ell^3$ for two main reasons: (1) to simplify the notation, as otherwise we would encounter numerous estimates involving terms like $e^{c\varepsilon^{-2}}$ when estimating $\Nrm{\sh(k_{N, t}^\varepsilon)}{\mathrm{HS}}$, and (2) to improve the dependence between $N$ and $\varepsilon$. However, even without this localization, the result remains the same, but with a dependence of
    $N \gtrsim \exp(\kappa_1 \exp(\kappa_2\, \varepsilon^{-\upgamma}))$.

\end{remark}

\begin{remark}
As we only require uniform $H^{4}$ bounds on the initial wave function $\phi^{\mathrm{in}}$, it is not hard to find examples of $\phi^{\mathrm{in}}$ satisfying the initial condition that the modulated energy $\mathcal{M}^{\mathrm{in}} \to 0$. One typical example is the WKB-type initial datum
\begin{align*}
    \phi^{\mathrm{in}}(x) = \sqrt{\rho^{\mathrm{in}}(x)} e^{iS(x)/\varepsilon}, \quad \bu^{\mathrm{in}}(x) = \nabla S(x),
\end{align*}
for which $\mathcal{M}^{\mathrm{in}} = \mathcal{O}(\varepsilon^{2})$. The rate $\varepsilon^{2}$ should also be optimal in the sense that it matches the order of the error terms from the kinetic energy part of the modulated energy.
\end{remark}

\begin{remark}
The convergence rate of the scattering length $\asc_0^{\mu}$ to the capacity $\mathfrak{c}_0$, denoted $\eta(\mu)$, depends significantly on the profile near the zero value of the potential $v(x)$. See Section $\ref{subsection:semiclasscial_limit_of_the_scattering_length}$ for more details.
\end{remark}

\begin{remark}
The condition that the interaction potential $v(x)$ is strictly positive in the set $B_{R_0}(0)$ could be further relaxed. Here, we impose the positivity condition to simplify some calculations in the semiclassical limit analysis of the scattering length.
\end{remark}

In the BGP, SGP, and HD limit regimes, we shall employ WKB analysis and restrict ourselves to WKB states.
\begin{theorem}[BGP/SGP limit regimes]\label{thm:main_result_HD}
    Assume the same hypotheses on $v$ and $\Psi^\init$ as in Theorem~\ref{thm:main_result_GP-HC}. Moreover, assume $\phi^\init$ is a $L^2$ normalized WKB state, that is, $\phi^\init$ has the form
    \begin{align}
        \phi^\init = a^{\ve,\, \init} e^{i\varphi^\init_{\mathrm{eik}}/\varepsilon}
    \end{align}
    where $(a^{\ve, \init}, \varphi^\init_{\mathrm{eik}})$ satisfies the following conditions:
    \begin{enumerate}[$(1)$]
        \item there exists $C$, independent of $\varepsilon$, such that
        \begin{align*}
            &\varphi_{\mathrm{eik}}^{\mathrm{in}}\in C^{\infty}(\R^3; \R),\quad \pa_{x}^{\sigma}\varphi_{\mathrm{eik}}^{\mathrm{in}}\in L^{\infty}(\R^3) \quad \forall \sigma \in \N_0^3,\ |\sigma|:=\textstyle\sum^3_{\jj=1}\sigma_\jj \geq 2\ ,\\
            &\text{ and }\quad  \nv{a^{\ve,\init}}_{H^{4}_x}\leq C\ .
        \end{align*}
        In particular, we have that $\weight{\varepsilon\grad}^4\phi^\init \in L^2(\R^3)$, uniformly bounded in $\varepsilon$.
        \item      Let $a^\init \in \bigcap_{k\ge 0} H^{k}(\R^3)$ where $\intd|a^\init|^2 \dd x = 1$ and suppose $(a, \varphi_{\mathrm{eik}})$,  defined on an order one interval $[0, T_0]$ and satisfies
        \begin{equation*}
            \left\{
            \begin{aligned}
                &\varphi_{\mathrm{eik}}\in C([0,T_{0}];C^\infty(\R^3)),\quad a\in C([0,T_{0}];\textstyle\bigcap_{k\ge 0}H^{k}(\R^3))\ ,\\
                &\textstyle\intd \n{a(t,x)}^2\dd x=1\ .
            \end{aligned}
            \right.
        \end{equation*}
    \end{enumerate}
    Under the restriction that
    \begin{align}
    N\gtrsim \exp(\varepsilon^{-\frac{100}{1-\upalpha}})\ ,
    \end{align}
 for $s>\frac{3}{2}$, $p\in[1,2]$,  the following statements hold:

    \begin{enumerate}[$(i)$]
        \item For the BGP regime in the critical case:
        $$\beta=1,\ \kappa=\tfrac{1}{2},\ \la(N)=(\ln N)^{\upalpha}\text{ with }\upalpha\in[0,1)\ ,$$
        we have the following quantitative convergence of mass, momentum, and energy densities
         \begin{equation}\label{equ:semiclassical limit,convergence rate,k=1/2}
        \left\{
        \begin{aligned}
        &\Nrm{\rho_{N:1}^{\varepsilon}-|a|^{2}}{L^{\infty}_t([0, T_{0}])L^{p}_x}\lesssim \nv{a^{\varepsilon,\mathrm{in}}-a^{\mathrm{in}}}_{H^{s}_x}+\varepsilon+\eta(\varepsilon)+\tfrac{1}{\ln N}\ ,\\[3pt]
        &\Nrm{\vect{J}_{N:1}^{\varepsilon}-|a|^{2}\nabla \varphi_{\mathrm{eik}}}{L^{\infty}_t([0, T_{0}])L^{1}_x}\lesssim \nv{a^{\varepsilon,\mathrm{in}}-a^{\mathrm{in}}}_{H^{s}_x}+\varepsilon+\eta(\varepsilon)+\tfrac{1}{\ln N}\ ,\\[3pt]
        &\Nrm{E_{N,\mathrm{kin.}}^{\varepsilon}-\tfrac{1}{2}|a|^{2}|\nabla \varphi_{\mathrm{eik}}|^{2}}{L^{\infty}_t([0,T_{0}])L^{1}_x}
        \lesssim \Nrm{a^{\varepsilon,\mathrm{in}}-a^{\mathrm{in}}}{H^{s}_x}+\varepsilon+\eta(\varepsilon)+\tfrac{1}{\ln N}\ ,\\[3pt]
        &\Nrm{E_{N,\mathrm{int.}}^{\varepsilon}}{L^{\infty}([0, T_{0}] ;L_{x}^{1})}\lesssim
        \varepsilon+\tfrac{1}{\ln N}\ ,
        \end{aligned}
        \right.
        \end{equation}
        where $(a,\varphi_{\mathrm{eik}})$ satisfies the eikonal system with the capacity constant $\mathfrak{c}_{0}$
         \begin{equation}\label{equ:eikonal equation,theorem}
         \left\{
         \begin{aligned}
        &\pa_{t}a+\nabla \varphi_{\mathrm{eik}}\cdot \nabla a+\tfrac{1}{2}a\,\Delta \varphi_{\mathrm{eik}}=-i 4\pi \mathfrak{c}_{0}|a|^{2}a\\
        &\pa_{t}\varphi_{\mathrm{eik}}+\tfrac{1}{2}|\nabla \varphi_{\mathrm{eik}}|^{2}=0
         \end{aligned}
         \right.\ .
         \end{equation}
        \item For the SGP regime:
        $$\beta=1,\ \kappa=0,\ \la(N)=(\ln N)^{\upalpha} \text{ with } \upalpha\in[0,1)\ ,$$
        we have
         \begin{equation}\label{equ:semiclassical limit,convergence rate,k=0}
        \left\{
        \begin{aligned}
        &\Nrm{\rho_{N:1}^{\varepsilon}-|a|^{2}}{L^{\infty}_t([0, T_{0}])L^{p}_x}\lesssim \nv{a^{\varepsilon,\mathrm{in}}-a^{\mathrm{in}}}_{H^{s}}+\varepsilon+\tfrac{1}{\ln N}\ ,\\[3pt]
        &\Nrm{\vect{J}_{N:1}^{\varepsilon}-|a|^{2}\nabla \varphi_{\mathrm{eik}}}{L^{\infty}_t([0, T_{0}])L^{1}_x}\lesssim
        \nv{a^{\varepsilon,\mathrm{in}}-a^{\mathrm{in}}}_{H^{s}}+\varepsilon+\tfrac{1}{\ln N}\ ,\\[3pt]
        &\Nrm{E_{N,\mathrm{kin.}}^{\varepsilon}-\tfrac{1}{2}|a|^{2}|\nabla \varphi_{\mathrm{eik}}|^{2}}{L^{\infty}_t([0, T_{0}])L^{1}_x}
        \lesssim \nv{a^{\varepsilon,\mathrm{in}}-a^{\mathrm{in}}}_{H^{s}}+\varepsilon+\tfrac{1}{\ln N}\ ,\\[3pt]
        &\Nrm{E_{N,\mathrm{\mathrm{int.}}}^{\varepsilon}}{L^{\infty}_t([0, T_{0}])L^{1}_x}\lesssim
        \varepsilon+\tfrac{1}{\ln N}\ ,
        \end{aligned}
        \right.
        \end{equation}
        where $(a,\varphi_{\mathrm{eik}})$ satisfies the trivial eikonal system
         \begin{equation}\label{equ:eikonal equation}
         \left\{
         \begin{aligned}
        &\pa_{t}a+\nabla \varphi_{\mathrm{eik}}\cdot \nabla a+\tfrac{1}{2}a\Delta \varphi_{\mathrm{eik}}=0\\
        &\pa_{t}\varphi_{\mathrm{eik}}+\tfrac{1}{2}|\nabla \varphi_{\mathrm{eik}}|^{2}=0
         \end{aligned}
         \right.\ .
         \end{equation}
        \end{enumerate}
    \end{theorem}
\begin{remark}
For the HD regime:
        $$\beta>1,\ \kappa\in[0,1],\ \la(N)=(\ln N)^{\upalpha} \text{ with } \upalpha \in [0,1)\ ,$$
the result is similar to the SGP regime and the limit equation is also the trivial eikonal system. To avoid redundancy, we omit the detailed description.
\end{remark}

\subsection{Outline, novelties, and notations: a guide to the paper} The essential idea behind the proof of the main results is to demonstrate the convergence from the dynamics of the many-body system to the solution of the compressible Euler system by factoring through the intermediate dynamics given by the Gross--Pitaevskii equation. The key to the proof lies in the quantitative nature of the approach: we use quantitative techniques from \cite{benedikter2015quantitative, boccato2017quantum} to estimate the difference between the many-body dynamics and the Gross--Pitaevskii equation, and then we compare the dynamics of the Gross--Pitaevskii equation to the compressible Euler system via the method of modulated energy as seen in \cite{lin2006semiclassical}, which is also a quantitative method.

One challenge in proving the convergence of local conserved densities, specifically energy densities, is the need to work at the regularity level of the many-body system's energy. In this context, the quantitative convergence problem is not only sharp from the viewpoint of energy conservation but also critical for the GP/HC scaling regime, as the strong particle correlations in this regime necessitate a nontrivial correction to the effective energy approximation, unlike the subcritical cases where the correlation structure is subleading \cite{chen2023derivation, CSZ23mean}. To address these difficulties, we use the effective approximation operator with a correction term:
\begin{align}\label{equ:correction operator,intro}
\Gamma^\varepsilon_{\mathrm{cor},t}:=\ket{\phi^\varepsilon_{N, t}}\!\!\bra{\phi^\varepsilon_{N, t}}+\tfrac{1}{N}\sh(k^\varepsilon_{N, t})\,\conj{\sh(k^\varepsilon_{N, t})}\ ,
\end{align}
which indeed plays a key role in explaining the transition mechanism of energies in appropriate regimes.

\subsubsection{Outline} In Section~\ref{section:Scattering Function and its Semiclassical Analysis}, we review some results regarding the properties of the scattering function. The emphasis of the section is on the role the semiclassical parameter $\mu$ plays in the estimates for the scattering function and how it affects the scattering length. Section~\ref{section:bogoliubov_approximation} provides a review of the Fock space formalism and shows how the local one-particle densities of the many-body dynamics can be approximated by densities of the (modified) Gross--Pitaevskii dynamics through controlling the fluctuation dynamics. Section~\ref{section:Bounds on the Fluctuations} contains all the analysis of the fluctuation dynamics needed to prove the main results. Section~\ref{section:Semiclassical Limit of the Modified Gross--Pitaevskii Equation} examines the convergence of the (modified) Gross--Pitaevskii dynamics to the dynamics of the compressible Euler system and more; moreover, we also provide a more detailed analysis on certain aspects of the space-time scaling. Finally, Section~\ref{sect:proof_of_main_results} combines the results from the previous sections to present the proofs of the main results.

Let us summarize the sketch and the key idea of the proof.

\vspace{0.25cm}
\noindent\textbf{Step 1. The $W^{1,1}$-type quantitative approximation: quantum many-body dynamics vs. modified Gross--Pitaevskii dynamics}.

For the quantum mean-field limit to the nonlinear Schr\"{o}dinger equations, the convergence in the trace norm, $L^{2}$ norm, or even weak sense, is sufficient to describe the physical phenomenon of the Bose--Einstein condensate. However,
in the derivation of macroscopic equations, a
$W^{1,1}$-type quantitative approximation is required, as the quantum momentum and kinetic energy densities require at least one derivative of regularity to be well-defined.

Continuing a great deal of efforts as in \cite{benedikter2015quantitative, boccato2017quantum,erdos2006derivation, erdos2007derivation, erdos2009rigorous, erdos2010derivation} regarding the study of the Gross--Pitaevskii regime, we establish the following $W^{1,1}$-type quantitative approximation estimates
\begin{align}
\Nrm{\diag(\Gamma_{N:1, t}^\varepsilon)-\diag(\Gamma^\varepsilon_{\mathrm{cor},t})}{L^1_x}\leq&
 \lrs{\frac{1}{\ln N}}^{100},\\
\Nrm{\diag(\varepsilon \grad_{x}\Gamma_{N:1, t}^\varepsilon)-\diag(\varepsilon \grad_{x}\Gamma^\varepsilon_{\mathrm{cor},t})}{L^1_x}\leq&\lrs{\frac{1}{\ln N}}^{100},\\
\Nrm{\diag(\varepsilon \grad_{x}\cdot\varepsilon\grad_{x'}\Gamma_{N:1, t}^\varepsilon)-\diag(\varepsilon \grad_{x}\cdot\varepsilon\grad_{x'}\Gamma^\varepsilon_{\mathrm{cor},t})}{L^1_x}\lesssim& \lrs{\frac{1}{\ln N}}^{100}.
\end{align}
It should be noted that the correction term
$\tfrac{1}{N}\sh(k^\varepsilon_{N, t})\,\conj{\sh(k^\varepsilon_{N, t})}$ in the  operator $\Gamma^\varepsilon_{\mathrm{cor},t}$ defined by \eqref{equ:correction operator,intro} is
a smallness factor of the order of $\frac{1}{N}$ in the $L^{2}$ level, but is indeed an effective quantity in the $H^{1}$ level, which we have shown in Lemma~\ref{lem:kinetic_energy_density_approx}. Moreover, the correction term also plays a key role in the classification of the different scaling regimes.

In addition to the quantum mass, momentum, and kinetic energy densities, a significant challenge arises in handling the quantum interaction energy density. Unlike these quantities, which are determined by the one-particle operator $\Gamma_{N:1,t}^{\ve}$, the interaction energy density is defined in terms of the two-particle operator $\Gamma_{N:2,t}^{\ve}$. Identifying an effective approximation is particularly nontrivial, especially when a correction term is incorporated into the one-particle operator. Nevertheless, we establish the effective approximation and rigorously prove that
\begin{align}\label{equ:interaction energy approximation,intro}
    \Nrm{\diag\(\tr_2\(V^\varepsilon_{N, 12}\Gamma^\varepsilon_{N:2, t}\)\)-4\pi\mu \la \(\asc_0^{\mu}-\mathfrak{b}_{0}^{\mu}\)\(\rho^\varepsilon_{t}\)^2}{L^1_x} \lesssim \lrs{\frac{1}{\ln N}}^{100}\ ,
\end{align}
where the parameter $\mu=\frac{N^{1-\beta}\ve^{2(1-\kappa
)}}{\la}$, $\mathfrak{a}_{0}^{\mu}$ denotes the scattering length, and $\mathfrak{b}_{0}^{\mu}$ is a characteristic length defined in \eqref{equ:characteristic length}.

\vspace{0.25cm}
\noindent\textbf{Step 2. Semiclassical limit of the scattering length.}

For clarity of presentation, we focus on the \textbf{hard-core limit regime}, characterized by the parameters $\lambda \to \infty$ and $\mu = \frac{1}{\lambda} \to 0$. As a direct consequence of \eqref{equ:interaction energy approximation,intro}, it follows that the interaction energy satisfies
\begin{align}\label{equ:interaction energy,intro,vanish}
    \Nrm{\diag\(\tr_2\(V^\varepsilon_{N, 12}\Gamma^\varepsilon_{N:2, t}\)\)}{L^{1}_{x}}\sim \lrs{\frac{1}{\ln N}}^{100}+|\asc_{0}^{\mu}-\mathfrak{b}_{0}^{\mu}|.
\end{align}

At present, it remains unclear whether the interaction energy given in \eqref{equ:interaction energy,intro,vanish} vanishes in the limit $N \to \infty$ and $\mu \to 0$. This behavior depends on the semiclassical limit of the scattering length $\mathfrak{a}_{0}^{\mu}$ and the characteristic length $\mathfrak{b}_{0}^{\mu}$ as $\mu \to 0$.

Furthermore, the modified Gross--Pitaevskii equation takes the form
\begin{align}\label{equ:mgp,intro}
    i\varepsilon\,\bd_t\phi_{N}^{\ve}  = -\tfrac{\varepsilon^2}{2}\lapl \phi_{N}^{\ve} + (K_{N}^{\ve}*|\phi_{N}^{\ve}|^{2})\phi_{N}^{\ve},
\end{align}
where the nonlinear term satisfies the approximation
\[
(K_{N}^{\ve}*|\phi_{N}^{\ve}|^{2})\phi_{N}^{\ve}\simeq 4\pi \asc_{0}^{\mu}|\phi_{N}^{\ve}|^{2}\phi_{N}^{\ve},
\]
which also necessitates a semiclassical limit analysis of the scattering length.

To address this, we develop a quantitative method for analyzing the convergence of the scattering length to the capacity constant $\mathfrak{c}_{0}$ of the interaction potential $v(x)$, establishing that
\begin{align}
    \eta(\mu):=\mathfrak{c}_{0}-\mathfrak{a}_{0}^{\mu}=\mathcal{O}(\mu^{\frac{1}{n+2}}),
\end{align}
where the exponent $\frac{1}{n+2}$ depends on the profile of $v(x)$.

The proof utilizes techniques from elliptic equations, including the comparison principle and the construction of auxiliary functions. For completeness, we provide a detailed proof in Appendix~\ref{appendix:semiclassical_two-body_problem}.

\vspace{0.25cm}
\noindent \textbf{Step 3. Deriving the limit equations from the modified Gross--Pitaevskii dynamics.}

The final step involves analyzing the semiclassical limit problem and comparing the modified Gross–Pitaevskii equation with the limiting equation. Significant advancements have been made in the study of the semiclassical limit of nonlinear Schrödinger equations (NLS), including the modulated energy method introduced in \cite{lin2006semiclassical} and the WKB expansion method developed in \cite{carles2007wkb,carles2021semi}.

In this work, we address the modified Gross–Pitaevskii equation \eqref{equ:mgp,intro}, which differs from NLS. The convolution structure $(K_{N}^{\ve}*|\phi_{N}^{\ve}|^{2})\phi_{N}^{\ve}$, a nonlocal term, obstructs the direct application of these established methods. Consequently, additional modifications are required, including a detailed error term analysis and a quantitative investigation of the semiclassical limit of the scattering length. Nevertheless, it remains necessary to establish a $W^{1,1}$-type quantitative approximation, analogous to Step 1:
\begin{align}
    &\Nrm{|\phi_{N}^{\ve}|^{2}-\rho}{L^{\infty}_t([0,T_{0}])L^{2}_x}^{2}\lesssim
    \cM \lrc{\phi_{N}^{\ve},\rho,\bu}(0)+\ve^{2}+\eta(\tfrac{1}{\lambda})+\frac{1}{\ln N}\ ,\\
    &\Nrm{\vect{J}_{N}^{\ve}-\rho \bu}{L^{\infty}_t([0,T_{0}])L^{1}_x}^{2}\lesssim \cM \lrc{\phi_{N}^{\ve},\rho, \bu}(0)+\ve^{2}+\eta(\tfrac{1}{\lambda})+\frac{1}{\ln N}\ ,\\
    &\Nrm{|i\ve\nabla \phi_{N}^{\ve}|^{2}-\rho\n{\bu}^{2}}{L^{\infty}([0,T_{0}])L^{1}_x}^{2}\lesssim
    \cM \lrc{\phi_{N}^{\ve},\rho,\bu}(0)+\ve^{2}+\eta(\tfrac{1}{\lambda})+\frac{1}{\ln N}\ .
\end{align}

By incorporating the $W^{1,1}$-type estimates established in Steps 1–3, along with the explicit bound on the correction term $\tfrac{1}{N}\sh(k^\varepsilon_{N, t})\,\conj{\sh(k^\varepsilon_{N, t})}$, we conclude the proof of the main result in the hard-core limit regime. Furthermore, the unified scheme outlined above is applicable to other scaling regimes as well.

\subsubsection{Novelties}
 Let us summarize the key contributions of our work:
\begin{enumerate}[(1)]
\item One of the main contributions is the derivation of the compressible Euler equation (with the pressure term  $P=2\pi \mathfrak{c}_0\rho^2$) from the dynamics of quantum Bose gases in the hard-core limit regime, up to the blow-up time of the compressible Euler equations, without factoring through kinetic-type descriptions such as the Boltzmann equation.  In fact, we provide a complete derivation of all the fluid conserved quantities from the microscopic dynamics.
 \item  In the hard-core limit regime, there are two novel findings: first, we rigorously prove that the internal specific energy in the macroscopic fluid description of a hard-core gas is entirely determined by its microscopic kinetic energy, which is a feature of hard-sphere fluid. Second, we identify the influence of microscopic dynamics on the fluid description; specifically, we demonstrate the emergence of the (electrostatic) capacity of the microscopic interaction potential $v(x)$, which is the scattering length of the hard-core potential, in the coupling constant of the Euler pressure (or the internal specific energy).

 \item Compared with the previous works \cite{chen2023derivation,CSZ23mean},  we address the Gross--Pitaevskii scaling regime, specifically the $\beta=1$ case, and derive the compressible Euler equation with the scattering length $\mathfrak{a}_{0}$. Technically, the natural extension of the previous mentioned works is the semiclassical Gross--Pitaevskii limit regime ($\kappa=0$ case), which we also handled.

 \item  In the beyond Gross--Pitaevskii regime, we give a rigorous derivation of the eikonal system, which justifies the usage of geometric optics approximation to study the dynamics of ultracold Bose gases.
 To our knowledge, this is the first result of its kind.
 Moreover, it offers a physical interpretation for the semiclassical limit of the NLS equation with nonlinearities containing various smallness factors, as seen in \cite{carles2007wkb,carles2021semi}, in terms of the scalings of many-body systems.  Specifically, our result provides a many-body interpretation for NLS models of the form
\begin{align*}
i\ve\pa_{t}\psi=-\ve^{2}\Delta\psi+\ve^{2(1-\kappa)}\mathfrak{c}_{0}|\psi|^{2}
\psi\ .
\end{align*}
 \item We examine the hyper-ultradilute regime ($\beta>1$), which was not addressed in previous works. We find that the limiting equation is the trivial pressureless compressible Euler equation, which, in some sense, explains why this scaling regime is less interesting.

 \item Another interesting result is the convergence rate from the quantum microscopic level to the effective equations in certain critical regimes. This rate depends on the profile of the interaction potential. As a result, it may be possible to infer information about the potential $v(x)$ from the convergence rate, which could be determined via numeric analysis.
\end{enumerate}

\subsubsection{Notations} We collect here the list of parameters and some useful notations for the convenience of the reader.
\begin{itemize}
\item $N$: The particle number which tends to the infinity, $N\to \infty$.
\item $\ve$: The semiclassical scaling parameter which tends to zero, $\ve\to 0$.
\item $\la(N)=(\ln N)^{\upalpha}$, $\upalpha\in [0,1)$: The prefactor of the interaction energy part. For $\upalpha=0$, $\la(N)\equiv 1$. For $\upalpha\in(0,1)$, $\la(N)\to \infty$.
\item $\ka\in [0,1]$: A given parameter which characterizes the different scaling regimes.
\item $\beta\geq 1$: A given parameter which characterizes the different physical regimes.
\item $\mathfrak{a}_{0}^{\mu}$: The scattering length which depends on the parameter $\mu$. When $\mu=1$, we use the short-hand notation $\mathfrak{a}_{0}=\mathfrak{a}_{0}^{1}$.
\item $R_{0}$: The radius of the support set of the potential $v(x)$.
\item $\mathfrak{c}_{0}$: The capacity of the interaction potential $v(x)$ which equals the $R_0$ when $v$ is radial and strictly positive in the interior of the support.
\item $\eta(\mu)=\mathfrak{c}_{0}-\mathfrak{a}_{0}^{\mu}$: The convergence rate between the scattering length $\mathfrak{a}_{0}^{\mu}$ and the capacity $\mathfrak{c}_{0}$, which also depends on the profile of the potential $v(x)$. See Subsection~\ref{subsection:semiclasscial_limit_of_the_scattering_length}.
\item $\mu=N^{1-\beta}\varepsilon^{2(1-\kappa)}/\la(N)$: The shorthand parameter. In most cases except the Gross--Pitaevskii regimes, we have that $\mu \to 0$.  See Subsection~\ref{subsection:the_neumann_problem}.
\item $l=\ve^{4}\to 0$: The localization parameter that occurs in the analysis of the Neumann boundary problem. See Subsection~\ref{subsection:the_neumann_problem}.
\item $L=N^{\beta}\ve^{2\kappa}l\to \infty$: The shorthand parameter that occurs in the analysis of the Neumann boundary problem. See Subsection~\ref{appendix:neumann}.
\item $V_{N}^{\varepsilon}(x):=\lambda(N) (N^{\beta}\varepsilon^{2\kappa})^3v(N^{\beta}\varepsilon^{2\kappa}\, x)$: the interaction potential between particles.
\item $U_N^\varepsilon(x):= \tfrac{1}{N}V_{N}^\varepsilon(x)$: The short notation that occurs in the semiclassical limit of the scattering length.
\item $f_{0}^{\mu}(x)$: The zero-energy scattering solution to Equation~\eqref{def:zero-energy_scattering_problem}.
\item $f_{N}^{\varepsilon}(x):= f^\mu_0(N^\beta \varepsilon^{2\kappa} x)$:  The zero-energy scattering solution of the scaling Problem  \eqref{eq:rescaled_two-body_zero-energy_scattering}.
\item $\asc_N^{\mu}:=\frac{\asc_0^{\mu}}{N^\beta \varepsilon^{2\kappa}}$: The scattering length of the rescaled Problem  \eqref{eq:rescaled_two-body_zero-energy_scattering}.
\item $\mathfrak{b}_0^{\mu}:= \frac{1}{4\pi}\intd \n{\grad f^\mu_0}^2\dd x$: A characteristic length that occurs in the analysis of the approximation of correction terms. See Subsection \ref{subsection:mean-field approximation of energies}.
\item $f^{\mu}_{L}(x)$: The ground state of the Neumann problem \eqref{eq:neumann_boundary_problem}.
\item  $f^\varepsilon_{N, \ell}(x):=f^\mu_L(N^\beta\varepsilon^{2\kappa}x)$: The ground state of the rescaled Neumann problem \eqref{def:correlation_structure_neumann_problem}.
 \end{itemize}

\section{Scattering Function and its Semiclassical Analysis}\label{section:Scattering Function and its Semiclassical Analysis}

In this section, we state some observations of the zero-energy scattering function which will allow us to handle all the limit regimes.

\subsection{Semiclassical limit of the scattering length}\label{subsection:semiclasscial_limit_of_the_scattering_length}
Let $\mu>0$ be a positive parameter. We assume that $v(x)$ is a bounded positive radial function with compact support in the ball $B_{R_0}(0)=\{x\in \R^3\mid \n{x}\le R_0\}$. Consider the zero-energy scattering problem:
\begin{equation}
\left\{
\begin{aligned}\label{def:zero-energy_scattering_problem}
  \(-\mu\,\lapl + v\)f_{0}^{\mu} = 0,\\
  \lim_{|x|\to \infty}f_{0}^{\mu}(x)=1.
\end{aligned}
\right.
\end{equation}
We study the asymptotic behavior of the scattering function $f_{0}^{\mu}$ as the parameter $\mu\to 0$. In particular, let us note that the limit of the scattering length associated  to Problem~\eqref{def:zero-energy_scattering_problem} $$\lim_{\mu\rightarrow 0}\asc_0^\mu$$ has a nontrivial limit. Here, the scattering length is defined by
\begin{align*}
\asc_0^\mu:=\frac{1}{4\pi \mu}\int_{\mathbb{R}^{3}}
v(x)f_{0}^{\mu}(x)\dd x\ ,
\end{align*}
In this section, we investigate its quantitative rate of convergence, and show that the convergence rate depends on the behavior near the boundary of the support of the potential $v(x)$.

The limiting behavior of the scattering length was studied by Kac and Luttinger in \cite{kac1974probabilistic, kac1975scattering} who conjectured that the semiclassical limit of the scattering length yields the electrostatic capacity $\mathfrak{c}_0$ of the support of $v$ (Kac used a probabilistic interpretaton). The conjecture was subsequently proved by Taylor \cite{taylor1976scattering} for $d\ge 3$. For an alternative proof not involving probability, see \cite{tamura1992semi}. However, our case is much simpler than the general positive compactly supported potential case. For radially symmetric compactly supported positive potential on the ball $B_R(0)$, the corresponding capacity is simply the radius of its support.


 By assuming additional assumptions on $v$, we can
 provide a more quantitative analysis of the convergence in the following proposition. The proof is given in Appendix \ref{appendix:semiclassical_two-body_problem}.

\begin{prop}\label{prop:quantitative_scattering_length}
    Suppose $v$ is a bounded  nonnegative radial potential satisfying the following assumptions:
    \begin{enumerate}[$(a)$]
    \item The interior of the support of $v$ is a simply connected region, that is,
    \begin{equation*}
    \begin{cases}
    v(x)>0 & \text{ if } |x|<R_{0}\\
    v(x)=0 & \text{ if }  |x|\geq R_{0}
    \end{cases}
    \ .
    \end{equation*}
    In particular, the capacity of $v$ is $\mathfrak{c}_{0}=R_{0}$.
    \item $v$ has $n$-order of vanishing on the boundary $\{\n{x}=R_0\}$ for some $n\in [0,\infty)$, i.e.,
    \begin{align*}
     v(x)=(R_{0}-|x|)^{n}v_{1}(x)\ ,
    \end{align*}
    with $v_{1}(x)>0$ for $|x|\leq R_{0}$.
    \end{enumerate}
    Consider the zero-energy scattering Problem \eqref{def:zero-energy_scattering_problem}.
    Then we have the quantitative rate of convergence for the scattering length
    \begin{align}\label{def:convergence_rate_of_scattering_length}
    \eta(\mu):=\mathfrak{c}_{0}-\mathfrak{a}_{0}^{\mu}=\mathcal{O}(\mu^{\frac{1}{n+2}})\ .
    \end{align}
    A standard example is given by
    $v(x)= \id_{B_{1}(0)}(x)(1-|x|)^{n}$, for which the capacity is $\mathfrak{c}_{0}=1$ and the convergence rate is $1-\mathfrak{a}_{0}^{\mu}=\mathcal{O}(\mu^{\frac{1}{n+2}})$.
\end{prop}

\begin{remark}
    Proposition $\ref{prop:quantitative_scattering_length}$ indicates that the convergence rate of the scattering length to the capacity, $\eta(\mu)$, depends highly on certain aspects of the profile of the potential.
\end{remark}

\subsection{Scattering length of the scaled potential and the Neumann problem}\label{subsection:the_neumann_problem}
We consider the zero-energy scattering problem
\begin{align}\label{eq:rescaled_two-body_zero-energy_scattering}
    \(- \varepsilon^2\lapl + U_N^\varepsilon\) f_{N}^{\ve} = 0 \quad \text{ and } \quad \lim_{\n{x}\rightarrow \infty} f_{N}^{\ve}(x) = 1\ ,
\end{align}
where the potential
$$U_N^\varepsilon(x)=\tfrac{\la}{N} (N^{\beta}\varepsilon^{2\kappa})^3v(N^{\beta}\varepsilon^{2\kappa}\, x)$$
is given by \eqref{def:rescaled_interaction_potential}.
By rescaling, the rescaled solution $f_{N}^{\varepsilon}(x):= f^\mu_0(N^\beta \varepsilon^{2\kappa} x)$ solves Problem \eqref{eq:rescaled_two-body_zero-energy_scattering}, where $f_{0}^{\mu}$ is the solution to the zero-energy scattering problem \eqref{def:zero-energy_scattering_problem}
with
$$\mu := N^{1-\beta}\varepsilon^{2(1-\kappa)}/\lambda\ .$$
 In particular, the corresponding scattering length of Problem \eqref{eq:rescaled_two-body_zero-energy_scattering}, denoted by $\asc_N^{\mu}$, is
\begin{align}
    \asc_N^{\mu}:=  \frac{1}{4\pi\, N\varepsilon^2}\intd \lambda \(N^{\beta}\varepsilon^{2\kappa}\)^3v(N^\beta \varepsilon^{2\kappa} x) f^\varepsilon_{N}(x)\dd x = \frac{\asc_0^{\mu}}{N^\beta \varepsilon^{2\kappa}}\ ,
\end{align}
where $\asc_0^{\mu}$ is the scattering length corresponding to Problem \eqref{def:zero-energy_scattering_problem}.

For the technical reasons, we shall localize to regions of size $\ell^3$ where $\asc_N^{\mu}\ll \ell$ to observe the two particle correlation structure. More precisely, we shall choose $\ell$ so that the `size' of the total correlation within the region is of order one; in fact, we shall choose $\ell=\varepsilon^4$. Therefore, we need to consider more than just the solution to Problem \eqref{eq:rescaled_two-body_zero-energy_scattering} for the correlation structure between two particles, we have to consider $f^\varepsilon_{N, \ell}$ the ground state of the Neumann problem
\begin{equation}\label{def:correlation_structure_neumann_problem}
\left\{
\begin{aligned}
    &\(-\varepsilon^2\lapl+ U_N^\varepsilon\) f^\varepsilon_{N, \ell} = E_{N, \ell}^\varepsilon f^\varepsilon_{N, \ell},\quad |x|\leq l,\\
    & f_{N,l}^{\ve}(l)=1,\quad \pa_{r}f_{N,l}^{\ve}(l)=0.
\end{aligned}
\right.
\end{equation}

Let us summarize some of its important properties in the following lemma, with the case $(\beta, \varepsilon, \lambda) = (1, 1, 1)$ already provided in \cite[Appendix A]{erdos2006derivation}. We need to extend the analysis to the semiclassical regime $\mu \to 0$. Due to the nonlinear dependence between the scattering length $\mathfrak{a}_{0}^{\mu}$ and the parameter $\mu$, a linear scaling analysis would not work smoothly. Instead, based on \cite{erdos2006derivation}, our proof also relies on techniques from elliptic equations, including the comparison principle and the construction of auxiliary functions. For the sake of clarity, we have included the proof in Appendix~\ref{appendix:semiclassical_two-body_problem}.

\begin{lem}\label{lem:correlation_structure}
    Suppose $v$ is a nonnegative bounded function with compact support  and $\asc_0^{\mu}$ be the corresponding scattering length. Let $f^\varepsilon_{N, \ell}$ be the ground state of the Neumann problem \eqref{def:correlation_structure_neumann_problem}.
    Then for $N$ sufficiently large (so that $R_0N^{-\beta}\varepsilon^{-2\kappa}<\ell$) we have:
    \begin{enumerate}[$(i)$]
        \item The ground state energy satisfies the asymptotic upper bound
\begin{align}\label{eq:neumann_gs_eigenvalue_approximation}
            E_{N, \ell}^\varepsilon \le \frac{3\asc_N^{\mu} \varepsilon^2}{\ell^3}\(1+\mathcal{O}\(\frac{R_0}{\ell\varepsilon^{2\kappa}N^\beta}\)\)= \frac{3\asc_0^{\mu}\varepsilon^2}{\ell^3\varepsilon^{2\kappa}N^\beta}\(1+\mathcal{O}\(\frac{R_0}{\ell\varepsilon^{2\kappa}N^\beta}\)\)\ .
        \end{align}
        \item There exists $\sfc_{1}^{\mu}>0$, dependent on $v$ and $\mu$, such that
        \begin{align}
           0<  \exp\(-\zeta_0^\frac12/\mu^{\frac12}\)\leq \sfc_{1}^{\mu} \le f^\varepsilon_{N, \ell}(x)\le 1\ ,
        \end{align}
        where $\zeta_0:= R_0^2  \Nrm{v}{L^\infty}$.
        Moreover, let $w^\varepsilon_{N, \ell}(x):= 1-f^\varepsilon_{N, \ell}(x)$. Then we have the following pointwise estimate
        \begin{align}\label{est:correlation_w_pointwise_bound}
            w^\varepsilon_{N, \ell}(x)\le \min\(1-\sfc_{1}^{\mu}, \frac{\asc_0^{\mu}}{N^{\beta}\varepsilon^{2\kappa}\n{x}}\)\ .
        \end{align}
        \item We also have the following pointwise gradient estimates
        \begin{align}\label{est:correlation_grad_w_pointwise_bound}
            \n{\grad  w^\varepsilon_{N, \ell}(x)} \le  \frac{\asc_0^{\mu}+\mathcal{O}(\frac{R_0}{\ell\varepsilon^{2\kappa}N^\beta})}{N^{\beta}\varepsilon^{2\kappa}\n{x}^2} \quad \text{ and } \quad \n{\grad  w^\varepsilon_{N, \ell}(x)} \le \frac{2+\mathcal{O}(\frac{R_0}{\ell\varepsilon^{2\kappa}N^\beta})}{\n{x}}
        \end{align}
        for all $\n{x}\le \ell$.
        \item  Finally, there exists $C>0$, independent of $\mu$, such that we have the following uniform point-wise bound on the gradient
        \begin{align}\label{est:correlation_w_uniform_bound}
            \n{\grad w_{N, \ell}^\varepsilon(x)} \le \frac{C N^{\beta}\varepsilon^{2\kappa}}{N^{2\beta}\varepsilon^{4\kappa}\n{x}^2+R_0^2}\(1+\frac{\mathfrak{c}_0-\asc_0^{\mu}}{\mu}\)\ .
        \end{align}
    \end{enumerate}
\end{lem}
\section{Bogoliubov Approximation of the Many-Body Dynamics}\label{section:bogoliubov_approximation}

\subsection{Second quantization} Let us start by stating some notations and reviewing some rudimentary facts about the second quantization formalism. For simplicity, we may forgo some of the mathematical rigor and care that is usually required when handling unbounded operators, which in our case are the creation and annilihation operators defined shortly below. For a more comprehensible treatment of the subject, we refer the interested audience to \cite{berezin_method_1966,bratteli1997operator,folland1989harmonic, solovej2007many}.

\subsubsection{Operators on the Fock Space}



For every $x \in \R^3$, we define the creation and annihilation operator-valued distributions, denoted by $a^\ast_x$ and $a_x$ respectively, by their actions on the sectors $\h^{\otimes_s n-1}$ and $\h^{\otimes_s n+1}$ of $\cF_s$ as follows
\begin{subequations}
    \begin{align}
        (a^\ast_x\Psi)_n =&\, \frac{1}{\sqrt{n}}\sum_{\jj=1}^n \delta(x-x_{\jj})\,\psi_{n-1}(x_1, \ldots, x_{\jj-1}, x_{\jj+1},\ldots, x_n),\\
        (a_x\Psi)_n =&\, \sqrt{n+1}\, \psi_{n+1}(x, x_1, \ldots, x_n)\ .
    \end{align}
\end{subequations}
Then, it can be readily checked that $a^\ast_x$ and $a_x$ satisfy the following canonical commutation relations (CCR):
\begin{align}
    [a_x, a^\ast_y] = \delta(x-y) \quad \text{ and } \quad [a_x, a_y]=[a^\ast_x, a^\ast_y] =0\ .
\end{align}
Also, for each $f \in \h$, we also define the operators
\begin{align}
    a^\ast(f)= \intd f(x)\, a^\ast_x \dd x \quad \text{ and } \quad a(f)=\intd \conj{f(x)}\, a_x\dd x\ .
\end{align}


 Let $O$ be a single particle observable on $\h$ with the distributional kernel $O(x, y)$, then we define the second quantized operator associated to $O$ by
\begin{align}
    \dG(O) = \intdd O(x, y)\, a^\ast_x a_y\dd x\d y\ .
\end{align}
An important example of second quantized operators is given by the second quantization of the identity map
\begin{align}
    \cN := \dG(\id) = 0\oplus \bigoplus^\infty_{n=1} n\, \id_{\h^{\otimes n}} =\intd a^\ast_x a_x\dd x
\end{align}
which counts the number of particles in each sector of $\cF_s$. More generally, for every $\alpha>0$ and real $K$, we could defined the operator
\begin{align*}
    \(\cN+K\)^\alpha := K\oplus \bigoplus^\infty_{n=1} (n+K)^\alpha\, \id_{\h^{\otimes n}}\ .
\end{align*}
Let us state some standard bounds for these operators without proof.
\begin{lem}\label{lem:estimates_for_creation_annilihation_operators}
    Let $f \in\h$, then we have the bounds
    \begin{align}
        \Nrm{a(f)\Psi}{} \le \Nrm{f}{L^2}\Nrm{\cN^\frac12 \Psi}{},\quad
        \Nrm{a^\ast(f)\Psi}{} \le \Nrm{f}{L^2}\Nrm{(\cN+1)^\frac12 \Psi}{}\ .
    \end{align}
    If $O$ is a bounded operator on $\h$, then we have the bound
    \begin{align}
        \Nrm{\dG(O)\Psi}{} \le \Nrm{O}{\mathrm{op}}\Nrm{\cN\,\Psi}{}\ .
    \end{align}
\end{lem}

If $W$ is a two-particle observable on $\h\otimes_s \h$, we define its second quantization by
\begin{align}
    \dG^{(2)}(W) := \frac12 \intdd \intdd W(x, y, x', y')\, a^\ast_x a^\ast_y a_{y'} a_{x'}\dd x\d y\d x' \d y'\ .
\end{align}
For our purpose, the most relevant example of second quantized two-particle operators is given by the second quantization of the pairwise interaction potential $v_{12}=v(x_1-x_2)$ on $\h\otimes_s\h$ with distributional kernel $v(x_1-x_2)\delta(x_1-x'_1)\delta(x_2-x_2')$ which is given by
\begin{align}
    \dG^{(2)}(v_{12}) = \frac12 \intdd v(x-y)\, a^\ast_x a^\ast_y a_y a_x\dd x\d y\ .
\end{align}

One of the more important second quantized operators that we will be working with is the Fock Hamiltonian operator acting on $\cF_s$  defined by
\begin{subequations}
    \begin{align}
        &\cH_N :=\, \cK + N^{-1}\cV\ ,\\
        &\cK :=\,\dG(-\tfrac{\varepsilon^2}{2}\lapl)  =\intd a^\ast_x\(-\tfrac{\varepsilon^2}{2}\lapl\)a_x\dd x\ ,\\
        &\cV:=\,  \frac12 \intdd V^\varepsilon_N(x-y)\, a^\ast_x a^\ast_y a_y a_x\dd x\d y\ .
    \end{align}
\end{subequations}
where  $V^\varepsilon_N(x)= \lambda(N^\beta \varepsilon^{2\kappa})^3v(N^\beta \varepsilon^{2\kappa}\, x)$. By direct computation, one can readily verify that $\cH_N$ commutes with $\cN$ (any operator on $\cF_s$ that commutes with $\cN$ is called a diagonal operator) and that
\begin{align}
    \(\cH_N\Psi\)_n = \(\sum^n_{j=1}-\frac{\varepsilon^2}{2}\lapl_{x_j}+ \frac{1}{N}\sum^n_{i<j}  V_N^\varepsilon \(x_i- x_j\)\)\psi_n\ ,
\end{align}
which yields the mean-field Hamiltonian $H_N$ on the $N$th sector. This agrees with the definition of $\cH_N$ given in the introduction. In particular, this means that the $N$th sector of the equation
\begin{align}\label{eq:Fock_Schrodinger}
    i\varepsilon\,\bd_t\Psi = \cH_N \Psi
\end{align}
encodes the $N$-body Schr\"odinger equation \eqref{eq:dimensionless_N-body}. Hence this procedure of lifting the $N$-body quantum dynamics to a Schr\"odinger dynamics on Fock space is known as the second quantization of the $N$-body system.

\subsubsection{Weyl Operators and Coherent States} Just as in the case of the Cauchy problem associated to the $N$-body Schr\"odinger equation \eqref{eq:dimensionless_N-body}, it is highly nontrivial to obtain an effective dynamical description for the many-body Schr\"odinger dynamics \eqref{eq:Fock_Schrodinger} starting from a general initial state. Therefore, it is natural to restrict our studies to a subclass of initial data which is also physically meaningful. An important class of states to consider is the set of coherent states where the elements exhibit attributes of classical fields. In general, coherent states can be viewed as the lifting of the $N$-fold tensor product $\phi^{\otimes N}:=\phi\otimes\cdots \otimes\phi$ states to $\cF_s$. From a physical perspective, $\Psi(f)$ offers a way to approximate the condensate in a Bose gas by means of a classical field $f$ which we shall discuss later.

For every $f \in \h$, we define the following unbounded, densely defined skew-adjoint closed operator
\begin{align}
    \cA(f) := a(f)-a^\ast(f)
\end{align}
and its corresponding bounded unitary operator
\begin{align}\label{def:Weyl_operator}
    e^{-\cA(f)} = e^{a^\ast(f)-a(f)}
\end{align}
on $\cF_s$ called the Weyl operator (also known as the displacement operator). With this, we could define the coherent state associated to $f$ by
\begin{align}\label{def:coherent_state}
    \Psi(f)= e^{-\cA(f)}\Omega\ .
\end{align}

Let us collect some well-known useful properties of Weyl operators and coherent states in the following lemma. The proof can be found in standard references, e.g. \cite[Section 5.2.1.2]{bratteli1997operator}.
\begin{lem}\label{lem:properties_of_weyl_operator}
    For every $f \in \h$, we define the unitary operator $e^{-\cA(f)}$. Then it follows that
    \begin{enumerate}[$(i)$]
        \item For every $f, g \in \h$, we have the identity
        \begin{align}
            e^{-\cA(f)}e^{-\cA(g)}= e^{-\cA(g)}e^{-\cA(f)}\,e^{-2i\im \inprod{f}{g}}
        \end{align}
        called the Weyl relation, which encodes the canonical commutation relations.
        \item For every $f \in \h$, we have the
        following operator-valued distribution identities
        \begin{subequations}\label{eq:operator-valued_dist_weyl_conjugation_id}
            \begin{align}
                e^{\cA(f)} a^\ast_x e^{-\cA(f)} =&\, a^\ast_x+\conj{f(x)}\ ,\\
                e^{\cA(f)} a_x e^{-\cA(f)} =&\, a_x + f(x)\ .
            \end{align}
        \end{subequations}
        \item The set of coherent states $\ecS=\{\Psi(f)\in \cF_s \mid f\in \h\}$ is linearly independent and the linear manifold $\operatorname{Span}\ecS$ generated by its algebraic span is dense in $\cF_s$.
        \item \label{lem:coherent_state_form} On the domain $\operatorname{Span}\ecS$ the Weyl operator admits the factorization
        \begin{align}
            e^{-\cA(f)} = e^{-\tfrac12 \Nrm{f}{}^2}e^{a^\ast(f)}e^{-a(f)}
        \end{align}
        which yields
        \begin{align}
            \Psi(f) =  e^{-\tfrac12 \Nrm{f}{}^2}e^{a^\ast(f)}\Omega  = e^{-\tfrac12 \Nrm{f}{}^2}\(1\oplus\bigoplus^\infty_{n=1} \frac{1}{\sqrt{n!}}\, f^{\otimes n}\)\ .
        \end{align}
    \end{enumerate}
\end{lem}

\subsubsection{Bogoliubov Transformation} In this section, we give the definition of the Bogoliubov transformation and briefly recall some its rudimentary properties. A rigorous treatment of the content presented below via the Weyl relations can be found in the work of Shale \cite{shale1962linear}. Also see \cite{berezin_method_1966} for a rigorous introduction to Bogoliubov transformations.

For $f_1\oplus J f_2 \in \h\oplus \h^\ast$ where $J f:= \conj{f}$, we define the generalized creation and annilihation operators $A^\ast$ and $A$ on $\h\oplus\h^\ast$ by
\begin{subequations}
    \begin{align}
        A^\ast(f_1\oplus Jf_2):=&\, a^\ast(f_1)+a(f_2)\ ,\\
        A(f_1\oplus Jf_2):=&\, a(f_1)+a^\ast(f_2)\ .
    \end{align}
\end{subequations}
Notice $A^\ast$ is a $\CC$-linear mapping from $\h\oplus\h^\ast$ to the space of operators on $\cF_s$ whereas $A$ is conjugate linear. Moreover, $A^\ast$ and $A$ satisfy the relation
\begin{align}
    A^\ast(f_1\oplus J f_2)  =  A(\sfJ (f_1\oplus J f_2))
\end{align}
where $\sfJ:\h\oplus\h^\ast\rightarrow \h\oplus\h^\ast$  is the anti-linear map defined by $\sfJ(f_1\oplus J f_2):= f_2\oplus J f_1$. From the CCR of $a$ and $a^\ast$, one could readily show that the generalized creation and annilihation operators satisfy the following commutation relations: for any two pairs $F=f_1\oplus J f_2, G = g_1\oplus J g_2 \in \h\oplus \h^\ast$, we have that
\begin{align*}
    [A(F), A^\ast(G)] = \inprod{F}{\sfS G}_{\h\oplus\h^\ast} \quad \text{ where } \quad  \sfS
    G = \begin{pmatrix}
        \id & 0\\
        0 & -\id
    \end{pmatrix}
    \begin{pmatrix}
        g_1\\
        g_2
    \end{pmatrix}
    =
    \begin{pmatrix}
        g_1\\
        -g_2
    \end{pmatrix}\ .
\end{align*}
A more convenient way to rewrite the above commutation relation is via the following expression
\begin{align*}
    [A(F), A(G)] = \inprod{F}{\sfS\sfJ G}_{\h\oplus\h^\ast} =: B(F, G)
\end{align*}
where $B(\cdot, \cdot)$ is a non-degenerate skew-symmetric bilinear form.

A linear continuous isomorphism $\nu:\h\oplus\h^\ast \rightarrow \h\oplus\h^\ast$ is called a (bosonic) Bogoliubov map if
\begin{align}
    A^\ast(\nu F) = A(\nu \sfJ F)
\end{align}
for all $F \in \h\oplus\h^\ast$ and
\begin{align}
    [A(\nu F), A(\nu G)] = \inprod{F}{\sfS\sfJ G}_{\h\oplus\h^\ast}
\end{align}
for all $F, G\in \h\oplus\h^\ast$. Equivalently, this means that $\nu$ is a Bogoliubov map provided it satisfies the conditions
\begin{align}\label{conditions:bogoliubov_map}
    \nu \sfJ = \sfJ \nu \quad \text{ and } \quad \sfS\sfJ = \nu^\ast \sfS\sfJ \nu \ .
\end{align}
In fact, the class of mapping that satisfies Conditions \eqref{conditions:bogoliubov_map} are the infinite-dimensional generalization of the class of $2n\times 2n$ real symplectic matrices $\mathrm{Sp}(2n)$ and we shall denote it by $\mathrm{Sp}(\CC)$. In particular, a general Bogoliubov map $\nu$ has the form
\begin{align}
    \nu =
    \begin{pmatrix}
        U & \conj{V}\\
        V & \conj{U}
    \end{pmatrix}
\end{align}
where $U, V: \h\rightarrow \h$ are bounded linear maps satisfying
\begin{align}
    U^\ast U - V^\ast V = \id \quad \text{ and } \quad U^\ast \conj{V}=V^\ast \conj{U}\ .
\end{align}
Here we have used $\conj{A}$ to denote $J A J$.

Following \cite{grillakis2010second, grillakis2011second}, we could also study the infinitesimal generators of the Bogoliubov map. Consider block `matrices' defined on $\h\oplus \h^\ast$ of the form
\begin{align}\label{def:symplectic_matrix}
    L =
    \begin{pmatrix}
        d & k\\
        l & -d^\top
    \end{pmatrix}
\end{align}
where $d:\h\rightarrow \h$ is a self-adjoint operator, and $k:\h^\ast \rightarrow \h$ and $l:\h \rightarrow \h^\ast$ are operators with symmetric kernel. This class of operators generalizes the lie algebra of the real symplectic matrices $\mathsf{sp}(2n)$, which in turn are yields all the infinitesimal generators of the Bogoliubov maps. For the purpose of work, we only need to consider the matrix
\begin{align}
    \sfK =
    \begin{pmatrix}
        0 & k\\
        \conj{k} & 0
    \end{pmatrix}
\end{align}
which is the infinitesimal generator of the symplectic matrix
\begin{align}\label{def:bogoliubov_map_of_k}
    \nu_k := e^\sfK =
    \begin{pmatrix}
        \ch(k) & \sh(k)\\
        \conj{\sh(k)} & \conj{\ch(k)}
    \end{pmatrix}
\end{align}
where $\ch(k)$ and $\sh(k)$ are defined by the following formal power series
\begin{align}
  \ch(k)= \sum^\infty_{n=0}\frac{(k\,\conj{k})^n}{(2n)!} \quad \text{ and } \quad \sh(k) = \sum^\infty_{n=0}\frac{(k\,\conj{k})^n k}{(2n+1)!}\ .
\end{align}
Here, $k$ is viewed as an operator and the product is given by composition of linear maps. We could also define the operators $\ch(2k)$ and $\sh(2k)$ using the above power series.
Furthermore, notice the above series are convergent in the operator norm provided $k$ is a Hilbert--Schmidt operator. It is also convenient to define the operators
\begin{subequations}
    \begin{align}
        p(k):=&\, \ch(k)-\id\ , \quad p(2k):=\ch(2k)-\id\ ,\\
        r(k):=&\, \sh(k)-k\ .
    \end{align}
Sometimes for compactness of notations, we simply write
\begin{align}\label{def:sh_ch_and_r}
    \sh:=&\, \sh(k), \quad \ch:=\ch(k), \quad p:=p(k), \quad r:=r(k)\ ,\\
    p_2:=&\, p(2k), \quad \sh_2:=\sh(2k), \quad r_2:=r(2k)\ .
\end{align}
\end{subequations}
It is straightforward to check that
\begin{subequations}\label{eq:hyperbolic_trig_identities}
    \begin{align}
        &\tfrac12\,\sh(2k) =\, \ch(k)\,\sh(k) = \sh(k)\,\conj{\ch(k)} = k+\tfrac12\, r(2k)\ ,\\
        &\tfrac12\,\(\ch(2k)-\id\) =\, \sh(k)\,\conj{\sh(k)}=\ch(k)^2-\id=\tfrac12\, p(2k)\ .
    \end{align}
\end{subequations}

Given a Bogoliubov map $\nu$, we call the associated mapping $A^\ast(F)\mapsto A^\ast(\nu F)$ of Fock space operators a Bogoliubov transformation. For instance, if we consider the Bogoliubov map \eqref{def:bogoliubov_map_of_k},  we have that
\begin{subequations}
    \begin{align}
        b^\ast(f):=&\,A^\ast(\nu_k\,(f\oplus 0))= a^\ast\(\ch(k)f\)+a\(\sh(k)\conj{f}\),\\
        b(f):=&\,A(\nu_k\,(0\oplus J f))= a^\ast(\sh(k)\conj{f})+a(\ch(k)f)
    \end{align}
\end{subequations}
are the transformed creation and annilihation operators. It is clear from the definition of Bogoliubov maps that $b^\ast$ and $b$ satisfies the canonical commutation relations. Moreover, a Bogoliubov transformation is said to be unitarily implementable in $\cF_s$ provided there exists a unitary map $\cU_\nu:\cF_s\rightarrow \cF_s$ such that
\begin{align*}
    A(\nu F) = \cU_\nu^\ast\, A(F)\,\cU_\nu
\end{align*}
for all $F \in \h\oplus \h^\ast$.  A necessary and sufficient condition for the transformation $\nu$
to be implementable is given by Shale \cite{shale1962linear}: a Bogoliubov map $\nu$ is unitarily implementable if and only if $V$ is a Hilbert--Schmidt operator. It is standard to refer to $\cU_\nu$ as the Bogoliubov transformation in the physics literature and is called the Segal--Shale--Weil (or metaplectic) representation (cf. \cite[Chapter 4]{folland1989harmonic}).

As observed in \cite{grillakis2010second, grillakis2011second}, a useful tool for computation is given by the infinitesimal Segal--Shale--Weil representation: for any block matrix operator $L$ of the form \eqref{def:symplectic_matrix}, we defined its mapping $\cI$ to the a class of skew-Hermitian quadratic operators acting $\cF_s$ via
\begin{align*}
    \cI(L):=-\frac12 \intdd d(x, y)\,a_x a^\ast_y + d(y, x)\, a^\ast_x a_y + k(x, y)\,a^\ast_x a^\ast_y - l(x, y)\,a_x a_y\dd x\d y\ .
\end{align*}
The crucial property of $\cI$ is the fact that it is a Lie algebra isomorphism between $\mathsf{sp}(\CC)$ and the space of skew-Hermitian quadratic operators, which means that
\begin{align*}
    \lt[\cI(L_1), \cI(L_2)\rt] = \cI([L_1, L_2])\ .
\end{align*}
In the case of the Bogoliubov map \eqref{def:bogoliubov_map_of_k}, if we know that $\nu_k$ is implementable, that is,
$\sh(k)$ is a Hilbert--Schmidt operator, then it can be show that
\begin{align}\label{def:Bogoliubov_transformation_of_k}
    e^{\cI(\sfK)} = e^{-\cB(k)}
\end{align}
is the unitary implementation of $\nu_k$ where its generator is given by the skew-adjoint operator
\begin{align*}
    \cB(k)= \frac12\intdd \conj{k(x, y)}\,a_xa_y-k(x, y)\,a^\ast_x a^\ast_y\dd x\d y\ .
\end{align*}

Let us collect some useful facts regarding $e^{-\cB(k)}$ and $\nu_k$. The proof of the following lemmas can be found in \cite{benedikter2015quantitative,berezin_method_1966, grillakis2010second, grillakis2011second}.

\begin{lem}\label{lem:bogoliubov_conjugation}
    Let $\nu_k$ denote the Bogoliubov map \eqref{def:bogoliubov_map_of_k} where $k$ is a Hilbert--Schmidt operator and the Bogoliubov transformation $e^{-\cB(k)}$ is defined by \eqref{def:Bogoliubov_transformation_of_k}. Then  for every $F\in \h\oplus\h^\ast$, we have that
        \begin{align*}
            e^{\cB(k)} A(F) e^{-\cB(k)} = A(\nu_k F)\ .
        \end{align*}
        In particular, for every $f \in \h$ we have the operator identities
        \begin{subequations}
            \begin{align}
               b^\ast(f):=&\, e^{\cB(k)} a^\ast(f) e^{-\cB(k)}  = a^\ast\(\ch(k)f\)+a\(\sh(k)\conj{f}\)\ , \\
               b(f):=&\, e^{\cB(k)} a(f) e^{-\cB(k)}  = a^\ast(\sh(k)\conj{f})+a(\ch(k)f)\ ,
            \end{align}
        \end{subequations}
        which yields the operator-valued distribution identities
        \begin{subequations}\label{eq:bogoliubov_conjugation_operator-value_distribution_id}
            \begin{align}
                b^\ast_x:=&\, e^{\cB(k)} a^\ast_x e^{-\cB(k)}  = \intd \ch_{x}(y)\, a^\ast_y+\conj{\sh_x(y)}\,a_y\dd y\ , \\
                b_x:=&\, e^{\cB(k)} a_x e^{-\cB(k)}  = \intd \sh_x(y)\, a^\ast_y+\conj{\ch_{x}(y)}\,a_y\dd y\ ,
             \end{align}
        \end{subequations}
        where $\sh_x(y):= \sh(y, x)=\sh(k)(y, x)$ and $\ch_{x}(y):=c(y, x)=\ch(k)(y, x)$.
\end{lem}

\begin{lem}\label{lem:bogoliubov_conjugation_operator_inequality}
    Under the same conditions as in the previous lemma. Then there exists a univerisal $C>0$ such we have the following operator inequality
        \begin{align}
            e^{\cB(k)}\,\cN\, e^{-\cB(k)}\le&\ C\(1+\Nrm{p}{\mathrm{HS}}^2+\Nrm{\sh}{\mathrm{HS}}^2\) \(\cN+1\)\ , \label{est:bogoliubov_conjugation_for_number_operator}\\
            e^{\cB(k)}\,\cN^2\, e^{-\cB(k)}\le&\ C\(1+\Nrm{p}{\mathrm{HS}}^2+\Nrm{\sh}{\mathrm{HS}}^2\)^2 \(\cN+1\)^2\ .
        \end{align}
\end{lem}

\subsubsection{Bogoliubov states and the Bogoliubov approximation} Despite the many rich properties of the class of coherent states, they only allow us to model the average behavior of a many-body system without capturing their particle correlations, which are necessary if we want to understand the behavior of an interacting system close to the ground state or under stronger topologies. To consider systems with correlations, we introduce the class of generalized Bogoliubov states: for every $\phi \in \h$  and $k \in L^2(\R^3\times\R^3)$ satisfying $k(x ,y)=k(y, x)$, we define the generalized Bogoliubov state by
\begin{align}\label{def:squeezed_coherent_states}
    \Psi^{\mathrm{Bog}}=\Psi(\sqrt{N}\phi, k):=e^{-\cA(\sqrt{N}\phi)} e^{-\cB(k)}\Xi_0\ ,
\end{align}
where $\Xi_0 \in \cF_s$ is a unit Fock vector satisfying the following uniform bound in $N$ and $\varepsilon$, that is, there exists $D_0$, independent of $N$ and $\varepsilon$, such that
\begin{align}\label{condition:Xi_0}
    \inprod{\Xi_0}{\(\cH_N+\frac{\cN^2}{N}+\cN+1\)\, \Xi_0} \le D_0\ ,
\end{align}
meaning the expected number of particles and energy in the state $\Xi_0$ are of order one for all $N$. In the case when $\Xi_0=\Omega$, $\Psi^{\mathrm{Bog}}$ is called a Bogoliubov state (also referred to as a squeezed coherent state) parametrized by $\phi$ and $k$.

In our case, we consider the Cauchy Problem~\eqref{def:Fock_space_cauchy_problem} on $\cF_s$ with initial data
\begin{align}\label{eq:Fock_space_cauchy_problem_data}
        \Psi(0)= e^{-\cA(\sqrt{N}\phi^{\init})} e^{-\cB(k_0)}\Xi_0\ ,
\end{align}
where the pair excitation function $k_0$ has the form
\begin{align}\label{def:k_initial}
    k^\varepsilon_{N, 0}(x, y) =
    -N w_{N, \ell}^{\varepsilon}(x-y) \phi^{\init}(x)\phi^{\init}(y)
\end{align}
and the condensate wave function is normalized $\Nrm{\phi^{\init}}{L^2}=1$,  $w_{N, \ell}^{\varepsilon}= 1-f_{N, \ell}^{\varepsilon}$
where $f_{N, \ell}^{\varepsilon}$ solve Problem \eqref{def:correlation_structure_neumann_problem}. Moreover, as discussed in Subsection~\ref{subsection:the_neumann_problem}, we work with the localization parameter $\ell = \varepsilon^4$.

In particular, the solution to the Cauchy Problem \eqref{def:Fock_space_cauchy_problem}
\begin{align}
    \Psi_{N, t}:= e^{-it\cH_N/\varepsilon}e^{-\cA(\sqrt{N}\phi^{\init})} e^{-\cB(k_0)}\Xi_0
\end{align}
can be approximated by
\begin{align}
    \Psi_{\mathrm{approx}, t}:=e^{-\sqrt{N}\cA(\phi_{N, t}^{\varepsilon})}e^{-\cB(k^{\varepsilon}_{N, t})}\Xi_0
\end{align}
where $\phi_{N, t}^{\varepsilon}$ solves the modified Gross--Pitaevskii equation
\begin{align}\label{eq:modified_Gross--Pitaevskii}
    i\varepsilon\,\bd_t\phi = -\tfrac{\varepsilon^2}{2}\lapl \phi + (K\ast |\phi|^2)\phi
\end{align}
with the effective potential
\begin{align}\label{def:effective_potential_K}
    K(x):=\lambda(N) (N^{\beta}\varepsilon^{2\kappa})^3 v(N^\beta \varepsilon^{2\kappa} x)f_{N, \ell}^{\varepsilon}(x)
\end{align}
and the pair excitation function is given by
\begin{align}\label{def:pair_excitation_function}
    k^\varepsilon_{N, t}(x, y) =
        -N w_{N, \ell}^{\varepsilon}(x-y)  \phi_{N, t}^{\varepsilon}(x)\phi_{N, t}^{\varepsilon}(y)\ .
\end{align}
Properties of $\phi^\varepsilon_{N, t}$ are given in Appendix~\ref{appendix:Dispersive PDEs} and
estimates on the pair excitation function $k^\varepsilon_{N, t}$ can be found in Appendix~\ref{appendix:pair_excitation_estimates}. Lastly, let us also define the fluctuation vector
\begin{align}\label{def:fluctuation_state}
    \Psi_{\mathrm{fluc}, t}:=  e^{\cB(k_{N, t}^\varepsilon)} e^{\cA(\sqrt{N}\phi_{N,t}^\varepsilon)} e^{-i t\cH_N/\varepsilon}e^{-\cA(\sqrt{N}\phi^{\init})} e^{-\cB(k_0)}\Xi_0\ .
\end{align}


Let us end this section with the following lemma regarding the generalized Bogoliubov states.

\begin{lem}\label{lem:bogoliubov_states}
    Let $\phi \in L^2(\R^3)$ and $k\in L^2(\R^3\times \R^3)$ satisfying $k(x, y)=k(y, x)$ such that their norms $\Nrm{\phi}{L^2_x}$ and $\Nrm{k}{\mathrm{HS}}$ are uniformly bounded in $N$ and $\varepsilon$. Consider the generalized Bogoliubov state $\Psi^{\mathrm{Bog}}$ as described above.
    Then $\Psi^{\mathrm{Bog}}$ satisfies the following properties
        \begin{align}
            &\inprod{\Psi^{\mathrm{Bog}}}{\cN\,\Psi^{\mathrm{Bog}}} =  N+\mathcal{O}(\sqrt{N}) \ ,\label{est:expected_number_bogoliubov_states}\\
            &\inprod{\Psi^{\mathrm{Bog}}}{\cN^2\,\Psi^{\mathrm{Bog}}} =  N^2+\mathcal{O}(N^{\frac32}) \ ,\label{est:expected_number-squared_bogoliubov_states}\\
              &\inprod{\Psi^{\mathrm{Bog}}}{\cN^2\,\Psi^{\mathrm{Bog}}}-\inprod{\Psi^{\mathrm{Bog}}}{\cN\,\Psi^{\mathrm{Bog}}}^2 =  \mathcal{O}(N) \ .\label{est:variance_bogoliubov_states}
        \end{align}
 Moreover, if we consider $k$ of the form \eqref{def:k_initial} and the Hamiltonian
    \begin{align*}
            \cH_N^{\rm trap} = \dG(-\tfrac{\varepsilon^2}{2}\lapl+W_{\rm trap})+N^{-1}\cV \ ,
    \end{align*}
    then we see that
        \begin{align}
            &\inprod{\Psi^{\mathrm{Bog}}}{\cH_N^{\rm trap}\,\Psi^{\mathrm{Bog}}}= N\cE_{\mathrm{GP}}^\mu(\phi)+\mathcal{O}(\varepsilon^{-s}\sqrt{N})\ \label{est:expected_energy_bogoliubov_states}\ ,
        \end{align}
        where $\cE_{\mathrm{GP}}^\mu(\phi)$ is the Gross--Pitaevskii functional defined by
        \begin{align*}
            \cE_{\mathrm{GP}}^\mu (\phi) = \frac12\intd \n{\varepsilon\grad\phi}^2+2W_{\rm trap}\n{\phi}^2 + 4\pi \asc_0^\mu \n{\phi}^4\dd x\ .
        \end{align*}
\end{lem}

\begin{proof}
The proofs of \eqref{est:expected_number_bogoliubov_states}--\eqref{est:variance_bogoliubov_states} are straightforward. They follow from Lemmas  \ref{lem:estimates_for_creation_annilihation_operators}, \ref{lem:properties_of_weyl_operator}, \ref{lem:bogoliubov_conjugation_operator_inequality}, and the fact that $e^{\cB(k)}$ is unitary. The proof of  \eqref{est:expected_energy_bogoliubov_states} follows from  Section~\ref{section:Bounds on the Fluctuations}.
\end{proof}

\subsection{Mean-field approximation to marginal densities with correction term}
We define the one-particle reduced density operator of a Fock state $\Psi$ via
\begin{align}
    \Gamma_{:1}(x, x') := \frac{\inprod{\Psi}{a^\ast_{x'}a_x\,\Psi}}{\inprod{\Psi}{\cN\,\Psi}}\ ,
\end{align}
which coincides with the definition of $\Gamma_{N:1, t}^\varepsilon$ given by Equation~\eqref{def:one-particle_reduced_density} (with the Fock vector $\Psi_{N, t}$), and its corresponding mass and momentum density functions by
\begin{align*}
    \rho_{:1}(x) :=\, \Gamma_{:1}(x, x) \quad \text{ and } \quad
    \vect{J}^\varepsilon_{:1}(x):=\, \im\diag(\varepsilon\grad_x \Gamma_{:1})(x)\ .
\end{align*}

To approximate the local densities of the many-body dynamics by some effective dynamics, we first introduce the operator with a correction term $$\Gamma^\varepsilon_{\mathrm{cor},t}:=\ket{\phi^\varepsilon_{N, t}}\!\!\bra{\phi^\varepsilon_{N, t}}+\tfrac{1}{N}\sh(k^\varepsilon_{N, t})\,\conj{\sh(k^\varepsilon_{N, t})}\ ,$$
then show that the one-particle density operator $\Gamma_{N:1, t}^\varepsilon$ is well-approximated by $\Gamma^\varepsilon_{\mathrm{cor},t}$ as $N$ tends to infinity. Here is the main proposition.

\begin{prop}\label{prop:mean-field_approximation_of_marginal_density}
    Suppose $\phi_t=\phi_{N, t}^\varepsilon$ is a solution to Equation \eqref{eq:modified_Gross--Pitaevskii} with initial data $\weight{\varepsilon \grad}^4\phi^{\init}\in L^2(\R^3)$ holds uniformly in $\varepsilon$. Suppose $k_{t}=k_{N, t}^\varepsilon$ is given by Equation \eqref{def:pair_excitation_function} with $\ell=\varepsilon^4$ for the necessary regimes. Assume $\Psi_{N, t}$ is a solution to the Cauchy problem~\eqref{def:Fock_space_cauchy_problem} with initial data $\Psi^{\init}$ being a generalized Bogoliubov state defined by Expression~\eqref{eq:Fock_space_cauchy_problem_data} with a normalized state vector $\Xi_0 \in \cF_s$  satisfying the following bounds: there exists $D>0$, independent of $N$ and $\varepsilon$, such that we have the bounds
        \begin{align*}
            \inprod{\Xi_0}{\cN\,\Xi_0},\ \frac{1}{N}\inprod{\Xi_0}{\cN^2\,\Xi_0},\ \inprod{\Xi_0}{\cH_N\,\Xi_0} \le D\ .
        \end{align*}
    Then we have the following estimates
\begin{align}
\Nrm{\diag(\Gamma_{N:1, t}^\varepsilon)-\diag(\Gamma^\varepsilon_{\mathrm{cor},t})}{L^1_x}\leq&\,
C(N,\la,\ve,t)\label{est:mean-field_approximation_of_mass density}\ ,\\
\Nrm{\diag(\varepsilon \grad_{x}\Gamma_{N:1, t}^\varepsilon)-\diag(\varepsilon \grad_{x}\Gamma^\varepsilon_{\mathrm{cor},t})}{L^1_x}\leq&\,
C(N,\la,\ve,t)\label{est:mean-field_approximation_of_derivative_of_one-particle_marginal}\ ,\\
\Nrm{\diag(\varepsilon \grad_{x}\cdot\varepsilon\grad_{x'}\Gamma_{N:1, t}^\varepsilon)-\diag(\varepsilon \grad_{x}\cdot\varepsilon\grad_{x'}\Gamma^\varepsilon_{\mathrm{cor},t})}{L^1_x}\leq&\, C(N,\la,\ve,t)
\label{est:mean-field_approximation_of_energy density}\ ,
\end{align}
where
\begin{align}
C(N,\la,\ve,t)=\frac{C}{\sqrt{N}\varepsilon^6}\exp(C\, \lambda\, \ve^{-M}\lra{t}^{-M})\ ,
\end{align}
for some constants $C$ and $M$.
Furthermore, under the restriction that
\begin{align}
        N\gtrsim_{T_{0}} \exp(\varepsilon^{-100/(1-\upalpha)})\ ,
    \end{align}
    we have
\begin{align}
C(N,\la,\ve,t)\lesssim \lrs{\frac{1}{\ln N}}^{100}.
\end{align}

\end{prop}

\begin{remark}
Actually, the convergence rate $C(N,\la,\ve,t)$ can be controlled by
\begin{align}
C(N,\la,\ve,t)\lesssim \frac{1}{N^{s}},\quad s\in(0,\tfrac{1}{2})\ .
\end{align}
We chose a crude estimate of $\ln N$ to avoid heavy notations and to make the entire proof more concise.
\end{remark}

\begin{proof}[Proof of Proposition~\ref{prop:mean-field_approximation_of_marginal_density}]
It suffices to consider the most difficult Estimate~\eqref{est:mean-field_approximation_of_energy density}, since Estimates~\eqref{est:mean-field_approximation_of_mass density}--\eqref{est:mean-field_approximation_of_derivative_of_one-particle_marginal} can be handled in a similar fashion. Using Lemma \ref{lem:properties_of_weyl_operator}, we obtain the expansion
\begin{align}\label{eq:one-particle_marginal_computation}
    \inprod{\Psi_{N, t}}{a^\ast_{x'}a_x\,\Psi_{N, t}}=&N\, \phi_{N, t}^\varepsilon(x)\conj{ \phi_{N, t}^\varepsilon(x')}\\
   & +\sqrt{N}\,\phi_{N, t}^\varepsilon(x)\inprod{\Psi_{N, t}}{e^{-\sqrt{N}\cA(\phi_{N, t}^\varepsilon)} a^\ast_{x'}\,e^{\sqrt{N}\cA(\phi_{N, t}^\varepsilon)}\,\Psi_{N, t}}\notag\\
    &+ \sqrt{N}\, \conj{\phi_{N, t}^\varepsilon}(x')\,\inprod{\Psi_{N, t}}{e^{-\sqrt{N}\cA(\phi_{N, t}^\varepsilon)}a_{x}\,e^{\sqrt{N}\cA(\phi_{N, t}^\varepsilon)}\,\Psi_{N, t}}\notag\\
   & +\inprod{\Psi_{N, t}}{e^{-\sqrt{N}\cA(\phi_{N, t}^\varepsilon)}a^\ast_{x'} a_{x}\,e^{\sqrt{N}\cA(\phi_{N, t}^\varepsilon)}\,\Psi_{N, t}}\ .\notag
\end{align}
Furthermore, by Identity \eqref{eq:one-particle_marginal_computation} and Identities \eqref{eq:bogoliubov_conjugation_operator-value_distribution_id}, we could write
\begin{align}\label{eq:one-particle_marginal_computation2}
    \inprod{\Psi_{N, t}}{a^\ast_{x'}a_x\,\Psi_{N, t}}=&N\Gamma^\varepsilon_{\mathrm{cor},t}(x, x')+\sqrt{N}\,\phi_{N, t}^\varepsilon(x)\inprod{\Psi_{\mathrm{fluc}, t}}{\(a^\ast(c_{x'})+a(\sh_{x'})\)\,\Psi_{\mathrm{fluc}, t}}\\
    &+ \sqrt{N}\, \conj{\phi_{N, t}^\varepsilon}(x')\,\inprod{\Psi_{\mathrm{fluc}, t}}{\(a(\ch_{x})+a^\ast(\sh_x)\)\,\Psi_{\mathrm{fluc}, t}}\notag\\
    &+\inprod{\Psi_{\mathrm{fluc}, t}}{\(a^\ast(c_{x'})a(\ch_{x}) + a(\sh_{x'})a(\ch_{x}) \)\,\Psi_{\mathrm{fluc}, t}} \notag\\
    &+\inprod{\Psi_{\mathrm{fluc}, t}}{\(a^\ast(c_{x'})a^\ast(\sh_x) + a^\ast(\sh_x) a(\sh_{x'})\)\,\Psi_{\mathrm{fluc}, t}}\ ,\notag
\end{align}
where  $\Psi_{\mathrm{fluc}, t}$ is the fluctuation vector defined by Expression~\eqref{def:fluctuation_state}.

From Identity~\eqref{eq:one-particle_marginal_computation2} and Equation~\eqref{est:expected_number_bogoliubov_states}, we immediately see that
\begin{align}\label{est:difference_of_kinetic_energy}
   &\n{\diag(\varepsilon \grad_{x}\cdot\varepsilon\grad_{x'}\Gamma_{N:1, t}^\varepsilon)(x)-\diag(\varepsilon \grad_{x}\cdot\varepsilon\grad_{x'}\Gamma^\varepsilon_{\mathrm{cor},t})(x)}\\
    \le& \frac{C}{\sqrt{N}}\n{\diag(\varepsilon \grad_{x}\cdot\varepsilon\grad_{x'}\Gamma^\varepsilon_{\mathrm{cor},t})(x)}\notag\\
    &+\frac{C}{\sqrt{N}}\n{\varepsilon\grad\phi_{N, t}^\varepsilon(x)}\Big(\Nrm{\varepsilon \grad_{x}a_{x}\,\Psi_{\mathrm{fluc}, t}}{}+\Nrm{a(\varepsilon \grad_{x}p_{x})\,\Psi_{\mathrm{fluc}, t}}{}\notag\\
   & +\Nrm{a(\varepsilon \grad_{x}k_{x})\,\Psi_{\mathrm{fluc}, t}}{}+\Nrm{a(\varepsilon \grad_{x}r_{x})\,\Psi_{\mathrm{fluc}, t}}{}\Big)+\frac{C}{N}\sfD(x)\notag
\end{align}
where
\begin{align*}
    \sfD(x)=&\, \Nrm{\varepsilon\grad_x a_x \Psi_{\mathrm{fluc}, t}}{}^2+\Nrm{a(\varepsilon\grad_xk_x) \Psi_{\mathrm{fluc}, t}}{}^2\\
    &\, +\Nrm{a(\varepsilon\grad_xr_x) \Psi_{\mathrm{fluc}, t}}{}^2+\Nrm{a(\varepsilon\grad_xp_x) \Psi_{\mathrm{fluc}, t}}{}^2 \\
    &\, +2\re\inprod{\Psi_{\mathrm{fluc}, t}}{\(a^\ast(\varepsilon \grad_{x}p_{x})\cdot\varepsilon \grad_{x}a_x+a(\varepsilon \grad_{x}k_{x})\cdot\varepsilon \grad_{x}a_x\)\,\Psi_{\mathrm{fluc}, t}} \\
    &\, +2\re\inprod{\Psi_{\mathrm{fluc}, t}}{\(a(\varepsilon \grad_{x}r_{x})\cdot\varepsilon \grad_{x}a_x+a^\ast(\varepsilon \grad_{x}k_x)\cdot a^\ast(\varepsilon \grad_{x}p_{x})\)\,\Psi_{\mathrm{fluc}, t}}\\
    &\, +2\re\inprod{\Psi_{\mathrm{fluc}, t}}{\(a^\ast(\varepsilon \grad_{x}r_x)\cdot a^\ast(\varepsilon \grad_{x} p_{x})+a^\ast(\varepsilon \grad_{x}k_x)\cdot a(\varepsilon \grad_{x}r_{x})\)\,\Psi_{\mathrm{fluc}, t}}\ .
\end{align*}
Then, by Lemmas~\ref{lem:estimates_for_creation_annilihation_operators} and \ref{lem:quadratic_operator_bounds_with_derivative2}, and the Cauchy--Schwarz inequality, we have that
\begin{align*}
    &\Nrm{\diag(\varepsilon \grad_{x}\cdot\varepsilon\grad_{x'}\Gamma_{N:1, t}^\varepsilon)-\diag(\varepsilon \grad_{x}\cdot\varepsilon\grad_{x'}\Gamma^\varepsilon_{\mathrm{cor},t})}{L^1_x}\\
    \lesssim& \frac{1}{\sqrt{N}}\Tr{\varepsilon \grad_{x}\cdot\varepsilon\grad_{x'}\Gamma_{\mathrm{cor},t}^\varepsilon} \\
     &+ \frac{1}{\sqrt{N}}\Nrm{\varepsilon\grad \phi}{L^2_x}\( \Nrm{k(-\varepsilon^2\lapl)\conj{k}}{\mathrm{HS}}+\Nrm{\varepsilon\grad p}{\mathrm{HS}}^2+\Nrm{\varepsilon\grad r}{\mathrm{HS}}^2\)^\frac12\norm{\cN^\frac12\, \Psi_{\mathrm{fluc}, t}}\\
    & + \frac{1}{\sqrt{N}}\Nrm{\varepsilon\grad \phi}{L^2_x}\norm{\cK^\frac12\, \Psi_{\mathrm{fluc}, t}}+ \frac{1}{N}\Nrm{\sfD}{L^1_x}
\end{align*}
where
\begin{align*}
    \Nrm{\sfD}{L^1_x} \lesssim&\, \norm{\cK^\frac12\, \Psi_{\mathrm{fluc}, t}}^2
   +\( \Nrm{k(-\varepsilon^2\lapl)\conj{k}}{\mathrm{HS}}+\Nrm{\varepsilon\grad p}{\mathrm{HS}}^2+\Nrm{\varepsilon\grad r}{\mathrm{HS}}^2\)\norm{\(\cN+1\)^\frac12\, \Psi_{\mathrm{fluc}, t}}^2\ .
\end{align*}
The estimate for  $\Tr{\varepsilon \grad_{x}\cdot\varepsilon\grad_{x'}\Gamma_{\mathrm{cor},t}^\varepsilon}$ follows from either Lemma~\ref{lem:pair_excitation_estimates}, or from Lemma~\ref{lem:kinetic_energy_density_approx} and the fact that the two-particle reduced density operator of the many-body system is bounded. Then, by Lemma~\ref{lem:kinetic_energy_density_approx}, Proposition~\ref{prop:bounds_on_growth_of_fluctuation}, and Lemma~\ref{lem:pair_excitation_estimates}, we complete the proof of Inequality~\eqref{est:kinetic_energy_many-body_mean-field_approximation}.
\end{proof}

\subsection{Mean-field approximation of mass and momentum densities}
Define following mass and momentum density functions
\begin{align}
\rho_{N:1,t}^\varepsilon(x)=&\, \diag(\Gamma_{N:1, t}^\varepsilon)(x)\ ,  &\vect{J}^\varepsilon_{N:1}(x):=&\,  \im\diag(\varepsilon\grad_x \Gamma^\varepsilon_{N:1})(x)\ ,\\
\rho_{\Gamma_{\mathrm{cor},t}}^{\varepsilon}(x)=&\, \diag(\Gamma^\varepsilon_{\mathrm{cor},t})(x)\ , &  \vect{J}^\varepsilon_{\Gamma_{\mathrm{cor}}, t}(x):=&\, \im\diag(\varepsilon\grad_x \Gamma^\varepsilon_{\mathrm{cor},t})(x)\ ,\\
\rho_{t}^{\varepsilon}(x)=&\, |\phi_{N,t}^{\ve}(x)|^{2}\ , & \vect{J}^\varepsilon_t(x):=&\, \im(\varepsilon\grad_x\phi^\varepsilon_{N, t}(x)\conj{\phi^\varepsilon_{N, t}(x)})\ .
\end{align}

Next, we prove the mean-field approximation of the mass and momentum densities. We start with the following lemma.

\begin{lem}\label{lem:approximation_by_condensate_density_and_momentum}
    Suppose $\phi_t=\phi_{N, t}^\varepsilon$ is a solution to Equation \eqref{eq:modified_Gross--Pitaevskii} with initial data $\weight{\varepsilon \grad}\phi^{\init}\in L^2(\R^3)$ holds uniformly in $\varepsilon$. Suppose $k_{t}=k_{N, t}^\varepsilon$ is given by Equation \eqref{def:pair_excitation_function} with $\ell=\varepsilon^4$ for the necessary regimes.
   There holds that
    \begin{align*}
        &\sup_{t\in [0, \infty)}\Nrm{\rho^\varepsilon_{\Gamma_{\mathrm{cor}}, t}-\rho^\varepsilon_{t}}{L_{x}^{1}}\lesssim \frac{1}{N},\quad
        \sup_{t\in [0, \infty)}\Nrm{\vect{J}^\varepsilon_{\Gamma_{\mathrm{cor}}, t}-\vect{J}^\varepsilon_{t}}{L^1_x}\lesssim \frac{\la^{\frac{1}{2}}}{\sqrt{N}}\ .
    \end{align*}
\end{lem}
\begin{proof}
First for the mass density, we have
    \begin{align*}
        \Nrm{\rho^\varepsilon_{\Gamma_{\mathrm{cor}}, t}-\rho^\varepsilon_{t}}{L_{x}^{1}}
        =&\, \frac{1}{N}\Nrm{\sh(k^\varepsilon_{N, t})(x, x')}{L_{x}^{2}L_{x'}^{2}}^2
        \le\, \frac{1}{N}\Nrm{\sh(k^\varepsilon_{N, t})}{\mathrm{HS}}^{2}\ .
    \end{align*}
    Hence, by \eqref{equ:estimate on u,L2} in Lemma~\ref{lem:pair_excitation_estimates}, we complete the proof of the estimate on the mass density.

    For the momentum density, we have that
    \begin{align*}
        \Nrm{\vect{J}^\varepsilon_{\Gamma_{\mathrm{cor}}, t}-\vect{J}^\varepsilon_{t}}{L^1_x}\le \frac{1}{\sqrt{N}}\Nrm{\sh(k^\varepsilon_{N, t})}{\mathrm{HS}}\frac{1}{\sqrt{N}}\Nrm{\varepsilon\grad\sh(k^\varepsilon_{N, t})}{\mathrm{HS}}\leq \frac{C\la^{\frac{1}{2}}}{\sqrt{N}}\ .
    \end{align*}
    Again, by \eqref{equ:estimate on k,H1} and \eqref{equ:estimate on r,H1} in Lemma~\ref{lem:pair_excitation_estimates}, we obtain the desired result.
\end{proof}

\begin{prop}\label{prop:approximation,density_and_momentum}
    Fix a time interval $[0, T_0]$. Assuming the same hypotheses as in Proposition~\ref{prop:mean-field_approximation_of_marginal_density}, we have the estimate
    \begin{align}
        \sup_{t\in [0, T_0]}\Nrm{\rho^\varepsilon_{N:1, t}-\rho^\varepsilon_{t}}{L^1_x\cap L^2_{x}}\lesssim_{T_0}&\, \frac{1}{\ln N}\ ,\\
        \sup_{t\in [0, T_0]}\Nrm{\vect{J}^\varepsilon_{N:1, t}-\vect{J}^\varepsilon_{t}}{L^1_x}\lesssim_{T_0}&\, \frac{1}{\ln N}\ .
    \end{align}
\end{prop}
\begin{proof}

The $L_{x}^{1}$ norm estimate follows from a triangle inequality argument, along with applying Proposition \ref{prop:mean-field_approximation_of_marginal_density} and Lemma \ref{lem:approximation_by_condensate_density_and_momentum}.

The $L_{x}^{2}$ estimate for the mass density can be obtained by using the interpolation inequality that
$\Nrm{f}{L_{x}^{2}}\leq \Nrm{f}{L_{x}^{1}}^{\frac{1}{4}}\Nrm{f}{L_{x}^{3}}^{\frac{3}{4}}$. Here, the bound for the $L_{x}^{3}$ norm of
$\rho^\varepsilon_{N:1, t}$ and $\rho^\varepsilon_{t}$ follows by Sobolev's inequality and the finite energy of both the many-body dynamics and the modified Gross--Pitaevskii dynamics, i.e., applying H\"older, Minkowski, Sobolev's inequalities, and Lemma~\ref{lem:bogoliubov_states}, we get that
\begin{align*}
    \Nrm{\rho_{N:1, t}^\varepsilon}{L^3_x}\le&\, \frac{1}{\inprod{\Psi_{N, t}}{\cN\, \Psi_{N, t}}}\Nrm{\Nrm{a_x\, \Psi_{N, t}}{}}{L^6_x}\Nrm{\Nrm{a_x\, \Psi_{N, t}}{}}{L^6_x}\\
    \le&\, \frac{1}{\varepsilon^2\inprod{\Psi_{N, t}}{\cN\, \Psi_{N, t}}}\Nrm{\Nrm{\varepsilon \grad_x a_x\, \Psi_{N, t}}{}}{L^2_x}^2
    \le\, \frac{\inprod{\Psi_{N, t}}{\cH_N\, \Psi_{N, t}}}{\varepsilon^2 \inprod{\Psi_{N, t}}{\cN\, \Psi_{N, t}}}\lesssim \varepsilon^{-2}\ .
\end{align*}
The argument for $\rho^\varepsilon_t$ is similar.  Hence, we complete the proof of the convergence of mass and momentum densities.
\end{proof}

\subsection{Mean-field approximation of energy densities}\label{subsection:mean-field approximation of energies}
Let $\Psi$ be a Fock state such that $\inprod{\Psi}{\cN\,\Psi}<\infty$. Suppose $k$ is an integer less than or equal to $\inprod{\Psi}{\cN\,\Psi}$, then we define the $k$-particle reduced density operator associated to a Fock state $\Psi$ by the kernel
\begin{align}
    \Gamma_{:k}(x_1, \ldots, x_k; x_1',\ldots, x_k'):= \frac{\inprod{\Psi}{\prod^k_{\jj=1} a^\ast_{x_\jj'}\prod^k_{\ii=1}a_{x_\ii}\, \Psi}}{\inprod{\Psi}{\cN(\cN-1)\cdots (\cN-k+1)\,\Psi}}\ .
\end{align}
Using the above definition, we define the many-body energy-per-particle densities associated to $\Psi_{N, t}$ as follows
\begin{align}
    E^\varepsilon_{N}(x):=&\, E^\varepsilon_{N,\, \mathrm{kin}.}(x)+E^\varepsilon_{N,\, \mathrm{int}.}(x)\ , \\
    E^\varepsilon_{N,\, \mathrm{kin}.}(x):=&\, \tfrac12\diag\(\varepsilon\grad_{x}\cdot\varepsilon\grad_{x'}\Gamma_{N:1}^\varepsilon\)(x)\ , \\
    E^\varepsilon_{N,\, \mathrm{int}.}(x):=&\, \textstyle \frac{\inprod{\Psi_N}{\cN(\cN-1)\,\Psi_N}}{2N\inprod{\Psi_N}{\cN\,\Psi_N}}\intd V^\varepsilon_{N}(x-y)\Gamma^{\varepsilon}_{N:2}(x, y; x, y)\dd y\ .
\end{align}
Here, $E^\varepsilon_{N,\, \mathrm{kin}.}(x)$ is the called the kinetic energy density and $E^\varepsilon_{N,\, \mathrm{int}.}(x)$ is called the interaction energy density.

The main result of this section is given by the following proposition.
\begin{prop}\label{prop:many-body_energies_mean-field_approximation}
    Fix a time interval $[0, T_0]$. Let us introduce the following characteristic length
    \begin{align}\label{equ:characteristic length}
        \mathfrak{b}_0^{\mu}:= \frac{1}{4\pi}\intd \n{\grad f^\mu_0}^2\dd x = \asc_0^{\mu} - \frac{1}{4\pi\, \mu}\intd v (f^\mu_0)^2\dd x\ .
    \end{align}
    Assuming the same hypotheses as in Proposition~\ref{prop:mean-field_approximation_of_marginal_density} for the Fock state $\Psi_{N, t}$.
   Then we have the following uniform estimates
    \begin{align}\label{est:kinetic_energy_many-body_mean-field_approximation}
        \sup_{t \in [0, T_0]}\Nrm{E^\varepsilon_{N,\, \mathrm{kin}., t}-\tfrac12\n{\varepsilon\grad\phi^\varepsilon_{N, t}}^2-4\pi\mu \la\, \mathfrak{b}_0^{\mu}\n{\phi^\varepsilon_{N,t}}^4}{L^1_x}\lesssim_{T_0}&\, \frac{1}{\ln N}
    \end{align}
    \begin{align}\label{est:interaction_energy_many-body_mean-field_approximation}
        \sup_{t \in [0, T_0]}\Nrm{E^\varepsilon_{N,\, \mathrm{int}., t}-4\pi\mu \la\, \(\mathfrak{a}_0^{\mu}-\mathfrak{b}_0^{\mu}\)\n{\phi^\varepsilon_{N,t}}^4}{L^1_x}\lesssim_{T_0}&\, \frac{1}{\ln N}\ .
    \end{align}

\end{prop}

\begin{remark}
The different contributions to the Gross--Pitaevskii energy from microscopic dynamics were first noted in \cite{cherny2002kinetic}. Subsequently, a rigorous proof starting from the many-body system was provided in \cite{lieb2002proof}. Our result goes one step further by discussing the contributions at the level of local energy densities.
\end{remark}

To prove the proposition, we shall start by proving the following series of lemmas.

\begin{lem}\label{lem:kinetic_energy_density_approx}
    Assume the same hypotheses as in Proposition~$\ref{prop:many-body_energies_mean-field_approximation}$.
    We have
    \begin{align}
        \sup_{t\in[0, \infty)}\Nrm{\frac{1}{N}\Nrm{\varepsilon\grad_x \sh(k_{N, t}^\varepsilon)(x, \cdot)}{L^2_{x'}}^2-4\pi\mu \la\, \mathfrak{b}_0^{\mu}\n{\phi^\varepsilon_{N,t}}^4}{L^1_x}
         \lesssim \frac{\lambda^\frac12}{N^\frac12}\ .
    \end{align}
\end{lem}

\begin{proof}
    Write $\sh=k+r$, then we have
    \begin{subequations}
        \begin{align}
            \Nrm{\varepsilon\grad_x \sh(x, \cdot)}{L^2_{x'}}^2 =&\, \Nrm{\varepsilon\grad_x k(x, \cdot)}{L^2_{x'}}^2 \label{def:approximation_kinetic_energy_main-term}\\
            &\, +\Nrm{\varepsilon\grad_x r(x, \cdot)}{L^2_{x'}}^2+2\re\inprod{\varepsilon\grad_x k(x, \cdot)}{\varepsilon\grad_x r(x, \cdot)}_{L^2}\ . \label{def:approximation_kinetic_energy_error-term}
        \end{align}
    \end{subequations}
    For the latter two terms, by the Cauchy--Schwarz inequality and Lemma~\ref{lem:pair_excitation_estimates}, we have the coarse bound
    \begin{align*}
       \frac{1}{N} \Nrm{\eqref{def:approximation_kinetic_energy_error-term}}{L^1_x}\le \frac{1}{N}\Nrm{\varepsilon\grad r}{\mathrm{HS}}^2 + \frac{2}{N}\Nrm{\varepsilon\grad k}{\mathrm{HS}}\Nrm{\varepsilon\grad r}{\mathrm{HS}} \lesssim \frac{\lambda^\frac12}{\sqrt{N}}\ .
    \end{align*}

    For Term~\eqref{def:approximation_kinetic_energy_main-term}, we further split the term
    \begin{subequations}
        \begin{align*}
            \frac{1}{N}\Nrm{\varepsilon\grad_x k(x, \cdot)}{L^2_{x'}}^2 =I_{\rm Main}+I_{\rm Err, 1}+I_{\rm Err, 1}
    \end{align*}
where
\begin{align*}
            I_{\rm Main}=& \frac{1}{N}\intd |N\ve \nabla_{x}w_{N,l}^{\ve}(x-x')|^{2}\n{\phi_{N}^{\ve}(x)}^{2}\n{\phi_{N}^{\ve}(x')}^{2}\dd x', \\
           I_{\rm Err,1}=& \frac{1}{N}\intd |N w_{N,l}^{\ve}(x-x')|^{2}\n{\ve \nabla \phi_{N}^{\ve}(x)}^{2}\n{\phi_{N}^{\ve}(x')}^{2}\dd x',\\
            I_{\rm Err,2}=&\frac{1}{N}\frac{N^2}{2} \intd \ve\grad_x [w_{N,l}^{\ve}(x-x')]^2 \n{\phi_{N}^{\ve}(x')}^2\dd x'\cdot \ve\nabla \n{\phi_{N}^{\ve}(x)}^2.
\end{align*}
For the error term $I_{\rm Err}$,
by the pointwise and gradient estimates \eqref{est:correlation_w_pointwise_bound}--\eqref{est:correlation_grad_w_pointwise_bound}, Hardy's inequality, we have that
\begin{align*}
|I_{\rm Err, 1}|\lesssim\frac{1}{N\varepsilon^6}\intdd \frac{\n{\ve \nabla \phi_{N}^{\ve}(x)}^{2}\n{\phi_{N}^{\ve}(x')}^{2}}{\n{x-x'}^2}\dd x\d x' \lesssim \frac{1}{N\varepsilon^6}.
\end{align*}
In a similar way as above, we also have
\begin{align*}
|I_{\rm Err, 1}|\lesssim \frac{1}{\sqrt{N}\varepsilon^6}.
\end{align*}
    \end{subequations}

    For $I_{\rm Main}$, noting that $-\nabla_{x}w_{N,l}^{\ve}(x)=N^{\be}\ve^{2\ka}(\nabla f_{L}^{\mu})(N^{\be}\ve^{2\ka}x)$
     where $f_{N,l}^{\mu}(x)=f_{L}^{\mu}(N^{\be}\ve^{2\ka}x)$, we apply the triangle inequality to obtain
     \begin{align}\label{equ:main term analysis,kinetic}
     &\Nrm{I_{\rm Main}-4\pi\mu \la\, \mathfrak{b}_0^{\mu}\n{\phi^\varepsilon_{N,t}}^4}{L^1_x}\leq A_{1}+A_{2},
     \end{align}
     where
\begin{align*}
A_{1}=&\mu \la\Nrm{\((N^{\be}\ve^{2\ka})^3\n{\nabla f_{L}^{\mu}(N^{\be}\ve^{2\ka}x)}^2-\delta_0\textstyle\(\intd |\grad f^\mu_{L}|^2\dd z\)\)\ast \rho^\varepsilon_t}{L^\frac32_x}\Nrm{\rho^\varepsilon_t}{L^3_x},\\
A_{2}=&\mu \la\Nrm{\(\delta_0\textstyle\(\intd |\grad f^\mu_{L}|^2\dd z\)-\delta_0\textstyle\(\intd |\grad f^\mu_{0}|^2\dd z\)\)\ast \rho^\varepsilon_t}{L^\frac32_x}\Nrm{\rho^\varepsilon_t}{L^3_x}.
\end{align*}
By Lemma~\ref{lem:quantative estimate for identity approximation}, we get that
    \begin{align*}
A_{1} \lesssim& \frac{\mu \la}{N\varepsilon^{6}}\norm{\weight{x}^\frac12\grad f_{L}^{\mu}}_{L^{2}_x}^2\Nrm{\weight{\varepsilon\grad}\phi^\varepsilon_{N, t}}{L^2_x}^4\lesssim  \frac{\lambda^2}{N\varepsilon^{6}}\ ,
    \end{align*}
    where the last bound follows by Inequality~\eqref{est:uniform_grad_w_bound_neumann}.
Finally, integrating by parts and applying \eqref{eq:groundstate_asymptotic_expansion_neumann}--\eqref{equ:Dirichlet,Neumann} in Lemma~\ref{lem:appendix_version_correlation_structures}, we have that
    \begin{align*}
       A_{2}\lesssim \Nrm{\rho_{t}^{\ve}}{L_{x}^{3}}\n{\intd \n{\grad f^\mu_0}^2- |\grad f^\mu_{L}|^2\dd x} \lesssim \frac{\ve^{-2}}{4\pi\mu} \n{\int_{\R^{3}} v(f^\mu_0)^2-(v-E_{\mathrm{gs}})(f^\mu_L)^2\dd x}\lesssim\frac{1}{L\ve^{2}}\ .
    \end{align*}
where $L=N^{\be}\ve^{2\ka}l\gg N^{\frac{1}{2}}$.

\end{proof}

\begin{lem}\label{lem:second_marginal_expansion}
    Consider the following two-particle operator $\Gamma^\varepsilon_{N:2, t}$ with integral kernel
    \begin{multline}
        \Gamma^\varepsilon_{N:2, t}(x_1, x_2; x_1', x_2'):= \Lambda^\varepsilon_t(x_1, x_2)\conj{\Lambda^\varepsilon_t(x_1', x_2')}+
        \Gamma^\varepsilon_{\mathrm{cor},t}(x_1, x_1')\Gamma^\varepsilon_{\mathrm{cor},t}(x_2, x_2')\\
        +\Gamma^\varepsilon_{\mathrm{cor},t}(x_1, x_2')\Gamma^\varepsilon_{\mathrm{cor},t}(x_2, x_1') - 2\, \phi_{N, t}^\varepsilon(x_1)\phi_{N, t}^\varepsilon(x_2)\conj{\phi_{N, t}^\varepsilon(x_1')\phi_{N, t}^\varepsilon(x_2)}
    \end{multline}
    where $\Lambda^\varepsilon_t(x_1, x_2)$ is defined by
    \begin{align}
        \Lambda^\varepsilon_t(x_1, x_2):=\phi_{N, t}^\varepsilon(x_1)\phi_{N, t}^\varepsilon(x_2)+\tfrac{1}{2N}\sh(2\,k_{N, t}^\varepsilon)(x_1, x_2)\ .
    \end{align}
    Then we have the following identity
    \begin{align*}
        \inprod{\Psi_{N, t}}{a^\ast_x a^\ast_y a_y a_x\, \Psi_{N, t}}-N^2\Gamma^\varepsilon_{N:2, t}(x, y; x, y)=\sfR_1(x, y)+\sfR_2(x, y)+\sfR_3(x, y)
    \end{align*}
    where
    \begin{align}
        \sfR_1(x, y)=&\, 2\sqrt{N}\re\(\phi_{N, t}^\varepsilon(x)\inprod{\Psi_{\mathrm{fluc}, t}}{\nor{b^\ast_{x}b^\ast_yb_y} \Psi_{\mathrm{fluc}, t}}\)\\
        &\, +2\sqrt{N}\re\(\phi_{N, t}^\varepsilon(y)\inprod{\Psi_{\mathrm{fluc}, t}}{\nor{b^\ast_{y}b^\ast_xb_x} \Psi_{\mathrm{fluc}, t}}\) \notag\\
        \sfR_2(x, y)=&\, 2N\,\re\(\Lambda_{t}^\varepsilon(x, y)\inprod{\Psi_{\mathrm{fluc}, t}}{\nor{b^\ast_{y}b^\ast_{x}}\Psi_{\mathrm{fluc}, t}}\)\\
        &+2N\,\re\(\Gamma^\varepsilon_{\mathrm{cor},t}(x, y)\inprod{\Psi_{\mathrm{fluc}, t}}{\nor{b^\ast_{x}b_{y}}\Psi_{\mathrm{fluc}, t}}\)\notag\\
        &+N\,\Gamma^\varepsilon_{\mathrm{cor},t}(x, x)\inprod{\Psi_{\mathrm{fluc}, t}}{\nor{b^\ast_{y}b_{y}}\Psi_{\mathrm{fluc}, t}}\notag\\
        &+N\,\Gamma^\varepsilon_{\mathrm{cor},t}(y, y)\inprod{\Psi_{\mathrm{fluc}, t}}{\nor{b^\ast_{x}b_{x}}\Psi_{\mathrm{fluc}, t}}\notag\\
        \sfR_3(x, y)=&\, 2N^\frac{3}{2} \re\(f_{\mathrm{HFB}}(x, y)\inprod{\Psi_{\mathrm{fluc}, t}}{b^\ast_y\,\Psi_{\mathrm{fluc}, t}}\)\\
        &+ 2N^\frac{3}{2} \re\(f_{\mathrm{HFB}}(y, x)\inprod{\Psi_{\mathrm{fluc}, t}}{b^\ast_x\,\Psi_{\mathrm{fluc}, t}}\) \notag
    \end{align}
    with
    \begin{align}\label{def:f_HFB}
        f_{\mathrm{HFB}}(x, y)
        = \conj{\phi_{N, t}^\varepsilon(x)}\Lambda_{t}^\varepsilon(x, y)+\tfrac{1}{2N}\phi_{N, t}^\varepsilon(y) p_{2, t}^\varepsilon(x, x)+\tfrac{1}{2N}\phi_{N, t}^\varepsilon(x) p_{2, t}^\varepsilon(x, y)\ .
    \end{align}
    Here, $b^\sharp_x$ are defined in Equations~\eqref{eq:bogoliubov_conjugation_operator-value_distribution_id} and $\nor{\cdot}$ denotes the normal ordering of $a, a^\ast$.
\end{lem}

\begin{proof}
    See \cite[Section 8]{grillakis2017pair} for the proof.
\end{proof}

\begin{remark}
    The normal ordering in Lemma~\ref{lem:second_marginal_expansion} could be explicitly computed. In fact, the exact form of the normal ordering are given in Proposition~\ref{prop:fluctuation_Hamiltonian} in the computation of $e^{\cB(k_t)}e^{\sqrt{N}\cA(\phi_t)}\cV e^{-\sqrt{N}\cA(\phi_t)}e^{-\cB(k_t)}$. More precisely, the explicit form of $\frac{1}{N^2}\intd V^\varepsilon_N(x-y)\sfR_1(x, y)\dd y$ is given by $\frac{1}{N}\cC$ in Expression~\eqref{def:normal_ordered_cubic_terms} (without the integration in $x$).
\end{remark}

\begin{lem}\label{lem:interaction_energy_density_approximation}
    Assuming the same hypotheses as in Proposition~\ref{prop:many-body_energies_mean-field_approximation}.
    Write
    \begin{align*}
        \tr_2\(V^\varepsilon_{N, 12}\Gamma^\varepsilon_{N:2, t}\)(x, x'):=\intd V^\varepsilon_{N}(x-y)\Gamma^{\varepsilon}_{N:2, t}(x, y; x', y)\dd y\ .
    \end{align*}
    Then we have the estimates
    \begin{align}
        \sup_{t\in[0, \infty)}\Nrm{\diag\(\tr_2\(V^\varepsilon_{N, 12}\Gamma^\varepsilon_{N:2, t}\)\)-4\pi\mu\la\(\asc_0^{\mu}-\mathfrak{b}_{0}^{\mu}\)\(\rho^\varepsilon_{t}\)^2}{L^1_x} \le \frac{C\lambda^\frac52}{N^\frac12\varepsilon^3}\ .
    \end{align}
\end{lem}

\begin{proof}
    By Lemma \ref{lem:second_marginal_expansion}, we have that
        \begin{align*}
            \diag\(\tr_2\(V^\varepsilon_{N, 12}\Gamma^\varepsilon_{N:2, t}\)\)(x)&=\, \intd V^\varepsilon_{N}(x-y)\Gamma^{\varepsilon}_{N:2, t}(x, y; x, y)\dd y=I_{1}+I_{2}+I_{3},
            \end{align*}
where
    \begin{align}
            I_{1}=& \intd V^\varepsilon_{N}(x-y)\n{\Lambda_t^\varepsilon(x, y)}^2\dd y,\label{def:approximation_potential_energy_lambda-term}\\
            I_{2}=& \intd V^\varepsilon_{N}(x-y)\(\rho^\varepsilon_{\Gamma_{\mathrm{cor}}, t}(x)\rho^\varepsilon_{\Gamma_{\mathrm{cor}}, t}(y)+\n{\Gamma_{\mathrm{cor},t}^\varepsilon(x, y)}^2\)\dd y,\label{def:approximation_potential_energy_gamma-term}\\
            I_{3}=&-2\intd V^\varepsilon_{N}(x-y)\n{\phi_{t, N}^\varepsilon(x)}^2\n{\phi_{t, N}^\varepsilon(y)}^2\dd  y \label{def:approximation_potential_energy_phi-term}\ .
        \end{align}

    For $I_{1}$, by Definition~\eqref{def:pair_excitation_function} and Identity~\eqref{eq:hyperbolic_trig_identities}, we have that
    \begin{align*}
        I_{1}=&\, \intd V^\varepsilon_{N}(x-y)\n{f^\varepsilon_{N, l}(x-y)\phi^\varepsilon_{N, t}(x)\phi^\varepsilon_{N, t}(y)+\tfrac{1}{2N}r(2\,k_{N, t}^\varepsilon)(x, y)}^2\dd y\\
        =&\, \intd V^\varepsilon_{N}(x-y)f^\varepsilon_{N, l}(x-y)^2\n{\phi^\varepsilon_{N, t}(x)}\n{\phi^\varepsilon_{N, t}(y)}^2\dd y\\
        &\, +\frac{1}{N}\re \intd V^\varepsilon_{N}(x-y) f^\varepsilon_{N, l}(x-y)\phi^\varepsilon_{N, t}(x)\phi^\varepsilon_{N, t}(y)\conj{r_{2, t}^{\varepsilon}(x, y)}\dd y\\
        &\, +\frac{1}{4N^2}\intd V^\varepsilon_{N}(x-y)\n{r_{2, t}^{\varepsilon}(x, y)}^2\dd y\\
        =:&I_{11}+I_{12}+I_{13}.
    \end{align*}
    By applying the similar approach as detailed in Estimate \eqref{equ:main term analysis,kinetic}, we also have
    \begin{align*}
    \Nrm{I_{11}-4\pi\mu\la\(\asc_0^{\mu}-\mathfrak{b}_{0}^{\mu}\)\(\rho^\varepsilon_{t}\)^2}{L_{x}^{1}}\lesssim \frac{1}{N\ve^{6}}+\frac{1}{L}.
    \end{align*}

   For $I_{12}$, by Part~\eqref{part:pair_excitation_pointwise_est_for_p} of \ref{lem:pair_excitation_estimates}, Lemmas~\ref{lem:phi4_bound}, and the Cauchy--Schwarz inequality, we have
    \begin{align*}
        \Nrm{I_{12}}{L_{x}^{1}} \le \frac{C}{N\varepsilon^{6}}\Nrm{V^\varepsilon_{N} f^\varepsilon_{N, l}}{L^1_x}\Nrm{\rho^\varepsilon_{t}}{L^2_x}^2\le \frac{C}{N\varepsilon^{6}} .
    \end{align*}

    For $I_{13}$, by Part~\eqref{part:pair_excitation_pointwise_est_for_p} of Lemma~\ref{lem:pair_excitation_estimates} and Lemma~\ref{lem:phi4_bound}, we get
    \begin{align*}
       \Nrm{I_{13}}{L_{x}^{1}}\le& \frac{C}{N^2}\Nrm{V^\varepsilon_{N}}{L^1_x}\Nrm{r_{2, t}^{\varepsilon}(x, x+z)}{L^\infty_z L^2_{x}}^2
       \le\, \frac{C}{N^2\varepsilon^{12}}\Nrm{V^\varepsilon_{N}}{L^1_x}\Nrm{\rho^\varepsilon_{t}}{L^2_x}^2\le \frac{C}{N^2\varepsilon^{12}} .
    \end{align*}

    Now, for $I_{2}$ and $I_{3}$, we have that
    \begin{align*}
        \Nrm{I_{2}+I_{3}}{L^1_x}\le&\, \frac{C}{N}\intdd V^\varepsilon_{N}(x-y)\n{\phi^\varepsilon_{N, t}(x)}^2\Nrm{\sh(y, \cdot)}{L^2_{z}}^2\dd x\d y\\
        &\, + \frac{C}{N^2}\intdd V^\varepsilon_{N}(x-y)\Nrm{\sh(x, \cdot)}{L^2_{z}}^2\Nrm{\sh(y, \cdot)}{L^2_{z}}^2\dd x\d y \\
        \le&\, \frac{C}{N}\Nrm{V^\varepsilon_N}{L^1_x}\Nrm{\phi^\varepsilon_{N, t}}{L^6_x}^2\Nrm{\sh}{L^3_x L^2_{x'}}^2+\frac{C}{N^2}\Nrm{V^\varepsilon_N}{L^1_x}\Nrm{\sh}{L^4_x L^2_{x'}}^4\\
        \le&\, \frac{C}{N\varepsilon^3}\Nrm{V^\varepsilon_N}{L^1_x}\Nrm{\weight{\varepsilon\grad}\phi^\varepsilon_{N, t}}{L^2_x}^2\Nrm{\sh}{\mathrm{HS}}\Nrm{\weight{\varepsilon\grad}\sh}{\mathrm{HS}}\\
        &\, + \frac{C}{N^2\varepsilon^3}\Nrm{V^\varepsilon_N}{L^1_x}\Nrm{\sh}{\mathrm{HS}}\Nrm{\weight{\varepsilon\grad} \sh}{\mathrm{HS}}^3
        \le\, \frac{C\lambda^\frac52}{N^\frac12\varepsilon^3}
    \end{align*}
     where the last inequality follows by Lemma~\ref{lem:pair_excitation_estimates}.

\end{proof}

\begin{lem}\label{lem:interaction_energy_bounds_on_remainder-terms}
    Assuming the same hypotheses as in Proposition~\ref{prop:many-body_energies_mean-field_approximation}. Then we have the following bound
    \begin{align}
        &\begin{multlined}[t][.87\textwidth]\label{est:potential_energy_density_R1-term}
        \frac{1}{N^2}\intdd V^\varepsilon_N(x-y)\n{\sfR_1(x, y)}\dd x\d y \\
        \lesssim_{T_{0}} \frac{1}{N}\inprod{\Psi_{\mathrm{fluc, t}}}{\cH_N\, \Psi_{\mathrm{fluc, t}}}+\frac{\lambda}{N\ve^{M}}\inprod{\Psi_{\mathrm{fluc, t}}}{(\cN+1)\, \Psi_{\mathrm{fluc, t}}}\ ,
        \end{multlined}\\
        &\begin{multlined}[t][.87\textwidth]\label{est:potential_energy_density_R2-term}
            \frac{1}{N^2}\intdd V^\varepsilon_N(x-y)\n{\sfR_2(x, y)}\dd x\d y \\
            \lesssim_{T_{0}} \frac{1}{N}\inprod{\Psi_{\mathrm{fluc, t}}}{\cH_N\, \Psi_{\mathrm{fluc, t}}}+\frac{\lambda}{N\ve^{M}}\inprod{\Psi_{\mathrm{fluc, t}}}{(\cN+1)\, \Psi_{\mathrm{fluc, t}}}\ ,
        \end{multlined}\\
        & \frac{1}{N^2}\intdd V^\varepsilon_N(x-y)\n{\sfR_3(x, y)}\dd x\d y
            \lesssim_{T_{0}} \frac{\lambda}{\sqrt{N}\varepsilon^3}\inprod{\Psi_{\mathrm{fluc, t}}}{(\cN+1)\, \Psi_{\mathrm{fluc, t}}}^\frac12\ .\label{est:potential_energy_density_R3-term}
    \end{align}
\end{lem}
\begin{proof}
    The proof of Inequality~\eqref{est:potential_energy_density_R1-term} follows from the proof of Proposition~\ref{prop:estimates_for_cubic_terms}, Lemma~\ref{lem:fluctuation_number_squared}, Lemmas~\ref{lem:semiclasscial_propagation_of_regularity} and \ref{lem:pair_excitation_estimates}.

    To prove Inequality~\eqref{est:potential_energy_density_R3-term}, by the Cauchy--Schwarz inequality and Lemma~\ref{lem:bogoliubov_conjugation_operator_inequality}, we have
    \begin{align*}
        &\intdd V^\varepsilon_N(x-y)\n{\sfR_3(x, y)}\dd x\d y \\
        &\le\, \frac{C}{\sqrt{N}}\Nrm{V^\varepsilon_{N}}{L^1_x}\Nrm{(\cN+1)^\frac12\,\Psi_{\mathrm{fluc}, t}}{}\sup_z\(\intd \n{f_{\mathrm{HFB}}(x, x+z)}^2\dd x\)^\frac12\ .
    \end{align*}
    Now, to estimate the $f_{\mathrm{HFB}}$ term, it suffices to bound the $\Lambda$ term in Expression~\eqref{def:f_HFB} since the others are handled similarly. By Part~\eqref{part:pair_excitation_pointwise_est_for_p} of Lemma~\ref{lem:pair_excitation_estimates}, we have
    \begin{align*}
        &\(\intd \n{\phi^\varepsilon_{N, t}(x)}^2\n{\Lambda_t(x, x+z)}^2\dd x\)^\frac12 \\
        &\le\, \Nrm{\phi^\varepsilon_{N, t}}{L^6_x}\sup_z\Nrm{\Lambda_t(\cdot, \cdot+z)}{L^3_x}\\
        &\le\, \Nrm{\phi^\varepsilon_{N, t}}{L^6_x}^3+\frac{1}{N}\Nrm{\phi^\varepsilon_{N, t}}{L^6_x}\sup_z\Nrm{k(\cdot, \cdot+z)}{L^3_x}+\frac{1}{N}\Nrm{\phi^\varepsilon_{N, t}}{L^6_x}\sup_z\Nrm{r_2(\cdot, \cdot+z)}{L^3_x}\\
        &\le\, \frac{C}{\varepsilon^3}\Nrm{\weight{\varepsilon\grad}\phi^\varepsilon_{N, t}}{L^2_x}^3\ ,
    \end{align*}
    which yields the desired result.

    Finally, to prove Inequality~\eqref{est:potential_energy_density_R2-term}, it suffices to estimate the terms involving $\Lambda_t^\varepsilon$ since the other terms can be handled similarly. Let us start by noting that
    \begin{align*}
        \nor{b^\ast_x b^\ast_y} = a^\ast(\ch_{x})a^\ast(\ch_{y})+a^\ast(\ch_{y})a(\sh_x)+a^\ast(\ch_{x})a(\sh_y)+a(\sh_x)a(\sh_y)\ .
    \end{align*}
    It suffices to consider just the first term.  By the Cauchy--Schwarz inequality, we have that
    \begin{align}
        &\frac{1}{N}\intdd V^\varepsilon_N(x-y)\n{\Lambda_t^\varepsilon(x, y)\, \inprod{\Psi_{\mathrm{fluc}, t}}{a(c_{x})a(c_{y})\Psi_{\mathrm{fluc}, t}}}\dd x\d y\label{est:potential_energy_density_R2-term1}\\
        &\le\, \frac{1}{N}\intdd V^\varepsilon_N(x-y)\n{\Lambda_t^\varepsilon(x, y)}^2\dd x\d y+\frac{1}{N}\intdd V^\varepsilon_N(x-y)\Nrm{a(c_{x})a(c_{y})\Psi_{\mathrm{fluc}, t}}{}^2 \dd x\d y\ . \notag
    \end{align}
    Then, by Part~\eqref{part:pair_excitation_pointwise_est_for_p} of Lemma~\ref{lem:pair_excitation_estimates} and Lemma~\ref{lem:singular_quartic_operator_bounds}, we have that
    \begin{align*}
        \mathrm{RHS}\eqref{est:potential_energy_density_R2-term1}\lesssim_{T_{0}}& \frac{\lambda}{N\, \ve^{3}}
        +\frac{1}{N}\inprod{\Psi_{\mathrm{fluc, t}}}{\cH_N\, \Psi_{\mathrm{fluc, t}}}+\frac{\lambda}{N\ve^{M}}\inprod{\Psi_{\mathrm{fluc, t}}}{(\cN+1)\, \Psi_{\mathrm{fluc, t}}}\ .
    \end{align*}
    This completes the proof of the lemma.
\end{proof}

Now, we are able to prove Proposition~\ref{prop:many-body_energies_mean-field_approximation}.
\begin{proof}[\textbf{Proof of Proposition~\ref{prop:many-body_energies_mean-field_approximation}}]
By the triangle inequality, we have
    \begin{align*}
       &\Nrm{E^\varepsilon_{N,\, \mathrm{kin}., t}-\tfrac12\n{\varepsilon\grad\phi^\varepsilon_{N, t}}^2-4\pi\mu\la\, \mathfrak{b}_{0}^{\mu}\n{\phi^\varepsilon_{N,t}}^4}{L^1_x}\\
        &\lesssim\, \Nrm{\diag(\varepsilon \grad_{x}\cdot\varepsilon\grad_{x'}\Gamma_{N:1, t}^\varepsilon)-\diag(\varepsilon \grad_{x}\cdot\varepsilon\grad_{x'}\Gamma^\varepsilon_{\mathrm{cor},t})}{L^1_x}\\
        &\quad +\Nrm{\frac{1}{N}\Nrm{\varepsilon\grad_x \sh(k_{N, t}^\varepsilon)(x, \cdot)}{L^2_{x'}}^2-4\pi\mu\la\, \mathfrak{b}_{0}^{\mu}\n{\phi^\varepsilon_{N,t}}^4}{L^1_x}.
    \end{align*}
 Then, by Proposition~\ref{prop:mean-field_approximation_of_marginal_density} and Lemma~\ref{lem:kinetic_energy_density_approx}, we complete the proof of Inequality~\eqref{est:kinetic_energy_many-body_mean-field_approximation}.

Next, we approximate $e_{N,\,\mathrm{int}.}$. By Lemma~\ref{lem:second_marginal_expansion} and Equation~\eqref{est:expected_number_bogoliubov_states}, we see that
\begin{align}\label{est:interaction_energy_approximation}
   & \Nrm{E^\varepsilon_{N,\, \mathrm{int}.,\, t}(x)-\textstyle\frac12\intd V^\varepsilon_{N}(x-y)\Gamma^{\varepsilon}_{N:2, t}(x, y; x, y)\dd y}{L^1_x}\\
    &\lesssim\, \frac{1}{\sqrt{N}}\intdd V^\varepsilon_{N}(x-y)\Gamma^{\varepsilon}_{N:2, t}(x, y; x, y)\dd x\d y\notag\\
    &\quad+\frac{1}{N^2}\intdd V^\varepsilon_{N}(x-y)\left\{\n{\sfR_1(x, y)}+ \n{\sfR_2(x, y)}+ \n{\sfR_3(x, y)}\right\}\dd x\d y\ .\notag
\end{align}
Applying Lemmas~\ref{lem:interaction_energy_density_approximation} and \ref{lem:interaction_energy_bounds_on_remainder-terms}, and Proposition~\ref{prop:bounds_on_growth_of_fluctuation}, we complete the proof of Inequality~\eqref{est:interaction_energy_many-body_mean-field_approximation}.
\end{proof}

\section{Bounds on the Fluctuations}\label{section:Bounds on the Fluctuations}
\subsection{The fluctuation Hamiltonian} Notice the fluctuation dynamics $\cU_{\mathrm{fluc}}(t, s)$ satisfies the equation
\begin{align}\label{def:fluctuation_dynamics}
    i\varepsilon\,\bd_t \cU_{\mathrm{fluc}}(t, s) = \cH_{\fluc}(t)\,\cU_{\mathrm{fluc}}(t, s) \quad \text{ with } \quad \cU_{\mathrm{fluc}}(s, s)=\id_{\cF}\quad \forall\, s\in \R\ ,
\end{align}
where its time-dependent generator is given by
\begin{align}
    \cG_{\fluc}(t)=&\, \(i\varepsilon\,\bd_t e^{\sqrt{N}\,\cA(\phi_{N, t}^\varepsilon)}\)e^{-\sqrt{N}\,\cA(\phi_{N, t}^\varepsilon)}+ e^{\sqrt{N}\,\cA(\phi_{N, t}^\varepsilon)}\,\cH_N\,e^{-\sqrt{N}\,\cA(\phi_{N, t}^\varepsilon)}\ ,\\
    \cH_{\fluc}(t)=&\, \(i\varepsilon\,\bd_t e^{\cB(k_{N, t}^\varepsilon)}\)e^{-\cB(k_{N, t}^\varepsilon)}
    +e^{\cB(k_{N, t}^\varepsilon)}\,\cG_{\fluc}(t)\,e^{-\cB(k_{N, t}^\varepsilon)}\ .\label{def:fluctuation_Hamiltonian}
\end{align}
Furthermore, the fluctuation Hamiltonian \eqref{def:fluctuation_Hamiltonian} could be rewrite as follows
\begin{align}
    \hfluc = N\chi +\cL_{\rm HFB} + \cQ_{\rm HFB} + \cC+ \cD+N^{-1}\cV
\end{align}
where $\chi$ is a time-dependent scalar, $\cL_{\rm HFB}, \cQ_{\rm HFB}, \cC$, and $\cD$ are homogeneous time-dependent operators in creation and annilihation operators. The explicit form of $\cG_{\fluc}$ and $\cH_{\fluc}$ are well known in the literature.
Moreover, the calculation of $\cG_{\fluc}$ and $\hfluc$ below in Lemma \ref{lem:Gfluc_Hamiltonian_definition} and Proposition \ref{prop:fluctuation_Hamiltonian} only requires $\phi_t$ and $k_t$ to be sufficiently smooth and integrable so that the quantities are well-defined.
For the proof of Lemma \ref{lem:Gfluc_Hamiltonian_definition} and Proposition \ref{prop:fluctuation_Hamiltonian}, we refer the reader, for instance, to \cite{grillakis2013beyond, grillakis2013pair, grillakis2017pair,hepp1974classical} (cf., \cite{benedikter2015quantitative,boccato2017quantum}).

Let us summarize the explicit form of $\cG_{\fluc}$ in the following lemma without proof.
\begin{lem}\label{lem:Gfluc_Hamiltonian_definition}
    We have the identity
    \begin{align}
        \cG_{\fluc}(t)= N\chi_1(t) +\cQ_{\rm Bog}(t)+\cR_1(t)+\cR_3(t)+N^{-1}\cV
    \end{align}
    where each of the terms are monomial operators in creation and annilihation operators.
    More precisely, the scalar term is given by
    \begin{subequations}
        \begin{align}\label{def:e0_term}
            \chi_1(t):= -\re\inprod{\phi_t}{i\varepsilon\,\dot\phi_t-h_{\mathrm{H}}(t)\phi_t}+\chi_0(t)
        \end{align}
        with $h_{\mathrm{H}}(t):=-\tfrac{\varepsilon^2}{2}\lapl+V_N^\varepsilon\ast|\phi_t|^2$ and $\chi_0(t)=-\tfrac12\intd (V_N^\varepsilon\ast|\phi_t|^2)|\phi_t|^2\dd x$. The leading term is given by the quadratic Hamiltonian operator
        \begin{align}\label{def:c-sub_quadratic_Hamiltonian}
            \cQ_{\rm Bog}(t) = \frac12\dG(h_{\mathrm{HF}}(t))
            +\frac12\intdd m(t, x, y)\, a^\ast_{x}a^\ast_{y}\dd x\d y+\mathrm{h.c.}\ ,
        \end{align}
        where $h_{\mathrm{HF}}(t)= h_{\mathrm{H}}(t)+\sfX_{c}(t)$. Here $\sfX_{c}$ and $m$ are given by
        \begin{align*}
            \sfX_{c}(t, x, y) =&\, V_N^\varepsilon(x-y)\, \phi_t(x)\conj{\phi_t}(y)\ ,\\
            m(t, x, y) =&\, V_N^\varepsilon(x-y)\, \phi_t(x) \phi_t(y)\ .
        \end{align*}
        Lastly, the remainder terms are given by
        \begin{align}
            \cR_1(t)=&\,  \sqrt{N}\,a(-i\varepsilon\,\dot\phi_t+h_{\mathrm{H}}(t)\,\phi_t)+\mathrm{h.c.}\ ,\label{def:linear_remainder_term}\\
            \cR_3(t)=&\,  \frac{1}{\sqrt{N}}\intd V_N^\varepsilon(x-y)\, \phi_t(y)\, a^\ast_x a^\ast_y a_x+\mathrm{h.c.}\label{def:cubic_remainder_term}\ .
        \end{align}
    \end{subequations}
    Furthermore, if $\phi_t$ satisfies Equation \eqref{eq:modified_Gross--Pitaevskii}, then we see that Term \eqref{def:e0_term} becomes
    \begin{align*}
        \chi_1(t)= \intd (V_N^\varepsilon(\cdot)\(\tfrac12-f_{N, l}^\varepsilon(\cdot)\)\ast |\phi_t|^2)|\phi_t|^2(x)\dd x\ ,
    \end{align*}
  and Term \eqref{def:linear_remainder_term} reduces to
    \begin{align*}
        \cR_1(t)= \sqrt{N}\,a((V_N^\varepsilon w_{N, l}^\varepsilon)\ast |\phi_t|^2\,\phi_t)+\mathrm{h.c.}\ .
    \end{align*}
\end{lem}

\begin{prop}\label{prop:fluctuation_Hamiltonian}
    The fluctuation Hamiltonian can be rewritten as follows
    \begin{align*}
        \hfluc = N\,\chi +\cL_{\rm HFB} + \cQ_{\rm HFB} + \cC+ \cD_1+\cD_2+N^{-1}\cV
    \end{align*}
    where each of the terms are normal-ordered operators in creation and annilihation operators. More precisely, the scalar term is given by
    \begin{align*}
        \chi =  \re\inprod{\phi}{\(-i\varepsilon\,\bd_t+h_{\mathrm{H}}\)\phi}-\chi_0
        -\tfrac{1}{2\,N}\re\(\Tr{\conj{\sh(k)}\,\vect{\widetilde S}(\sh(k))}\)+\tfrac{1}{2\, N}\re\(\Tr{\sh(2k)\, \conj{\widetilde m}}\)
    \end{align*}
    where $\widetilde{m}(x, y)=V_N^\varepsilon(x-y)\Lambda(x, y)$ and
    \begin{align*}
        \vect{\widetilde S}(\sh) = i\varepsilon\,\bd_t \sh - \widetilde h_{\mathrm{HF}}\,\sh-\sh\,\widetilde h_{\mathrm{HF}}^\top \quad \text{ with } \quad \widetilde h_{\mathrm{HF}}=-\tfrac{\varepsilon^2}{2}\lapl +V_{\Gamma^\varepsilon}+\sfX_{\Gamma^\varepsilon}\ .
    \end{align*}

    The linear terms are given by
    \begin{align}\label{def:normal_ordered_linear_terms}
        \cL_{\rm HFB} =  \sqrt{N}\intd \conj{(-i\varepsilon\,\bd_t\phi+h_{\mathrm{HFB}}(\phi))(x)}\, e^{\cB(k)}a_x\,e^{-\cB(k)}\d x+\mathrm{h.c.}
    \end{align}
    where
    \begin{align*}
        h_{\mathrm{HFB}}(\phi)(x) :=& (\widetilde h_{\mathrm{H}}\,\phi)(x)
         + \tfrac{1}{2\,N}\intd V_N^\varepsilon(x-y)\,\phi(y)\, p(2k)(y, x)\,\dd y\\
         &+ \tfrac{1}{2\,N}\intd V_N^\varepsilon(x-y)\,\conj{\phi(y)}\, \sh(2k)(y, x)\,\dd y\ .
    \end{align*}

    The quadratic terms are given by
    \begin{align}\label{def:normal_ordered_HFB_quadratic_operator}
        \cQ_{\rm HFB} :=& \tfrac12\dG(\widetilde h_{\mathrm{HF}})+\tfrac12\dG(d_{\rm HFB})-\tfrac12\intdd l_{\rm HFB}(x, y)\, a^\ast_x a^\ast_y\dd x\d y+\mathrm{h.c.}\\
        :=& \dG(\widetilde h_{\mathrm{HF}})+\tfrac12\dG(d_{\rm HFB}+d_{\rm HFB}^\top)+\cQ_{\rm OD}\ ,\notag
    \end{align}
    where the diagonal and off-diagonal entries are given by
    \begin{subequations}\label{def:entries_of_QHFB}
        \begin{align}
            d_{\rm HFB}:=&\,\(\vect{\widetilde W}(\ch(k))+\sh(k)\, \conj{\widetilde m}\)\ch(k)-\(\vect{\widetilde S}(\sh(k))-\ch(k)\, \widetilde m \)\conj{\sh(k)}\ , \label{def:diagonal_entry_of_QHFB}\\
            l_{\rm HFB}:=&\, \(\vect{\widetilde S}(\sh(k))-\ch(k)\,\widetilde m\)\conj{\ch(k)}-\(\vect{\widetilde W}(\ch(k))+\sh(k)\, \conj{\widetilde m}\)\sh(k)\ . \label{def:off-diagonal_entry_of_QHFB}
        \end{align}
    \end{subequations}
    and the operator $\vect{\widetilde W}$ is given by
    \begin{align*}
        \vect{\widetilde W}(p)=i\varepsilon\,\bd_t p-[\widetilde h_{\mathrm{HF}}, p]\ .
    \end{align*}

    The cubic terms are given by
    \begin{subequations}\label{def:normal_ordered_cubic_terms}
    \begin{align}
        \cC&= \frac{1}{\sqrt{N}}\intdd V_N^\varepsilon(x-y)\phi(y)\,a^\ast(\ch_{x}) a^\ast(\ch_{y})  a^\ast(\sh_x)\dd x\d y\label{def:normal_ordered_cubic_term1}\\\
        &\quad +\frac{1}{\sqrt{N}}\intdd V_N^\varepsilon(x-y)\phi(y)\,a^\ast(\ch_{x})a^\ast(\ch_{y})a(\ch_{x})\dd x\d y\label{def:normal_ordered_cubic_term2}\\
        &\quad +\frac{1}{\sqrt{N}}\intdd V_N^\varepsilon(x-y)\phi(y)\,a^\ast(\ch_{y}) a^\ast(\sh_x)a(\sh_x)\dd x\d y\label{def:normal_ordered_cubic_term3}\\
        &\quad +\frac{1}{\sqrt{N}}\intdd V_N^\varepsilon(x-y)\phi(y)\,a^\ast(\ch_{x}) a^\ast(\sh_x) a(\sh_y)\dd x\d y\label{def:normal_ordered_cubic_term4}\\
        &\quad +\frac{1}{\sqrt{N}}\intdd V_N^\varepsilon(x-y)\phi(y)\,a^\ast(\ch_{x}) a(\sh_y) a(\ch_{x})\dd x\d y\label{def:normal_ordered_cubic_term5}\\
        &\quad +\frac{1}{\sqrt{N}}\intdd V_N^\varepsilon(x-y)\phi(y)\,a^\ast(\ch_{y}) a(\sh_x)  a(\ch_{x})\dd x\d y\label{def:normal_ordered_cubic_term6}\\
        &\quad +\frac{1}{\sqrt{N}}\intdd V_N^\varepsilon(x-y)\phi(y)\,a^\ast(\sh_x) a(\sh_x)  a(\sh_y)\dd x\d y \label{def:normal_ordered_cubic_term7}\\
        &\quad +\frac{1}{\sqrt{N}}\intdd V_N^\varepsilon(x-y)\phi(y)\,a(\sh_x)  a(\sh_y)  a(\ch_{x}) \dd x\d y +\mathrm{h.c.}
    \end{align}
    where $g_x(y):=g(y, x)$.
\end{subequations}

    Lastly, the quartic terms are split into $\cD_1+\cD_2+N^{-1}\cV$, where
    \begin{subequations}\label{def:normal_ordered_quartic_D1_terms}
        \begin{align}
         \mathcal{D}_1 =&\  \frac{1}{2\, N}\intdd   V_N^\varepsilon(x-y)\, a^\ast(\ch_{x})a^\ast(\ch_{y})a^\ast(\sh_x)a^\ast(\sh_y)\dd x \d y \label{def:normal_ordered_quartic_D1_term1}\\
         & + \frac{1}{2\, N}\intdd  V_N^\varepsilon(x-y)\, a^\ast(\ch_{x})a^\ast(\ch_{y})a^\ast(\sh_x)a(\ch_{y})\dd x\d y \label{def:normal_ordered_quartic_D1_term2}\\
         &+ \frac{1}{2\, N}\intdd   V_N^\varepsilon(x-y)\, a^\ast(\ch_{x})a^\ast(\ch_{y})a^\ast(\sh_y)a(\ch_{x})\dd x\d y\\
         &+ \frac{1}{2\, N}\intdd   V_N^\varepsilon(x-y)\, a^\ast(\ch_{x})a^\ast(\sh_x)a^\ast(\sh_y)a(\sh_y)\dd x\d y \\
         &+ \frac{1}{2\, N}\intdd V_N^\varepsilon(x-y)\, a^\ast(\sh_x)a^\ast(\ch_{y})a^\ast(\sh_y)a(\sh_x)\dd x\d y +\mathrm{h.c.}
         \end{align}
    \end{subequations}
        and
    \begin{subequations}\label{def:normal_ordered_quartic_D2_terms}
        \begin{align}
        \mathcal{D}_2 =&\,
         \frac{1}{2\, N}\intdd V_N^\varepsilon(x-y)\, a^\ast(\ch_{x})a^\ast(\sh_x)a(\ch_{y})a(\sh_y)\dd x\d y \\
        &+ \frac{1}{2\, N}\intdd V_N^\varepsilon(x-y)\, a^\ast(\sh_x)a^\ast(\ch_{y})a(\ch_{y})a(\sh_x)\dd x\d y\\
        &+\frac{1}{4\,N}\intdd V_N^\varepsilon(x-y) a^\ast(\sh_x)a^\ast(\sh_y)a(\sh_x)a(\sh_y)\dd x\d y\\
        &+ \frac{1}{4\, N}\intdd V_N^\varepsilon(x-y)a^\ast(\ch_{x})a^\ast(\ch_{y})a(\ch_{x})a(\ch_{y})\dd x\d y\\
        &- \frac{1}{4\, N}\intdd V_N^\varepsilon(x-y)a^\ast_xa^\ast_ya_y a_x\dd x\d y+\mathrm{h.c.}
        \end{align}
    \end{subequations}
        where $\cD_2$ commutes with $\cN$ and $\cD_1$ does not. We also write
        \begin{align}\label{def:renormalized_fluctuation_Hamiltonian}
            \widetilde\cH_{\rm fluc} = \hfluc-N\chi\ .
        \end{align}
\end{prop}

Then the main result of this section is given by the following proposition.

\begin{prop}\label{prop:bounds_on_fluctuation_Hamiltonian}
    Suppose $\phi_t=\phi_{N, t}^\varepsilon$ is a solution to Equation \eqref{eq:modified_Gross--Pitaevskii} with initial data $\weight{\varepsilon \grad}^4\phi^{\init}\in L^2(\R^3)$ holds uniformly in $\varepsilon$. Suppose $k_{t}=k_{N, t}^\varepsilon$ is given by Equation \eqref{def:pair_excitation_function} with $\ell=\varepsilon^4$ for the necessary regimes.
    Then there exists $C, M>0$, independent on $\varepsilon$ and $N$,  such that we have the following operator bounds
    \begin{align}
        \widetilde\cH_{\rm fluc}(t) \ge&\, \tfrac12\cH_N - C\lambda \ve^{-M}\lra{t}^{-M} \(\frac{\cN^2}{N}+\cN+1\)\label{est:lower_bound_for_fluctuation_Hamiltonian}\ ,\\
        \widetilde\cH_{\rm fluc}(t) \le&\, \tfrac{3}{2}\cH_N+ C\lambda \ve^{-M}\lra{t}^{-M} \(\frac{\cN^2}{N}+\cN+1\)\ .\label{est:upper_bound_for_fluctuation_Hamiltonian}
    \end{align}
    Moreover, we also have that
    \begin{align}
        \pm [\cN,  \widetilde\cH_{\rm fluc}(t)] \le&\, \cH_N + C\lambda\ve^{-M}\lra{t}^{-M}\(\frac{\cN^2}{N}+\cN+1\)\ ,\label{est:bound_for_commutator_of_fluctuation_Hamiltonian_with_number}\\
        \pm \dot{\widetilde\cH}_{\rm fluc}(t) \le&\,
        \cH_N +C\lambda\ve^{-M}\lra{t}^{-M} \(\frac{\cN^2}{N}+\cN+1\)\ . \label{est:bound_for_time-derivative_of_fluctuation_Hamiltonian}
    \end{align}
\end{prop}

\subsection{Preliminary operator estimates} To prove Proposition \ref{prop:bounds_on_fluctuation_Hamiltonian}, let us start by recording some useful lemmas.
\begin{lem}[quadratic operator bounds]\label{lem:quadratic_operator_bounds}
    Let $j_1, j_2$ be Hilbert--Schmidt operators and $f, g \in L^\infty(\R^3)$. Consider the operators
        \begin{align*}
            \cQ_1:=&\, \intdd V_N^\varepsilon(x-y)f(x)g(y)\, a^\sharp(j_{1, x})a^\sharp(j_{2, y})\dd x\d y\ ,\\
            \cQ_2:=&\, \intdd V_N^\varepsilon(x-y)f(x)g(y)\, a^\sharp(j_{1, x})a_y\dd x\d y\ ,\\
            \cQ_3:=&\, \intdd V_N^\varepsilon(x-y)f(x)g(y)\, a^\sharp_xa_y\dd x\d y\ ,
        \end{align*}
            and similarly the operators
        \begin{align*}
            \cQ_4:=&\, \intdd V_N^\varepsilon(x-y)f(x)g(x)\, a^\sharp(j_{1, y})a^\sharp(j_{2, y})\dd x\d y\ , \\
            \cQ_5:=&\, \intdd V_N^\varepsilon(x-y)f(x)g(x)\, a^\sharp(j_{1, y})a_y\dd x\d y\ ,\\
            \cQ_6:=&\, \intdd V_N^\varepsilon(x-y)f(x)g(x)\, a^\sharp_y a_y \dd x\d y\ ,
        \end{align*}
    where $a^\sharp$ is either $a$ or $a^\ast$ and $j_x(y):=j(y, x)$. Then we have the following operator estimates
        \begin{align}
            \pm\cQ_{1}, \cQ_4\le&\, \Nrm{V_N^\varepsilon}{L^1} \Nrm{f}{L^\infty}\Nrm{g}{L^\infty}\Nrm{j_1}{\rm HS}\Nrm{j_2}{\rm HS}(\cN+1)\ ,\\
            \pm\cQ_{2}, \cQ_5\le&\, \Nrm{V_N^\varepsilon}{L^1} \Nrm{f}{L^\infty}\Nrm{g}{L^\infty}\Nrm{j_1}{\rm HS}(\cN+1)\ ,\\
            \pm\cQ_{3}, \cQ_6\le&\, \Nrm{V_N^\varepsilon}{L^1} \Nrm{f}{L^\infty}\Nrm{g}{L^\infty}(\cN+1)\ .
        \end{align}
\end{lem}

\begin{lem}[quartic operator bounds]\label{lem:quartic_operator_bounds}
    Let $j_1, j_2$ be Hilbert--Schmidt operators with the property that
    \begin{align*}
        M_\ii:= \max\(\sup_{x}\intd \n{j_\ii(x, y)}^2\dd y,  \sup_{y}\intd \n{j_\ii(x, y)}^2\dd x\)<\infty\ .
    \end{align*}
    Then we have the following bounds
        \begin{align}
            & \intdd V_N^\varepsilon(x-y) \Nrm{a^\sharp(j_{1,x}) a_x \Psi}{}^2\dd x \d y \le\,  \Nrm{V_N^\varepsilon}{L^1_x} M_1\Nrm{(\cN+1)\Psi}{}^2,\\
             &\intdd V_N^\varepsilon(x-y) \Nrm{a^\sharp(j_{1,x}) a(j_{2, x}) \Psi}{}^2\dd x \d y \\
             &\le\, \Nrm{V_N^\varepsilon}{L^1_x} \min\(M_1 \Nrm{j_2}{\rm HS}^2, M_2\Nrm{j_1}{\rm HS}^2\)\Nrm{(\cN+1)\Psi}{}^2,\nonumber\\
            &\intdd V_N^\varepsilon(x-y) \Nrm{a^\sharp(j_{1,x}) a(\ch_{x}) \Psi}{}^2\dd x \d y
           \le\, 2\Nrm{V_N^\varepsilon}{L^1_x} M_1(1+\Nrm{p}{\rm HS}^2)\Nrm{(\cN+1)\Psi}{}^2 \ .
        \end{align}
\end{lem}

\begin{lem}\label{lem:singular_quartic_operator_bounds}
    For every $\delta'>0$ we have the following bound
    \begin{align*}
        &\intdd V_N^\varepsilon(x-y) \Nrm{a(\ch_{x})a(\ch_{y}) \Psi}{}^2\dd x \d y \\
        \le& \(1+\delta'\)\intdd V_N^\varepsilon(x-y) \Nrm{a_x a_y \Psi}{}^2\dd x \d y
        + \frac{C}{\delta'}\Nrm{V^\varepsilon_N}{L^1_x} \Nrm{p}{L^\infty_xL^2_y}^2\(1+\Nrm{p}{\rm HS}^2\)\Nrm{(\cN+1)\Psi}{}^2.
    \end{align*}
\end{lem}

The proofs of Lemmas \ref{lem:quadratic_operator_bounds}--\ref{lem:singular_quartic_operator_bounds} can be found in \cite[Section 6]{benedikter2015quantitative}.

The following lemmas involves quadratic operators with derivatives and the pair excitation function \eqref{def:pair_excitation_function}.
\begin{lem}\label{lem:quadratic_operator_bounds_with_derivative}
    Let $j_1, j_2,$ and $[\varepsilon\grad, j_1]$ be Hilbert--Schmidt operators where $[\varepsilon\grad, j_1]$ has the integral kernel $\varepsilon\grad_{x+y}j_1(x, y)$.
    Consider the quadratic operators
    \begin{align*}
        \cQ_{7}:=&\, \intd a^\ast(j_{1, x})\, a^\sharp(j_{2,x})\dd x\ ,\\
        \cQ_{8}:=&\, \intd \varepsilon\grad_x a^\ast_x\, a^\sharp(j_{2, x})\dd x\ , \\
        \cQ_{9}:=&\, \intd a^\ast(\varepsilon\grad_x j_{1, x})\, a^\sharp(j_{2, x})\dd x\ ,
    \end{align*}
    where $a^\sharp$ is either $a$ or $a^\ast$ and $j_x(y):=j(y, x)$. Then there exists a universal constant $C>0$ such that for every $0<\delta\le 1$ we have the following estimates
    \begin{align}
        \pm \cQ_7\le&\, \Nrm{j_1}{\rm HS}\Nrm{j_2}{\rm HS}\(\cN+1\)\ ,\\
        \pm \cQ_8^\sharp \le&\, \delta \cK+C\delta^{-1}\Nrm{j_1}{\rm HS}^2(\cN+1)\ , \\
        \pm \cQ_9^\sharp\le&\, \delta \cK +C\(\delta^{-1}\Nrm{j_1}{\rm HS}^2\Nrm{j_2}{\rm HS}^2+\Nrm{[\varepsilon\grad, j_1]}{\rm HS}\Nrm{j_2}{\rm HS}\)(\cN+1)\ , \label{est:Q9_operator}
    \end{align}
    where $\sfA^\sharp$ is either $\sfA$ or its adjoint $\sfA^\ast$.
\end{lem}
\begin{proof}
    We shall consider only the proof for $\cQ_9$ since the others are much easier.
    Let us start by rewriting $\cQ_9$ as follows
    \begin{align*}
        \cQ_{9}=&\, \intdd j_1(y, x)\, \varepsilon \grad_y a_y^\ast a^\sharp (j_{2, x})\dd x +\intdd \varepsilon\grad_{x+y}j_1(y, x)\, a^\ast_y a^\sharp(j_{2, x})\dd x\d y\ .
    \end{align*}
    Then, by the Cauchy--Schwarz inequality and Lemma \ref{lem:estimates_for_creation_annilihation_operators}, we have the estimate
    \begin{align*}
        \n{\inprod{\Psi}{\cQ_9\,\Psi}} \le&\,  \intdd \n{j_1(y, x)}\Nrm{j_{2, x}}{L^2_z}\Nrm{\varepsilon \grad_y a_y \Psi }{}\norm{(\cN+1)^\frac12\Psi}\dd x \d y\\
        &\, + \intdd \n{\varepsilon\grad_{x+y}j_1(y, x)}\Nrm{j_{2, x}}{L^2_z}\Nrm{a_y \Psi }{}\norm{(\cN+1)^\frac12\Psi}\dd x \d y\\
        \le &\, \Nrm{j_1}{\rm HS}\Nrm{j_{2}}{\rm HS}\norm{\cK^\frac12\Psi}\norm{(\cN+1)^\frac12\Psi}
         +\Nrm{[\varepsilon\grad, j_1]}{\rm HS}\Nrm{j_{2}}{\rm HS}\norm{(\cN+1)^\frac12\Psi}^2
    \end{align*}
    which yields Inequality \eqref{est:Q9_operator}.
\end{proof}

As an immediate consequence of the previous lemma, we obtain the following result.
\begin{lem}\label{lem:quadratic_operator_bounds_with_derivative_of_k}
    Suppose $j$ is a Hilbert--Schmidt operator. Consider the operators
    \begin{align*}
        \cQ_{10}:=&\, \intd a^\ast(\varepsilon\grad_x k_{x})\, a^\sharp(j_{x})\dd x\ , \\
        \cQ_{11}:=&\, \intd a^\ast(\varepsilon^2\grad_x \dot k_{x})\, a^\sharp(j_{x})\dd x\ ,
    \end{align*}
    where $k=k_{N, t}^\varepsilon$ is defined by Expression \eqref{def:pair_excitation_function}  and $\phi_t$ solves Equation \eqref{eq:modified_Gross--Pitaevskii}.
    Then there exists a univerisal constant $C>0$ such that we have the bounds
    \begin{align}
        \pm \cQ_{10}^\sharp \le&\, \delta \cK +C\(\delta^{-1}\Nrm{k}{\rm HS}^2\Nrm{j}{\rm HS}^2+\Nrm{[\varepsilon\grad, k]}{\rm HS}\Nrm{j}{\rm HS}\)(\cN+1)\ ,\label{est:Q10_operator}\\
        \pm \cQ_{11}^\sharp \le&\, \delta \cK +C\(\delta^{-1}\Nrm{\varepsilon\dot k}{\rm HS}^2\Nrm{j}{\rm HS}^2+\Nrm{[\varepsilon\grad, \varepsilon \dot k]}{\rm HS}\Nrm{j}{\rm HS}\)(\cN+1)\ . \label{est:Q11_operator}
    \end{align}
\end{lem}

\begin{lem}\label{lem:quadratic_operator_bounds_with_derivative2}
    Consider the operators
    \begin{align*}
        \cQ_{12}:=&\, \intd a^\ast(\varepsilon\grad_x k_{x})\, a(\varepsilon\grad_x k_{x})\dd x\ , \\
        \cQ_{13}:=&\, \intd a^\ast(\varepsilon^2\grad_x \dot k_{x})\, a(\varepsilon\grad_x k_{x})\dd x\ ,
    \end{align*}
    Then we have the bounds
    \begin{align}
        \pm \cQ_{12}\le&\, \Nrm{k(-\varepsilon^2\lapl)\conj{k}}{\rm HS} \cN\ ,\\
        \pm \cQ_{13}^\sharp \le&\, \Nrm{\varepsilon\dot k(-\varepsilon^2\lapl)\conj{k}}{\rm HS}\cN\ .
    \end{align}
\end{lem}

\begin{proof}
    It suffices to consider $\cQ_{13}$ since the estimate for $\cQ_{12}$ is similar. Begin by rewriting the operator as follows
    \begin{align*}
        \cQ_{13} =&\, \int_{\R^{9}} \varepsilon^3\grad_x \dot k(y_1, x) \conj{\grad_x k(x, y_2)}\, a^\ast_{y_1}a_{y_2}\dd x\d y_1 \d y_2
        =:\, \intdd G(y_1, y_2)\, a^\ast_{y_1}a_{y_2}\dd y_1\d y_2\ ,
    \end{align*}
    then, by the Cauchy--Schwarz inequality, we have the estimate
    \begin{align*}
        \n{\inprod{\Psi}{\cQ_{13}\, \Psi}} \le \(\intdd \n{G(y_1, y_2)}^2\dd y_1 \d y_2\)^\frac12\norm{\cN^\frac12 \Psi}^2\ .
    \end{align*}
    This completes the proof.
\end{proof}

\subsection{Operator estimates for the fluctuation Hamiltonian}
Let us introduce a quantity
\begin{align}\label{equ:common quantity}
\sfC_0(t):=&1+\Nrm{k_{t}}{\rm HS}^{10}+\Nrm{p_{t}}{\rm HS}^{10}+\Nrm{r_{t}}{\rm HS}^{10}+
\Nrm{k_{t}}{L^\infty_xL^2_y}^{10}+\Nrm{p_{t}}{L^\infty_xL^2_y}^{10}+\Nrm{r_{t}}{L^\infty_x L^2_y}^{10}\\
&+\Nrm{\varepsilon\dot k_{t}}{L^\infty_xL^2_y}^{10}+\Nrm{\varepsilon\dot p_{t}}{L^\infty_xL^2_y}^{10}+\Nrm{\varepsilon\dot r_{t}}{L^\infty_x L^2_y}^{10}+\Nrm{\phi_{t}}{L^\infty_x}^{10}+\Nrm{\varepsilon\dot{\phi}_{t}}{L^\infty_x}^{10}+\Nrm{\varepsilon^2\ddot p_t}{\rm HS}^{10}\notag\\
&+\Nrm{\varepsilon^2\ddot \sh_t}{\rm HS}^{10}+ \Nrm{[\varepsilon\grad, \varepsilon\dot k]}{\rm HS}^{10}+\Nrm{\varepsilon\dot k(\varepsilon^2\lapl)\conj{k}}{\rm HS}^{10}+\Nrm{[\varepsilon\grad, k]}{\rm HS}^{10}
+\Nrm{k(\varepsilon^2\lapl)\conj{k}}{\rm HS}^{10} \notag\\
&+\Nrm{\varepsilon \grad p_{t}}{\rm HS}^{10}+\Nrm{\varepsilon \grad r_{t}}{\rm HS}^{10}+\Nrm{\varepsilon^2 \grad \dot p_{t}}{\rm HS}^{10}+\Nrm{\varepsilon^2 \grad \dot r_{t}}{\rm HS}^{10}\notag\ ,
\end{align}
which serves as a crude upper bound, and can be bounded by $1+\Nrm{\lra{\ve\nabla}^{4}\phi_{N}^{\ve}}{L^{2}}^{20}$. In the operator estimates for the fluctuation Hamiltonian, we shall employ the crude upper bound. Although not the most precise, it suffices for our primary goal and streamlines the proof.

\subsubsection{Cubic Terms} Here we estimate $\cC$ and the other corresponding operators.
\begin{prop}\label{prop:estimates_for_cubic_terms}
    Let $\cC$ be defined as in \eqref{def:normal_ordered_cubic_terms}, then for any $\delta>0$ there exists $C_{\delta}>0$  such that we have the following operator estimates
        \begin{align}
            \pm \cC, [\cN, \cC] \le&\, \frac{\delta }{N}\cV+ C_{\delta}\Nrm{V_N^\varepsilon}{L^1}\sfC_0(t) \( \,\frac{\cN^2}{N}+(\cN+1)\)\ , \label{est:cubic_term_bound}\\
            \pm \varepsilon\dot\cC \le&\,  \frac{\delta }{N}\cV+ C_{\delta}\Nrm{V_N^\varepsilon}{L^1}\sfC_0(t) \( \,\frac{\cN^2}{N}+(\cN+1)\) \label{est:cubic_time-derivative_term_bound} .
        \end{align}

\end{prop}

\begin{proof}
    Let us consider an example cubic term of $\cC$ (see Expressions \eqref{def:normal_ordered_cubic_terms})
    \begin{align}\label{eq:generic_cubic_term}
        \frac{1}{\sqrt{N}}\intdd V_N^\varepsilon(x-y)\phi_{t}(y)\, a^\ast(j_{1, x})\, a^\ast(j_{2, x})\,a(j_{3, x})\dd x\d y\ ,
    \end{align}
    where $j_i$ takes on either $\ch(k_t)$ or $\sh(k_t)$. The main idea is to estimate the cubic operator in terms of quadratic and quartic operators estimates, that is, for every $\delta>0$, it follows that
    \begin{align*}
        &\n{\inprod{\Psi}{\eqref{eq:generic_cubic_term}\, \Psi}}\\
        &\le \frac{1}{\sqrt{N}} \intdd V_N^\varepsilon(x-y)\n{\phi_{t}(y)}\Nrm{a(j_{1, y})a(j_{2, x})\Psi}{}\Nrm{a(j_{3, x})\Psi}{}\dd x\d y\\
        &\le \(\frac{\delta}{N}\intdd V_N^\varepsilon(x-y)\Nrm{a(j_{1, y})a(j_{2, x})\Psi}{}^2\dd x\d y\)^\frac12\\
        &\qquad \times \(\frac{1}{\delta} \intdd V_N^\varepsilon(x-y)\n{\phi_{t}(y)}^2\Nrm{a(j_{3, x})\Psi}{}^2\dd x\d y\)^\frac12.
    \end{align*}

    In the case where either $j_1$ or $j_2$ is equal to $\sh(k_{t})$, say $j_1=\sh(k_{t})$, then by Lemma \ref{lem:quadratic_operator_bounds} and Lemma \ref{lem:quartic_operator_bounds}, we have the estimate
    \begin{align*}
        \n{\inprod{\Psi}{\eqref{eq:generic_cubic_term}\, \Psi}}
        \le&\, C \Nrm{V^\varepsilon_N}{L^1} \sfC_0(t)\lrs{\inprod{\Psi}{N^{-1}(\cN+1)^2\, \Psi}+\inprod{\Psi}{(\cN+1)\,\Psi}\ }.
    \end{align*}
    In the other case where $j_{1}=j_2 = \ch(k_{t})$, then, by Lemma \ref{lem:quadratic_operator_bounds} and Lemma \ref{lem:singular_quartic_operator_bounds} (with $\delta'=1$), we have that
    \begin{align*}
        &\n{\inprod{\Psi}{\eqref{eq:generic_cubic_term}\, \Psi}} \\
        &\le\,
        \(\frac{1}{N}\intdd V_N^\varepsilon(x-y)\Nrm{a(c_{x})a(c_{y})\Psi}{}^2\dd x\d y\)^\frac12
        \Nrm{V^\varepsilon_N}{L^1}^\frac12\sfC_0(t)^\frac12\Nrm{(\cN+1)^\frac12\Psi}{}\\
        &\le\,  \delta \inprod{\Psi}{N^{-1}\cV\, \Psi}
         +  C_{\delta} \Nrm{V^\varepsilon_N}{L^1} \sfC_0(t)\lrs{\inprod{\Psi}{N^{-1}(\cN+1)^2\, \Psi}+\inprod{\Psi}{(\cN+1)\,\Psi}\ }\ .
    \end{align*}

    The other cubic terms are handled in a similar manner. This completes the proof of Inequality \eqref{est:cubic_term_bound}. The proof for the commutator term with the number operator is exactly the same since computing the commutator yields
    \begin{align*}
        [\cN, \text{generic cubic term}]
         = \frac{C}{\sqrt{N}}\intdd V_N^\varepsilon(x-y)\phi_{t}(y)a^\sharp(j_{1, y})a^\sharp(j_{2, x})a^\sharp(j_{3, x})\dd x\d y\ .
    \end{align*}
    where $a^\sharp$ is either $a$ or $a^\ast$ (normal ordered). The same argument also applies for Inequality \eqref{est:cubic_time-derivative_term_bound} involving the time derivative.
\end{proof}

\subsubsection{Quartic Terms} The following proposition handles the quartic terms $\cD$.
\begin{prop}\label{prop:estimates_for_quartic_terms}
    Suppose $\cD=\cD_1+\cD_2$ as defined by Expressions \eqref{def:normal_ordered_quartic_D1_terms}--\eqref{def:normal_ordered_quartic_D2_terms}, then there exists a univerisal $C>0$  such that for any $0<\delta<1$ we have the following operator estimates
        \begin{align}
            \pm \cD, [\cN, \cD_1]
            \le&\, \frac{\delta}{N}\cV+ C_{\delta}\Nrm{V_N^\varepsilon}{L^1}\sfC_0(t) \frac{(\cN+1)^2}{N}\ , \label{est:quartic_term_bound}\\
            \pm \varepsilon\dot\cD \le&\,  \frac{\delta}{N}\cV+  C_{\delta}\Nrm{V_N^\varepsilon}{L^1}\sfC_0(t) \frac{(\cN+1)^2}{N}\ , \label{est:quartic_time-derivative_term_bound}
        \end{align}
    where $\sfC_0(t)$ is defined as in \eqref{prop:estimates_for_cubic_terms}.
\end{prop}

\begin{proof}
    It suffices to consider the following term
    \begin{align*}
        \eqref{def:normal_ordered_quartic_D1_term1} =&\, \frac{1}{2\, N}\intdd   V_N^\varepsilon(x-y)\, a^\ast(\ch_{x})a^\ast(\ch_{y})a^\ast(\sh_x)a^\ast(\sh_y)\dd x \d y,
    \end{align*}
    since the rest of the terms can be handled in the exact same manner. In the first term, we  have that
    \begin{align}\label{est:lem_quartic_D1_term1}
        \n{\inprod{\Psi}{\eqref{def:normal_ordered_quartic_D1_term1}\, \Psi}}
        \le& \(\frac{\delta}{2\,N}\intdd V_N^\varepsilon(x-y)\Nrm{a(\ch_{x})a(\ch_{y})\Psi}{}^2\dd x\d y\)^\frac12\\
        &\times \(\frac{1}{2N\delta} \intdd V_N^\varepsilon(x-y)\Nrm{a^\ast(\sh_{x})a^\ast(\sh_y)\Psi}{}^2\dd x\d y\)^\frac12.\notag
    \end{align}
    Then, by Lemmas \ref{lem:quartic_operator_bounds}--\ref{lem:singular_quartic_operator_bounds}, we have that
    \begin{align*}
        \mathrm{RHS}\,\eqref{est:lem_quartic_D1_term1}
        \le&\,     \(\frac{\delta}{2\,N}\intdd V_N^\varepsilon(x-y)\Nrm{a(\ch_{x})a(\ch_{y})\Psi}{}^2\dd x\d y\)^\frac12\\
        &\, \times \frac{1}{\sqrt{2N\delta}} \Nrm{V^\varepsilon_N}{L^1}^\frac12  \Nrm{\sh_t}{L^\infty_x L^2_y}\Nrm{\sh_t}{\rm HS}\Nrm{(\cN+1)\Psi}{}\\
        \le&\,  \delta \inprod{\Psi}{N^{-1}\cV\, \Psi}+ C_{\delta}\Nrm{V_N^\varepsilon}{L^1}\sfC_0(t) \Nrm{(\cN+1)\Psi}{}^2.
    \end{align*}
    Inequality for $\com{\cN, \cD}$ and Inequality \eqref{est:quartic_time-derivative_term_bound} follow from similar arguments.
\end{proof}

\subsubsection{Linear Terms} To estimate the linear terms, we need to impose conditions on $\phi_{t}$ and $k_t$ so that the order one terms vanishes leaving only lower order terms.

\begin{prop}\label{prop:estimates_for_linear_terms}
    Suppose $\cL_{\rm HFB}$ is as defined in Expression \eqref{def:normal_ordered_linear_terms} where $\phi_{t}$ satisfies Equation \eqref{eq:modified_Gross--Pitaevskii} and $k_{t}$ is defined by Expression \eqref{def:pair_excitation_function}. Then there exists a univerisal $C>0$  such that we have the following operator estimates
        \begin{align}
            \pm \cL_{\rm HFB},  [\cN, \cL_{\rm HFB}] \le&\, C \Nrm{V^\varepsilon_N}{L^1}\sfC_{0}(t) N^{-\frac12} (\cN+1)\ , \label{est:linear_term_bound}\\
            \pm \varepsilon\dot \cL_{\rm HFB} \le&\, C \Nrm{V^\varepsilon_N}{L^1}\sfC_{0}(t) N^{-\frac12} (\cN+1)\label{est:linear_term_time-derivative_bound} \ ,
        \end{align}
    where $\sfC_0(t)$ is defined as in \eqref{equ:common quantity}.
\end{prop}

\begin{proof}
    Let us write $\vect{S}_{\rm HFB}(\phi):= -i\varepsilon\,\bd_t\phi+h_{\mathrm{HFB}}(\phi)$.
    Using Lemma \ref{lem:bogoliubov_conjugation}, we get that
    \begin{align*}
    \cL_{\rm HFB}
    = \sqrt{N} \(a(c\, \vect{S}_{\rm HFB}(\phi))+a^\ast(\sh\, \conj{\vect{S}_{\rm HFB}(\phi)})\)\ .
    \end{align*}
    By Lemma \ref{lem:estimates_for_creation_annilihation_operators}, we have that
    \begin{align*}
        \n{\inprod{\Psi}{\cL_{\rm HFB}\, \Psi}} \le \sqrt{N} \(1+\Nrm{p}{\rm HS}+\Nrm{\sh}{\rm HS}\)\Nrm{\vect{S}_{\rm HFB}(\phi)}{L^2_x}\inprod{\Psi}{(\cN+1)\,\Psi}\ .
    \end{align*}
    Hence it remains to estimate $\vect{S}_{\rm HFB}(\phi)$ in $L^2$. Notice that
    \begin{subequations}
        \begin{align}
            \vect{S}_{\rm HFB}(\phi)
            =&\,-i\varepsilon\,\bd_t\phi+\(-\tfrac{\varepsilon^2}{2}\lapl +(V_N^\varepsilon f_{N, l}^\varepsilon\ast|\phi|^2)\)\phi \label{def:phi_HFB_operator_eq_term}\\
            &\, +\intd V_N^\varepsilon(x-y)\,\(w_{N,l}^\varepsilon(x-y)\phi(x)\phi(y)+\tfrac{1}{N}k(x, y)\)\conj{\phi(y)}\dd y \label{def:phi_HFB_operator_cancellation_term}\\
            &\, +\frac{1}{2\,N}\intd V_N^\varepsilon(x-y)\,r_2(x, y)\,\conj{\phi(y)}\dd y \label{def:phi_HFB_operator_error_term1}\\
            &\, +\frac{1}{2\, N}\intd V_N^\varepsilon(x-y)\, p_2(y, y)\,\phi(x)\dd y \label{def:phi_HFB_operator_error_term2}\\
            &\, +\frac{1}{2\, N}\intd V_N^\varepsilon(x-y)\, p_2(x, y)\,\phi(y)\dd y\ . \label{def:phi_HFB_operator_error_term3}
        \end{align}
    \end{subequations}
    By our hypotheses, Term \eqref{def:phi_HFB_operator_eq_term} and Term \eqref{def:phi_HFB_operator_cancellation_term} are zeroes.
    To estimate Term \eqref{def:phi_HFB_operator_error_term1}--\eqref{def:phi_HFB_operator_error_term3}, we use the following $L^p$ estimates
    \begin{align*}
        \Nrm{\intd V_N^\varepsilon(x-y)\,g(x, y)\,h(y)\dd y}{L^p_x}
        \le&\, \Nrm{V^\varepsilon_N}{L^1}\sup_z\Nrm{g(\cdot, z)}{L^p_x}\Nrm{h}{L^\infty_x}
    \end{align*}
    and
    \begin{align*}
        \Nrm{\intd V_N^\varepsilon(x-y)\,g(y, y)\,h(x)\dd y}{L^p_x}\le \Nrm{V^\varepsilon_N}{L^1}\Nrm{\diag g}{L^p_x}\Nrm{h}{L^\infty_x}\ .
    \end{align*}
    Hence it follows
    \begin{align*}
        \Nrm{\vect{S}_{\rm HFB}(\phi)}{L^2_x} \le \frac{1}{2N}\Nrm{V^\varepsilon_N}{L^1}\(2\Nrm{p_2}{L^\infty_xL^2_y}+\Nrm{r_2}{L^\infty_xL^2_y}\)\Nrm{\phi}{L^\infty_x}\ .
    \end{align*}
    This completes the proof of Inequality \eqref{est:linear_term_bound}.
    The other inequalities follow from similar arguments.
\end{proof}

\subsubsection{Quadratic Terms} Finally, we consider the quadratic term. Let us split the analysis into two parts: the part of the quadratic operator that commutes with $\cN$ and the part that does not. Let us start with the following proposition which bounds the part of the quadratic term that commutes with $\cN$.
\begin{prop}\label{prop:estimates_for_diagonal_quadratic_terms}
    Consider the part of the operator $\cQ_{\rm HFB}$ defined by Expression \eqref{def:normal_ordered_HFB_quadratic_operator} that commutes with $\cN$ without the kinetic energy, i.e.
    \begin{align*}
        \cQ_{\rm D}=\dG(V_{\Gamma^\varepsilon} +\sfX_{\Gamma^\varepsilon}+\tfrac12 d_{\rm HFB}+\tfrac12 d_{\rm HFB}^\top)
    \end{align*}
    where $V_{\Gamma^\varepsilon}, \sfX_{\Gamma^\varepsilon}$,  and $d_{\rm HFB}$ are as defined in Proposition \ref{prop:fluctuation_Hamiltonian}. More precisely, we consider the splitting
    \begin{align*}
        \cQ_{\rm D} =&\, \cQ_{\rm D, time}+\cQ_{\rm D, kinetic}+\cQ_{\rm D, potential}
    \end{align*}
    with
        \begin{align*}
            &\cQ_{\rm D, time} :=\, \tfrac12\dG\(i\varepsilon\,\bd_t p\, (\id +p)-i\varepsilon\,\bd_t \sh\,\conj{\sh}\)+\mathrm{transpose}\, ,\\
            &\cQ_{\rm D, kinetic} :=\, \tfrac12\dG\( \tfrac{\varepsilon^2}{2}\lapl-(\id+p)\tfrac{\varepsilon^2}{2} \lapl (\id +p)-\sh\, (\conj{\tfrac{\varepsilon^2}{2} \lapl \sh})\)+\mathrm{transpose}\ , \\
            &\cQ_{\rm D, potential} :=\, \cQ_{\rm D}-\cQ_{\rm D, time}-\cQ_{\rm D, kinetic} =: \tfrac12\dG(d_{\rm HFB, potential})+\mathrm{transpose}\ ,
        \end{align*}
        where we have used Identity \eqref{eq:hyperbolic_trig_identities} to simplify $\cQ_{\rm D, kinetic}$.
    Then there exists $C>0$ such that we have the following operator estimates
        \begin{align}
            \pm \cQ_{\rm D, time} \le&\, C  \sfC_{0}(t) \, \cN\ , \label{est:quadratic_diagonal_time-derivative}\\
            \pm \varepsilon\dot\cQ_{\rm D, time} \le&\, C\sfC_{0}(t) \, \cN\ , \label{est:time-derivative_quadratic_diagonal_time-derivative}\\
            \pm \cQ_{\rm D, potential} \le&\, C\Nrm{V^\varepsilon_N}{L^1}\,\sfC_0(t) \, \cN\ , \label{est:quadratic_diagonal_potential}\\
            \pm \varepsilon\dot \cQ_{\rm D, potential}\le&\, C \Nrm{V^\varepsilon_N}{L^1}\sfC_0(t) \, \cN \label{est:time-derivative_quadratic_diagonal_potential}\ ,
        \end{align}
    where $\sfC_0(t)$ is defined as in \eqref{prop:estimates_for_cubic_terms}.
    Moreover,  for any $0<\delta\le 1$ there exists a $C_\delta>0$  such that  we have the following operator estimates
    \begin{align}
        \pm \cQ_{\rm D, kinetic} \le&\, \delta \cK +  C_{\delta}\Nrm{V_N^\varepsilon}{L^1}\sfC_0(t)\,\cN  \label{est:kinetic_quadratic_diagonal_term}\ , \\
        \pm \varepsilon\dot\cQ_{\rm D, kinetic} \le&\, \delta \cK + C_{\delta}\Nrm{V_N^\varepsilon}{L^1}\sfC_0(t)\,\cN \label{est:time_derivative_kinetic_quadratic_diagonal_term}\ .
    \end{align}
\end{prop}

\begin{proof}
 For $\cQ_{\rm D, time}$, by Lemma \ref{lem:estimates_for_creation_annilihation_operators}, we have that
    \begin{align*}
        \n{\inprod{\Psi}{\cQ_{\rm D, time}\,\Psi}} \le&\, C \(\Nrm{\varepsilon\dot p_t (\id+p_t)}{\mathrm{op}}+\Nrm{\varepsilon\,\dot \sh_t\,\conj{\sh}_t}{\rm op}\)\inprod{\Psi}{\cN\, \Psi}\\
        \le&\, C \(\Nrm{\varepsilon\dot p_t}{\mathrm{HS}}\(1+\Nrm{p_t}{\rm HS}\)+\Nrm{\varepsilon\,\dot \sh_t}{\rm HS}\Nrm{\sh_t}{\rm HS}\)\inprod{\Psi}{\cN\, \Psi}\ .
    \end{align*}
    A similar argument applies for $\varepsilon\dot\cQ_{\rm D, time}$. This completes the proof of Inequalities \eqref{est:quadratic_diagonal_time-derivative}--\eqref{est:time-derivative_quadratic_diagonal_time-derivative}.

  For $\cQ_{\rm D, potential}$, by Lemma \ref{lem:estimates_for_creation_annilihation_operators} again, we have that
    \begin{align*}
        \n{\inprod{\Psi}{\cQ_{\rm D, potential}\, \Psi}} \le \(\Nrm{V_{\Gamma^\varepsilon}}{\rm op}+\Nrm{\sfX_{\Gamma^\varepsilon}}{\rm op}+\Nrm{d_{\rm HFB, potential}}{\rm op}\)\inprod{\Psi}{\cN\, \Psi}\ .
    \end{align*}
    It is straightforward to check that
    \begin{align}\label{est:operator_norm_bound_effective_potential_and_exchange_term}
        \norm{V_{\Gamma^\varepsilon}}_{\rm op}+\norm{\sfX_{\Gamma^\varepsilon}}_{\rm op} \le  2\Nrm{V^\varepsilon_N}{L^1}\(\Nrm{\phi_t}{L^\infty_x}^2+\tfrac{1}{N}\Nrm{\sh_t}{L^\infty_xL^2_y}^2\)\ .
    \end{align}
    Hence it remains to estimate the operator norm of $d_{\rm HFB, potential}$.
    Observe we have the estimate
    \begin{align*}
        \Nrm{\ch(k_t)\, \widetilde m_t \,\conj{\sh(k_t)}}{\rm op} \le \(1+\Nrm{p_t}{\rm HS}\)\Nrm{\widetilde m_t}{\rm op}\Nrm{\sh_t}{\rm HS}\ .
    \end{align*}
    To estimate the operator norm of $\widetilde m_t$, notice we have
    \begin{align}\label{est:operator_norm_bound_of_tilde_m}
        \Nrm{\widetilde m_t f}{L^2_x}
        \le \Nrm{V^\varepsilon_N}{L^1}\(\Nrm{\phi_t}{L^\infty_x}^2+\tfrac{1}{2\,N}\Nrm{s_{2, t}}{L^\infty_{x, y}}\)\Nrm{f}{L^2_x}\ .
    \end{align}
    With this, it follows that
    \begin{align*}
        \Nrm{d_{\rm HFB, potential}}{\rm op}
         \le C\Nrm{V^\varepsilon_N}{L^1}\sfC_0(t) \ .
    \end{align*}
    A similar argument applies for $\varepsilon\dot\cQ_{\rm D, potential}$. This completes the proof of Inequalities \eqref{est:quadratic_diagonal_potential}--\eqref{est:time-derivative_quadratic_diagonal_potential}.

    Let us consider the kinetic part of the quadratic operator.  Notice that
    \begin{subequations}
        \begin{align}
            &\dG\(\tfrac{\varepsilon^2}{2}\lapl-(\id+p)\tfrac{\varepsilon^2}{2} \lapl (\id +p)-\sh\, (\conj{\tfrac{\varepsilon^2}{2} \lapl \sh})\) \notag\\
            &=\frac12\intd a^\ast(\varepsilon \grad_x p_{x})a(\varepsilon \grad_x p_{x})\dd x+\frac12\intd a^\ast(\varepsilon \grad_x r_{x})a(\varepsilon \grad_x r_{x})\dd x \label{def:quadratic_kinetic_Term1}\\
            &\quad +\frac12\intd \varepsilon \grad_x a^\ast_x a(\varepsilon \grad_x p_{x})\dd x + \frac12\intd a^\ast(\varepsilon \grad_x p_{x})\varepsilon \grad_x a_x \dd x \label{def:quadratic_kinetic_Term2}\\
            &\quad +\frac12\intd a^\ast(\varepsilon \grad_x r_{x})a(\varepsilon \grad_x k_{x})\dd x+\frac12\intd a^\ast(\varepsilon \grad_x k_{x})a(\varepsilon \grad_x r_{x})\dd x \label{def:quadratic_kinetic_Term3}\\
            &\quad +\frac12\intd a^\ast(\varepsilon \grad_x k_{x})a(\varepsilon \grad_x k_{x})\dd x\ .\label{def:quadratic_kinetic_Term4}
        \end{align}
    \end{subequations}
    By Lemmas \ref{lem:quadratic_operator_bounds_with_derivative}--\ref{lem:quadratic_operator_bounds_with_derivative2}, we see that
    \begin{align*}
        \eqref{def:quadratic_kinetic_Term1}+\eqref{def:quadratic_kinetic_Term2} &\le\, \delta \cK + C_ \delta\(\Nrm{\varepsilon\grad p}{\rm HS}^2+\Nrm{\varepsilon\grad r}{\rm HS}^2\)\cN\ ,\\
        \eqref{def:quadratic_kinetic_Term3} &\le\, \delta \cK+C_\delta \(\Nrm{k}{\rm HS}^2\Nrm{\varepsilon\grad r}{\rm HS}^2+\Nrm{[\varepsilon\grad, k]}{\rm HS}\Nrm{\varepsilon\grad r}{\rm HS}\) \cN\ ,\\
        \eqref{def:quadratic_kinetic_Term4} &\le\, \Nrm{k(-\varepsilon^2\lapl)\conj{k}}{\rm HS}\cN\ .
    \end{align*}
 This completes the proof of Inequality \eqref{est:kinetic_quadratic_diagonal_term}. The proof of Inequality \eqref{est:time_derivative_kinetic_quadratic_diagonal_term} is similar. In fact, we have
 \begin{align*}
    \n{\varepsilon\bd_t\(\eqref{def:quadratic_kinetic_Term1}+\eqref{def:quadratic_kinetic_Term2} \)}
    & \le\, \delta \cK + C_ \delta\(\Nrm{\varepsilon\grad p}{\rm HS}\Nrm{\varepsilon^2\grad \dot p}{\rm HS}+\Nrm{\varepsilon\grad r}{\rm HS}\Nrm{\varepsilon^2\grad \dot r}{\rm HS}\)\cN\ ,\\
    \n{\varepsilon\bd_t\eqref{def:quadratic_kinetic_Term3}}
    & \le\,  \delta \cK+C_\delta \(\Nrm{k}{\rm HS}^2\Nrm{\varepsilon^2\grad \dot r}{\rm HS}^2+\Nrm{[\varepsilon\grad, k]}{\rm HS}\Nrm{\varepsilon^2\grad \dot r}{\rm HS}\) \cN \\
    &\qquad\quad +\, C_\delta \(\Nrm{\varepsilon\dot k}{\rm HS}^2\Nrm{\varepsilon\grad r}{\rm HS}^2+\Nrm{[\varepsilon\grad, \varepsilon\dot k]}{\rm HS}\Nrm{\varepsilon\grad r}{\rm HS}\) \cN\ ,\\
    \n{\varepsilon\bd_t\eqref{def:quadratic_kinetic_Term4}} &\le C\Nrm{\varepsilon \dot k(-\varepsilon^2\lapl)\conj{k}}{\rm HS}\cN\ ,
 \end{align*}
 which yield the desired estimates.
\end{proof}

To handle the off-diagonal operator $\cQ_{\rm OD}$, we need to first rewrite Expression \eqref{def:off-diagonal_entry_of_QHFB}. Observe we could write
\begin{align}\label{eq:expanding_off-diagonal_term}
    &\(\vect{\widetilde S}(\sh_t)-(\id + p_t)\,\widetilde m\)\\
    &=\, i\varepsilon \bd_t \sh +\tfrac{\varepsilon^2}{2}\lapl_x \sh + \tfrac{\varepsilon^2}{2}\lapl_y \sh
    - V_N^\varepsilon(x-y)(1-w_{N, l}^\varepsilon(x-y))\phi(x)\phi(y)+ R,\notag
\end{align}
where $R$ denotes the remainder terms. Moreover, using the fact that $\sh(k)=k+r(k)$, we could further express the right hand side of Equation \eqref{eq:expanding_off-diagonal_term} as
\begin{subequations}
    \begin{align}
        \mathrm{RHS} \eqref{eq:expanding_off-diagonal_term} =&\,  \tfrac{\varepsilon^2}{2}\lapl_x k + \tfrac{\varepsilon^2}{2}\lapl_y k - V_N^\varepsilon(x-y)f_{N,l}^\varepsilon(x-y)\phi(x)\phi(y)- \widetilde R \notag\\
        =&\, -N\(-\varepsilon^2\lapl f_{N,l}^\varepsilon(x-y)+U_N(x-y)f_{N,l}^\varepsilon(x-y)\)\phi(x)\phi(y)\label{eq:zero-energy_scattering_in_off-diagonal}\\
        &\, -\tfrac12 N\varepsilon\grad_x w_{N,l}^\varepsilon(x-y)\cdot \varepsilon\grad\phi(x)\phi(y)\\
        &\, -\tfrac12 N\varepsilon\grad_y w_{N,l}^\varepsilon(x-y)\cdot \varepsilon\grad\phi(y)\phi(x)\\
        &\, -\tfrac12 \varepsilon\grad_x\cdot\(N w_{N,l}^\varepsilon(x-y)\varepsilon\grad\phi(x)\phi(y)\)\\
        &\, -\tfrac12 \varepsilon\grad_y\cdot\(N w_{N,l}^\varepsilon(x-y)\varepsilon\grad\phi(y)\phi(x)\)-\widetilde R
    \end{align}
\end{subequations}
where
\begin{align*}
    \widetilde R(x, y) =&\, -i\varepsilon\bd_t \sh(x, y)-\tfrac{\varepsilon^2}{2}\lapl_x r(x, y) - \tfrac{\varepsilon^2}{2}\lapl_y r(x, y)\\
    &\, +\(V_{\Gamma^\varepsilon}(x)+V_{\Gamma^\varepsilon}(y)\)\sh(x, y)+(\sfX_{\Gamma^\varepsilon}\sh)(x, y)+(\sh\sfX_{\Gamma^\varepsilon})(x, y)\\
    &\, +\(p\widetilde m \)(x, y)+\tfrac{1}{2N}V_N^\varepsilon(x-y)r_{2}(x, y)\ .
\end{align*}
Using the fact that $f_{N,l}^\varepsilon = 1-w_{N,l}^\varepsilon$ and $f_{N,l}^\varepsilon$ satisfies Equation \eqref{def:correlation_structure_neumann_problem}, we have
\begin{subequations}
    \begin{align*}
        \eqref{eq:zero-energy_scattering_in_off-diagonal}=&\, -N E_{N, \ell}^\varepsilon f_{N, \ell}^\varepsilon(x-y)\chi_{\{\n{x-y}\le \ell\}}\phi(x)\phi(y)\ .
    \end{align*}
\end{subequations}

\begin{prop}\label{prop:estimates_for_off-diagonal_quadratic_terms}
    Consider the splitting $\cQ_{\rm OD}=\cQ_{\rm OD,\, s}+\cQ_{\rm OD,\, r}$ with
    \begin{align*}
        \cQ_{\rm OD,\, s}:=&\, -\tfrac12\intdd \(\vect{\widetilde S}(\sh)-(\id + p)\,\widetilde m\)(x, y)\ a^\ast_x a^\ast_y\dd x \d y+\mathrm{h.c.}\ ,\\
        \cQ_{\rm OD,\, r}:=&\,-\tfrac12 \intdd \(\vect{\widetilde S}(\sh)-(\id + p)\,\widetilde m\)p(x, y)\ a^\ast_x a^\ast_y\dd x \d y\\
        &\, +\tfrac12 \intdd \(\vect{\widetilde W}(p)+\sh\, \conj{\widetilde m}\)\sh(x, y)\, a^\ast_x a^\ast_y\dd x \d y\ +\mathrm{h.c.}\ , \notag
    \end{align*}
    where $\phi$ satisfies Equation \eqref{eq:modified_Gross--Pitaevskii} and $k$ is defined by Expression \eqref{def:pair_excitation_function}.
    By the calculation preceding the proposition, we could further split $\cQ_{\rm OD,\, s}$ as follows
    \begin{align*}
        \cQ_{\rm OD,\, s} =&\, \cQ_{\rm OD,\, s}^{(1)}+\cQ_{\rm OD,\, s}^{(2)}-\cQ_{\rm OD,\, s}^{(3)}+\cQ_{\rm OD,\, s}^{(4)}+\mathrm{h.c.}
    \end{align*}
    where
        \begin{align*}
            \cQ_{\rm OD,\, s}^{(1)} :=&\, \tfrac12NE_{N, \ell}^\varepsilon\int_{\n{x-y}\le \ell}f_{N, \ell}^\varepsilon(x-y)\phi(x)\phi(y)\, a^\ast_x a^\ast_y\dd x\d y\ ,\\
            \cQ_{\rm OD,\, s}^{(2)} :=&\, \tfrac12 N \intdd \varepsilon\grad_x w_{N,l}^\varepsilon(x-y)\cdot\varepsilon\grad\phi(x)\phi(y)\, a^\ast_x a^\ast_y\dd x \d y\ ,\\
            \cQ_{\rm OD,\, s}^{(3)} :=&\, \tfrac12N \intdd w_{N,l}^\varepsilon(x-y)\varepsilon \grad \phi(x)\phi(y)\, \varepsilon \grad_x a^\ast_x a^\ast_y\dd x \d y\ ,\\
            \cQ_{\rm OD,\, s}^{(4)} :=&\, \tfrac12\intdd \widetilde R(x, y)\, a^\ast_x a^\ast_y\dd x \d y\ .
        \end{align*}
        Then there exists a universal $C>0$ such that for every $0<\delta\le 1$  we have the following operator estimates
        \begin{align}
            \pm  \cQ_{\rm OD,\, s}^{(1)\, \sharp},\ [\cN, \cQ_{\rm OD,\, s}^{(1)}]^\sharp \le&\, C  \sfC_0(t)\,(\cN+1)\ , \label{est:off-diagonal_singular_eigenvalue-term}\\
            \pm \varepsilon\dot\cQ_{\rm OD,\, s}^{(1)\sharp} \le&\, C\sfC_0(t)\,(\cN+1)\ , \label{est:time-derivative_off-diagonal_singular_eigenvalue}\\
            \pm  \cQ_{\rm OD,\, s}^{(2)\, \sharp},\  \cQ_{\rm OD,\, s}^{(3)\, \sharp},\  [\cN, \cQ_{\rm OD,\, s}^{(2)}]^\sharp,\ [\cN, \cQ_{\rm OD,\, s}^{(3)}]^\sharp\le&\, \delta \cK+  C_{\delta}\Nrm{V_N^\varepsilon}{L^1}\sfC_0(t)\, (\cN+1)\ , \\
            \pm  \varepsilon\dot\cQ_{\rm OD,\, s}^{(2)\, \sharp},\  \varepsilon\dot\cQ_{\rm OD,\, s}^{(3)\, \sharp}\le&\, \delta \cK+ C_{\delta}\Nrm{V_N^\varepsilon}{L^1}\sfC_0(t)\, (\cN+1)\ .
        \end{align}

        For the operators $\cQ_{\rm OD,\, s}^{(4)}$ and $\cQ_{\rm OD,\, r}$, as in the previous proposition, we shall consider the further splitting (of course with respect to the operators $aa$ and  $a^\ast a^\ast$)
        \begin{align*}
            \cQ_{\rm OD,\, s}^{(4)}=&\, \cQ_{\rm OD,\, s,\, time}^{(4)}+\cQ_{\rm OD,\, s,\, kinetic}^{(4)}+\cQ_{\rm OD,\, s,\, potential}^{(4)}\ ,\\
            \cQ_{\rm OD,\, r}=&\, \cQ_{\rm OD,\, r,\, time}+\cQ_{\rm OD,\, r,\, kinetic}+\cQ_{\rm OD,\, r,\, potential}\ .
        \end{align*}
        Then for every $0<\delta\le 1$ there are a universal $C>0$ and $C_{\delta}>0$ such that
        \begin{align*}
\pm \cQ_{\rm OD,\, s,\, time}^{(4)\sharp},\, \cQ_{\rm OD,\, r,\, time}^{\sharp},\, [\cN, \cQ_{\rm OD,\, s,\, time}^{(4)}]^\sharp,  [\cN, \cQ_{\rm OD,\, r,\, time}]^{\sharp}  \le&\,  C  \sfC_0(t)(\cN+1)\ ,\\
\pm \cQ_{\rm OD,\, s,\, potential}^{(4)\sharp}, \cQ_{\rm OD,\, r,\, potential}^{\sharp}, \le&\, C\Nrm{V^\varepsilon_N}{L^1}\sfC_0(t) (\cN+1)\ , \\
[\cN, \cQ_{\rm OD,\, s,\, potential}^{(4)}]^\sharp+[\cN, \cQ_{\rm OD,\, r,\, potential}]^{\sharp} \le&\, C\Nrm{V^\varepsilon_N}{L^1}\sfC_0(t) (\cN+1)\ ,\\
\pm \cQ_{\rm OD,\, s,\, kinetic}^{(4)\sharp}, \cQ_{\rm OD,\, r,\, kinetic}^{\sharp}, [\cN, \cQ_{\rm OD,\, s,\, kinetic}^{(4)}]^\sharp, [\cN, \cQ_{\rm OD,\, r,\, kinetic}]^{\sharp} \le&\, \delta\cK+C_{\delta}\sfC_0(t)\, (\cN+1)\ ,\\
\pm \varepsilon\dot\cQ_{\rm OD,\, s,\, time}^{(4)\sharp}, \varepsilon\dot\cQ_{\rm OD,\, r,\, time}^{\sharp}  \le&\,  C  \sfC_0(t)(\cN+1)\ ,\\
 \pm \varepsilon\dot \cQ_{\rm OD,\, s,\, potential}^{(4)\sharp}, \varepsilon\dot\cQ_{\rm OD,\, r,\, potential}^{\sharp}
    \le& \,  C\Nrm{V^\varepsilon_N}{L^1}\sfC_0(t) (\cN+1)  \ ,\\
  \pm \varepsilon\dot\cQ_{\rm OD,\, s,\, kinetic}^{(4)\sharp}, \varepsilon\dot\cQ_{\rm OD,\, r,\, kinetic}^{\sharp} \le&\, \delta\cK+C_{\delta}\sfC_0(t)\, (\cN+1) \ .
        \end{align*}
\end{prop}

\begin{proof}
    To prove Inequality~\eqref{est:off-diagonal_singular_eigenvalue-term}, we apply Inequality \eqref{eq:neumann_gs_eigenvalue_approximation},  $\norm{f^\varepsilon_{N, \ell}}_{L^\infty_x}\le 1$, and  H\"older's inequality, to see that
    \begin{align*}
        \n{\inprod{\Psi}{\cQ_{\rm OD,\, s}^{(1)}\Psi}} \le&\, \frac{3\asc_0^{\mu}}{\ell^3} \int_{\n{x-y}\le \ell} \n{\phi(x)}\n{\phi(y)}\Nrm{a_x \Psi}{}\Nrm{a_y^\ast \Psi}{}\dd x\d y
        \le\, C \Nrm{\phi_t}{L^\infty_x}^2\norm{(\cN+1)^\frac12 \Psi}^2\ .
    \end{align*}
    The other inequalities involving $\cQ_{\rm OD,\, s}^{(1)}$ are handled similarly.

    For $\cQ_{\rm OD,\, s}^{(2)}$, notice we have that
    \begin{align*}
        \cQ_{\rm OD,\, s}^{(2)} =&\,    \tfrac12N \intdd w_{N,l}^\varepsilon(x-y)\varepsilon\grad\phi(x)\cdot\varepsilon\grad \phi(y)\, a^\ast_x a^\ast_y\dd x \d y\\
        &\, +\tfrac12 N \intdd w_{N,l}^\varepsilon(x-y)\phi(y)\varepsilon\grad\phi(x)\cdot\varepsilon\grad_y  a^\ast_y a^\ast_x\dd x \d y\ .
    \end{align*}
    Then it follows that
    \begin{align*}
        \n{\inprod{\Psi}{\cQ_{\rm OD,\, s}^{(2)}\Psi}}
        \le&\, CC_0(t)\norm{(\cN+1)^\frac12 \Psi}^2
        + CC_0(t)\norm{(\cN+1)^\frac12 \Psi}\norm{\cK^\frac12 \Psi}\\
        \le&\, \delta\norm{\cK^\frac12 \Psi}^2+ C_{\delta}C_0(t)\norm{(\cN+1)^\frac12 \Psi}^2\ .
    \end{align*}
    $\cQ_{\rm OD,\, s}^{(3)}$ and the related operators is handled in a similar manner.

    For $\cQ_{\rm OD,\, s}^{(4)}$, by Lemma \ref{lem:quadratic_operator_bounds_with_derivative}, we have the estimates
    \begin{align*}
        \pm\cQ_{\rm OD,\, s,\, time}^{(4)}
        \le&\, C \Nrm{\varepsilon\dot \sh}{\rm HS}(\cN+1)\ , \\
        \pm \cQ_{\rm OD,\, s,\, kinetic}^{(4)}\le&\,  \delta\cK + C\delta^{-1}\Nrm{\varepsilon \grad_x r_t}{\rm HS}^2(\cN+1)\ ,\\
        \pm\cQ_{\rm OD,\, s,\, potential}^{(4)}
        \le&\, C \(\Nrm{V_{\Gamma^\varepsilon}}{\rm op}+\Nrm{\sfX_{\Gamma^\varepsilon}}{\rm op}\)\Nrm{\sh_t}{\rm HS}(\cN+1)\\
        &\, + C\Nrm{\widetilde m_t}{\rm op}\Nrm{p_t}{\rm HS}(\cN+1) + CN^{-1}\Nrm{V^\varepsilon_N}{L^1}\Nrm{r_{2, t}^{\varepsilon}}{L^\infty_{x, y}}(\cN+1)\\
        \le&\,  C\Nrm{V^\varepsilon_N}{L^1}\sfC_0(t) (\cN+1)
    \end{align*}
    where the last inequality follows from Inequalities~\eqref{est:operator_norm_bound_effective_potential_and_exchange_term}--\eqref{est:operator_norm_bound_of_tilde_m}.

    Now consider $\cQ_{\rm OD,\, r}$. The first part of $\cQ_{\rm OD,\, r}$ follows from the argument of $\cQ_{\rm OD,\, s}$ and the fact that $p_{N, t}^\varepsilon, \varepsilon\dot p_{N, t}^\varepsilon,$ and $\varepsilon\grad p_{N, t}^\varepsilon$ are bounded operators. For the second part, we again split the operator into three parts (time, kinetic, and potential) then estimate each operator separately.
\end{proof}

\subsection{Proof of Proposition~\ref{prop:bounds_on_fluctuation_Hamiltonian}}

\begin{proof}[Proof of Proposition \ref{prop:bounds_on_fluctuation_Hamiltonian}]
    We write
    \begin{align*}
        \widetilde\cH_{\rm fluc}= \cH_N+\cL_{\rm HFB} + \cQ_{\rm D}+\cQ_{\rm OD}+ \cC+ \cD_1+\cD_2.
    \end{align*}
    Then, by Propositions \ref{prop:estimates_for_cubic_terms}--\ref{prop:estimates_for_off-diagonal_quadratic_terms}, we have the following operator inequalites
    \begin{align*}
        \widetilde\cH_{\rm fluc} \le& (1+\delta)\cH_N + C_{\delta}\Nrm{V_N^\varepsilon}{L^1}\sfC_0(t) \( \,\frac{\cN^2}{N}+(\cN+1)\)\ .
    \end{align*}
    Then, by the estimates of $k_t$ in Lemmas \ref{lem:pair_excitation_estimates}--\ref{lem:pair_excitation_estimates_time-derivative}, there exists $M$ such that
    \begin{align*}
        \sfC_{0}(t)\le C\ve^{-M}\lra{t}^{M}\ .
    \end{align*}
    In particular, with $\Nrm{V_N^\varepsilon}{L^1}=\la$ we obtain
    \begin{equation*}
        \widetilde\cH_{\rm fluc} \le (1+\delta)\cH_N + C_{\delta}\lambda \ve^{-M}\lra{t}^{M} \( \,\frac{\cN^2}{N}+(\cN+1)\)\ .
    \end{equation*}
    Similarly, we have that
    \begin{equation*}
        \widetilde\cH_{\rm fluc} \ge (1-\delta)\cH_N -  C_{\delta}\lambda \ve^{-M}\lra{t}^{M} \( \,\frac{\cN^2}{N}+(\cN+1)\)\ .
    \end{equation*}
    This completes the proof of Inequalities \eqref{est:lower_bound_for_fluctuation_Hamiltonian}--\eqref{est:upper_bound_for_fluctuation_Hamiltonian} if we set $\delta = \frac12$. Moreover, Inequalities \eqref{est:bound_for_commutator_of_fluctuation_Hamiltonian_with_number}--\eqref{est:bound_for_time-derivative_of_fluctuation_Hamiltonian} also follows from Propositions \ref{prop:estimates_for_cubic_terms}--\ref{prop:estimates_for_off-diagonal_quadratic_terms}.
\end{proof}

\subsection{The growth of number and energy of the fluctuations} In this section, we obtain bounds (not necessarily uniform in $N$ and $\varepsilon$) on the growth of the expectation values of the number operator and the many-body Hamiltonian with respect to the fluctuation dynamics.

\begin{lemma}\label{lem:fluctuation_number_squared}
    Suppose $\cU_{\mathrm{fluc}}(t; 0)$ is as defined in Equation~\eqref{def:fluctuation_dynamics}. Then there exists $C>0$ independent of $N$ and $\varepsilon$ such that
    \begin{align*}
        \cU_{\mathrm{fluc}}(t; 0)^\ast\, \cN^2\, \cU_{\mathrm{fluc}}(t; 0) \le C\(N\,\cU_{\mathrm{fluc}}(t; 0)^\ast\, \cN\, \cU_{\mathrm{fluc}}(t; 0)+N(\cN+1)+(\cN+1)^2\)\ .
    \end{align*}
\end{lemma}
\begin{proof}
    The proof follows from Lemmas~\ref{lem:properties_of_weyl_operator} and \ref{lem:bogoliubov_conjugation_operator_inequality}. See \cite[Proposition 4.2]{benedikter2015quantitative}.
\end{proof}

\begin{prop}\label{prop:bounds_on_growth_of_fluctuation}
    Suppose $\phi_t=\phi_{N, t}^\varepsilon$ is a solution to Equation \eqref{eq:modified_Gross--Pitaevskii} with initial data $\weight{\varepsilon\grad}^4\phi^\init\in L^2(\R^3)$ holds uniformly in $\varepsilon$. Suppose $k_{t}=k_{N, t}^\varepsilon$ is given by Equation \eqref{def:pair_excitation_function} with $\ell=\varepsilon^4$ for the necessary regimes.
    Then there exists $C, M>0$, such that we have the following operator bounds
    \begin{align}\label{equ:operator bounds,HN}
        \inprod{\Psi_{\mathrm{fluc}, t}}{\cN\, \Psi_{\mathrm{fluc}, t}}+ \inprod{\Psi_{\mathrm{fluc}, t}}{\cH_N\, \Psi_{\mathrm{fluc}, t}} \le C \exp(C\lambda \ve^{-M}\lra{t}^{M})\ .
    \end{align}
\end{prop}

\begin{proof}
    Recall from Expression~\eqref{def:renormalized_fluctuation_Hamiltonian} that $\widetilde\cH_{\mathrm{fluc}}$ and $\cH_{\mathrm{fluc}}$ differs by a scalar quantity. In particular, their dynamics differs by a phase factor. Hence, from the perspective of the time-independent observable, the two dynamics are equal, i.e.
    \begin{align*}
        \inprod{\Xi_0}{\cU_{\mathrm{fluc}}^\ast(t; 0)\cN\cU_{\mathrm{fluc}}(t; 0)\Xi_0}=\inprod{\Xi_0}{\widetilde\cU_{\mathrm{fluc}}^\ast(t; 0)\cN\widetilde\cU_{\mathrm{fluc}}(t; 0)\Xi_0}
    \end{align*}
    where $\widetilde\cU_{\mathrm{fluc}}(t; 0)$ is the dynamics generated by $\widetilde\cH_{\mathrm{fluc}}$.

    Using Inequality~\eqref{est:upper_bound_for_fluctuation_Hamiltonian} established in Proposition~\ref{prop:bounds_on_fluctuation_Hamiltonian}, we immediately arrive at
    \begin{align}\label{est:energy_bounded_by_fluc_energy}
        \cH_N \lesssim \widetilde\cH_{\mathrm{fluc}}(t)+ \lambda\ve^{-M}\lra{t}^{-M} \(\frac{\cN^2}{N}+\cN+1\)\ .
    \end{align}
    In particular, by Lemma~\ref{lem:fluctuation_number_squared}, we have
    \begin{align}\label{est:energy_bounded_by_fluc_energy2}
        \inprod{\Psi_{\mathrm{fluc}, t}}{\cH_N\Psi_{\mathrm{fluc}, t}}
        \lesssim& \inprod{\Xi_0}{\widetilde\cU_{\mathrm{fluc}}(t; 0)^\ast\widetilde\cH_{\mathrm{fluc}}(t)\widetilde\cU_{\mathrm{fluc}}(t; 0)\Xi_0} \\
       & +\lambda\ve^{-M}\lra{t}^{-M} \inprod{\Psi_{\mathrm{fluc}, t}}{(\cN+1)\Psi_{\mathrm{fluc}, t}}\ .\notag
    \end{align}
    By differentiating the expected number of the fluctuation and applying Inequality~\eqref{est:bound_for_commutator_of_fluctuation_Hamiltonian_with_number}, Lemma~\ref{lem:fluctuation_number_squared}, and Inequality~\eqref{est:energy_bounded_by_fluc_energy}, we get
    \begin{align*}
        \frac{\d}{\d t}\inprod{\Psi_{\mathrm{fluc}, t}}{\cN\,\Psi_{\mathrm{fluc}, t}}
        \lesssim& \frac{1}{\varepsilon}\inprod{\Xi_0}{\widetilde\cU_{\mathrm{fluc}}^\ast(t; 0)\widetilde\cH_{\mathrm{fluc}}(t)\widetilde\cU_{\mathrm{fluc}}(t; 0)\Xi_0} \\
       & +\lambda \varepsilon^{-M-1}\lra{t}^{M} \inprod{\Psi_{\mathrm{fluc}, t}}{(\cN+1)\Psi_{\mathrm{fluc}, t}}\ .
    \end{align*}

    Next, let us estimate the growth of $\widetilde\cH_{\mathrm{fluc}}(t)$ with respect to the fluctuations. By differentiating the expectation value of $\widetilde\cH_{\mathrm{fluc}}(t)$ and applying Inequality~\eqref{est:bound_for_time-derivative_of_fluctuation_Hamiltonian}, Lemma~\ref{lem:fluctuation_number_squared}, and Inequality~\eqref{est:energy_bounded_by_fluc_energy}, we obtain
    \begin{align*}
       & \frac{\d}{\d t}\inprod{\Xi_0}{\widetilde\cU_{\mathrm{fluc}}^\ast(t; 0)\widetilde\cH_{\mathrm{fluc}}(t)\widetilde\cU_{\mathrm{fluc}}(t; 0)\Xi_0}\\
        &\, =\inprod{\Xi_0}{\widetilde\cU_{\mathrm{fluc}}^\ast(t; 0)\dot{\widetilde\cH}_{\mathrm{fluc}}(t)\widetilde\cU_{\mathrm{fluc}}(t; 0)\Xi_0}\\
        &\lesssim\, \frac{1}{\varepsilon}\inprod{\Xi_0}{\widetilde\cU_{\mathrm{fluc}}^\ast(t; 0)\widetilde\cH_{\mathrm{fluc}}(t)\widetilde\cU_{\mathrm{fluc}}(t; 0)\Xi_0} +\lambda\ve^{-M-1}\lra{t}^{M} \inprod{\Xi_0}{\widetilde\cU_{\mathrm{fluc}}^\ast(t; 0)\(\cN+1\)\widetilde\cU_{\mathrm{fluc}}(t; 0)\Xi_0}.
    \end{align*}
     Combining the above estimates, we have that
    \begin{align*}
        &\frac{\d}{\d t}\inprod{\Xi_0}{\widetilde\cU_{\mathrm{fluc}}^\ast(t; 0)\(\widetilde\cH_{\mathrm{fluc}}(t)+2\lambda \varepsilon^{-M}\lra{t}^{M}\(\cN+1\)\)  \widetilde\cU_{\mathrm{fluc}}(t; 0)\Xi_0}\\
        &\lesssim\, \frac{1}{\varepsilon}\(1+\lambda \varepsilon^{-M}\lra{t}^{M}\)\inprod{\Xi_0}{\widetilde\cU_{\mathrm{fluc}}^\ast(t; 0)\widetilde\cH_{\mathrm{fluc}}(t)\widetilde\cU_{\mathrm{fluc}}(t; 0)\Xi_0} \\
        &\quad +\lambda^2 \varepsilon^{-3M}\lra{t}^{3M}\inprod{\Xi_0}{\widetilde\cU_{\mathrm{fluc}}^\ast(t; 0)\(\cN+1\)\widetilde\cU_{\mathrm{fluc}}(t; 0)\Xi_0}\\
        &\lesssim\,  \lambda\ve^{-2M}\lra{t}^{2M}\inprod{\Xi_0}{\widetilde\cU_{\mathrm{fluc}}^\ast(t; 0)\(\widetilde\cH_{\mathrm{fluc}}(t)+2\lambda \varepsilon^{-M}\lra{t}^{M}\(\cN+1\)\)  \widetilde\cU_{\mathrm{fluc}}(t; 0)\Xi_0}.
    \end{align*}

    Finally, by Gr\"onwall's lemma and Inequality~\eqref{est:upper_bound_for_fluctuation_Hamiltonian}, we arrive at the result
    \begin{align}\label{equ:operator bound,HN,N}
        &\inprod{\Xi_0}{\widetilde\cU_{\mathrm{fluc}}^\ast(t; 0)\(\widetilde\cH_{\mathrm{fluc}}(t)+2\lambda \varepsilon^{-M}\lra{t}^{M}\(\cN+1\)\)  \widetilde\cU_{\mathrm{fluc}}(t; 0)\Xi_0}\\
        \lesssim& e^{C\lambda \varepsilon^{-2M}\lra{t}^{2M}}\inprod{\Xi_0}{\(\widetilde\cH_{\mathrm{fluc}}(0)+\lambda \varepsilon^{-M}\lra{t}^{M}\(\cN+1\)\)  \Xi_0}\notag\\
        \lesssim& e^{C\lambda \varepsilon^{-3M}\lra{t}^{3M}}\inprod{\Xi_0}{\(\cH_N+\frac{\cN^2}{N}+ (\cN+1)\)  \Xi_0}\ .\notag
    \end{align}
    Combining this with Inequality~\eqref{est:energy_bounded_by_fluc_energy2}, that is,
    \begin{align*}
        -\lambda\ve^{-M}\lra{t}^{-M} \inprod{\Psi_{\mathrm{fluc}, t}}{(\cN+1)\Psi_{\mathrm{fluc}, t}}
        \le \inprod{\Xi_0}{\widetilde\cU_{\mathrm{fluc}}(t; 0)^\ast\widetilde\cH_{\mathrm{fluc}}(t)\widetilde\cU_{\mathrm{fluc}}(t; 0)\Xi_0},
    \end{align*}
   we arrived at the estimate
    \begin{align}\label{equ:operator bound,N}
        \inprod{\Psi_{\mathrm{fluc}, t}}{(\cN+1)\Psi_{\mathrm{fluc}, t}}
        \lesssim e^{C\lambda \varepsilon^{-4M}\lra{t}^{4M}}
        \inprod{\Xi_0}{\(\cH_N+\frac{\cN^2}{N}+(\cN+1)\)  \Xi_0}\ .
    \end{align}
    Then the desired estimate for the energy \eqref{equ:operator bounds,HN} follows immediately from Inequality~\eqref{est:energy_bounded_by_fluc_energy2}, \eqref{equ:operator bound,HN,N}, and \eqref{equ:operator bound,N}, in which we replaced $4M$ by $M$ for simplicity.
\end{proof}

\section{Semiclassical Limit of the Modified Gross--Pitaevskii Equation}\label{section:Semiclassical Limit of the Modified Gross--Pitaevskii Equation}

In this section, we study the semiclassical limit of the modified Gross--Pitaevskii equation~\eqref{eq:modified_Gross--Pitaevskii}. For the reader's convenience, we recall the Cauchy problem
\begin{equation}\label{equ:N-hartree equation}
\left\{\begin{aligned}
i\varepsilon\,\bd_t\phi  =& -\tfrac{\varepsilon^2}{2}\lapl \phi + F(\phi)\phi  \\
\phi(0)=&\phi^{\init}
\end{aligned}
\right.
\ ,
\end{equation}
where $F(\phi)=K\ast\n{\phi}^2$ and $K(x)=
        \lambda(N) (N^{\beta}\varepsilon^{2\kappa})^3 v(N^\beta \varepsilon^{2\kappa} x)f_{L}^{\mu}(N^\beta \varepsilon^{2\kappa} x)$.
Here, $f_{L}^{\mu}(x)=f_{L}^{\mu}(\n{x})$ is the positive ground state of the Neumann problem (cf. Appendix~\ref{appendix:neumann})
\begin{equation}\label{equ:scattering equation,neumann}
\left\{
\begin{aligned}
&\lrs{-\mu\, \lapl+ v}f_{L}^{\mu}=E_{\mu}^{L}f_{L}^{\mu},\quad \text{for $|x|<L$}\\
&\text{$f_{L}^{\mu}(L)=1$ and $\pa_{r}f_{L}^{\mu}(L)=0$}
\end{aligned}
\right. \ ,
\end{equation}
where the parameters $\mu=N^{1-\be}\ve^{2(1-\ka)}/\la$ and $L=N^\beta\ve^{2\ka}l=N^\beta\ve^{2\ka+4}$.
Let us also recall the range of the parameters
\begin{align*}
\be\geq 1,\quad \ka\in [0,1],\quad \text{ and }
\la(N)\in \lr{1,\(\ln N\)^\upalpha},\quad \upalpha\in(0,1) \ .
\end{align*}

More precisely, we consider the semiclassical limit of the hydrodynamical formulation of Equation~\eqref{equ:N-hartree equation}. Suppose $\phi^\varepsilon_{N, t}$ is a solution to Equation~\eqref{equ:N-hartree equation} and define the following local density quantities
    \begin{align*}
        &\rho^{\varepsilon}(t, x)=\, |\phi^\varepsilon_{N, t}|^2\ , &
        &J^{\varepsilon}_{\jj}(t, x)=\,  \Im{\varepsilon\bd_{\jj}\phi^\varepsilon_{N, t}\, \conj{\phi^\varepsilon_{N, t}}}\ , \\
        &\sigma^\varepsilon_{\jj k}(t, x)=\, \re\(\varepsilon\bd_{\jj}\phi^\varepsilon_{N, t}\, \conj{\varepsilon\bd_{k}\phi^\varepsilon_{N, t}}\), &  &\sigma^\varepsilon_{0, \jj}(t, x)= \re\(\varepsilon\bd_{0}\phi^\varepsilon_{N, t}\ , \conj{\varepsilon\bd_{\jj}\phi^\varepsilon_{N, t}}\)\ ,\\
        &P^\varepsilon(t, x)=\, \tfrac{\varepsilon^2}{4}\lapl \rho^{\varepsilon}_t-\tfrac{1}{2} \(K\ast\rho^{\varepsilon}_t\)\rho^{\varepsilon}_t\ , &
        &e^\varepsilon(t, x)=\, \tfrac{1}{2}|\varepsilon\grad\phi^\varepsilon_{N, t}|^2+\tfrac{1}{2} (K\ast \rho^{\varepsilon}_t)\rho^{\varepsilon}_t\ ,
    \end{align*}
and lastly the error terms
\begin{align*}
    &l_{\jj}^\varepsilon(t, x)=\textstyle\frac12 \intd K(x-y)\(\bd_{\jj}\rho^{\varepsilon}_t(y)\rho^{\varepsilon}_t(x)-\rho^{\varepsilon}_t(y)\bd_{\jj}\rho^{\varepsilon}_t(x)\)\dd y\\
    &l_{0}^\varepsilon(t, x)=\textstyle\frac12 \intd K(x-y)\(\bd_{0}\rho^{\varepsilon}_t(y)\rho^{\varepsilon}_t(x)-\rho^{\varepsilon}_t(y)\bd_{0}\rho^{\varepsilon}_t(x)\)\dd y
\end{align*}
where the indices $\jj=1, 2, 3$  with $\bd_0 = \bd_t$ and $\bd_{\jj}=\bd_{x_{\jj}}$. We write $\vect{J}^\varepsilon=(J_1^\varepsilon, J_2^\varepsilon, J_3^\varepsilon)$, similarly for $\vect{l}^\varepsilon$, and $\vect{\sigma}^\varepsilon = (\sigma^\varepsilon_{\jj k})_{\jj, k \in\{1, 2, 3\}}$.


\subsection{Limiting equations}  Although the family of modified Gross--Pitaevskii equation \eqref{equ:N-hartree equation} with different parameters has almost the same nonlinear structure, their limiting behaviors are completely different for different choices of the parameters $\ka$ and $\be$. It is not immediately obvious how the choice of parameters affects the limiting behavior. One of our goals in this section is to answer this question.

\textbf{The \texorpdfstring{$\be=1$}{beta=1} regime.} According to the different nonlinear effects, we divide them into the following three typical cases. See also the monograph \cite{carles2021semi} for more details on the classification by WKB analysis.

\begin{enumerate}[$(i)$]
\item $\ka=1$ (supercritical case, $\kappa>\frac12$).

\begin{itemize}
    \item For $\la(N)\equiv 1$, the limiting equation is
    the compressible Euler equation with the scattering length $\mathfrak{a}_{0}$:
    \begin{equation}\label{equ:euler equation,a0}
    \begin{cases}
        \partial _{t}\rho +\nabla \cdot \left( \rho\, \bu\right) =0 \\
        \partial _{t}\bu+(\bu\cdot \nabla )\bu+4\pi \mathfrak{a}_{0}\nabla \rho =0
    \end{cases}\ .
    \end{equation}
    \item For $\la(N)\equiv(\ln N)^\upalpha$, the limiting equation is
    the compressible Euler equation with the capacity $\mathfrak{c}_{0}$:
    \begin{equation}\label{equ:euler equation,c0}
    \begin{cases}
    \partial _{t}\rho +\nabla \cdot \left( \rho\, \bu\right) =0 \\
    \partial _{t}\bu+(\bu\cdot \nabla )\bu+4\pi \mathfrak{c}_{0}\nabla \rho =0
    \end{cases} \ .
    \end{equation}
\end{itemize}

\item $\ka=\frac{1}{2}$ (critical case).

For both $\la(N)\equiv 1$ and $\la(N)\equiv (\ln N)^\upalpha$,
the limiting equation is the eikonal system with the capacity constant $\mathfrak{c}_{0}$:
 \begin{equation}\label{equ:eikonal equation,main}
 \begin{cases}
    \pa_{t}a+\nabla \varphi_{\mathrm{eik}}\cdot \nabla a+\tfrac{1}{2}a\,\Delta \varphi_{\mathrm{eik}}=-i4\pi \mathfrak{c}_{0}|a|^{2}a\\[3pt]
    \pa_{t}\varphi_{\mathrm{eik}}+\tfrac{1}{2}|\nabla \varphi_{\mathrm{eik}}|^{2}=0
 \end{cases}\ .
 \end{equation}
\item $\ka=0$ (subcritical case, $\kappa<\frac12$).

For both $\la(N)\equiv 1$ and $\la(N)\equiv (\ln N)^\upalpha$,
the limiting equation is the trivial eikonal system (or pressureless compressible Euler equation):
 \begin{equation}\label{eq:trivial_eikonal}
    \begin{cases}
        \pa_{t}a+\nabla \varphi_{\mathrm{eik}}\cdot \nabla a+\tfrac{1}{2}a\, \Delta \varphi_{\mathrm{eik}}=0\\[3pt]
        \pa_{t}\varphi_{\mathrm{eik}}+\tfrac{1}{2}|\nabla \varphi_{\mathrm{eik}}|^{2}=0
    \end{cases}\ .
 \end{equation}
\end{enumerate}
\begin{remark}In the $\beta>1$ regime, for all $\ka\in [0,1]$ and $\la(N)\in \lr{1,(\ln N)^\upalpha}$,
the limiting equation is the trivial eikonal system~\eqref{eq:trivial_eikonal}.
This is much expected; since for $\beta>1$, heuristically, the particle scattering effects become negligible at this scale.
\end{remark}

The plan for the remainder of the section is to prove the convergence of the mass, momentum, and energy densities associated with solution $\phi^\varepsilon_{N, t}$ to Equation~\eqref{equ:N-hartree equation}
in the semiclassical limit as $\varepsilon\rightarrow 0$.
In Section~\ref{section:Modulated Energy Method}, we deal with the supercritical case $\ka=1$ (and $\beta=1$) by the method of modulated energy, and, in~Section \ref{section:WKB Analysis Method},  we handle the critical and subcritical cases $\ka=\frac{1}{2}, 0$ by the method of WKB analysis.

\subsection{The method of modulated energy}\label{section:Modulated Energy Method}
In this section, we treat the case $\be=1$, $\ka=1$ by the method of modulated energy, which is a powerful tool for the study of the semiclassical limit of nonlinear Schr\"odinger equations.

Consider the modulated energy functional
\begin{align}\label{equ:modulated energy}
\cM\lrc{\phi_{N}^{\ve},\rho,\bu}(t)=& \frac{1}{2}\int_{\R^{3}}\n{(i\ve\nabla+\bu_t)\phi_{N, t}^{\ve}}^2\d x
+\cM_{\mathrm{pot.}}\lrc{\phi_{N}^{\ve},\rho}(t)
\end{align}
where
\begin{align*}
    \cM_{\mathrm{pot.}}[\phi_{N}^{\ve},\rho](t):=&\,  \frac12\intd (K\ast\rho^\varepsilon_t)\rho^\varepsilon_t\dd x
      +
    \left\{
    \begin{aligned}
         \textstyle  4\pi \mathfrak{a}_{0},\,\, &\text{GP} \\
         \textstyle  4\pi \mathfrak{c}_{0},\,\, &\text{HC}
    \end{aligned}
    \right\}\times
     \frac12\intd \(\rho_t^2-2\,\rho_t\rho_t^\varepsilon\)\dd x\ .
\end{align*}
Here, $K(x)=\la(N)(N\ve^{2})^{3}(vf_{L}^{\mu})(N\ve^{2}x)$. This modulated energy functional is similar to the one in \cite{chen2023derivation} but carries an extra scattering function $f_{L}^{\mu}$ satisfying the scattering equation \eqref{equ:scattering equation,neumann}.

Here is the main proposition of the section.
\begin{prop}\label{prop:semiclassical limit and the modulated energy}
Let $(\be, \ka)=(1, 1)$, $\phi_{N, t}^{\ve}$ be the solution to Equation~\eqref{equ:N-hartree equation} with the initial data $\phi^\init$, and $(\rho,\bu)$ be the solution to the compressible Euler equation with the initial data $(\rho^{\mathrm{in}},\bu^\init)\in H^5(\R^3)\times H^4(\R^3; \R^3)$.
Then for any $T_0>0$ before the blow-up time $T_{\ast}$ of Equation~\eqref{equ:euler equation,a0} (or Equation~\eqref{equ:euler equation,c0}), we have the following convergence of mass and momentum densities as $\frac{1}{N}+\varepsilon\rightarrow 0$:
\begin{enumerate}[$(i)$]
\item In the case $\la(N)=1$, the limiting equation is given by Equation~\eqref{equ:euler equation,a0} where we see the emergence of the scattering length $\mathfrak{a}_{0}$ in the pressure term. Moreover, there exists $C=C_{T_0}>0$ such that
\begin{align}
    &\Nrm{\rho^{\ve}-\rho}{L^{\infty}_t([0,T_{0}])L^{2}_x}^{2}\le C\(
    \cM \lrc{\phi_{N}^{\ve},\rho,\bu}(0)+\ve^{2}+\frac{1}{L}\)\ ,\label{equ:convergence,uniform in h,a0, hartree-euler,density}\\
    &\Nrm{\vect{J}^{\ve}-\rho \bu}{L^{\infty}_t([0,T_{0}])L^{1}_x}^{2}\le  C\( \cM \lrc{\phi_{N}^{\ve},\rho, \bu}(0)+\ve^{2}+\frac{1}{L}\)\label{equ:convergence,uniform in h,a0, hartree-euler,moment} \ ,\\
    &\Nrm{|i\ve\nabla \phi_{N}^{\ve}|^{2}-\rho\n{\bu}^{2}}{L^{\infty}([0,T_{0}])L^{1}_x}^{2}\le C\(
    \cM \lrc{\phi_{N}^{\ve},\rho,\bu}(0)+\ve^{2}+\frac{1}{L}\)\ .
\end{align}
\item In the case $\la(N)\to \infty$ with $\la(N)= \(\ln N\)^\upalpha$, the limiting equation is given by Equation~\eqref{equ:euler equation,c0} where we see the emergence of the electrostatic capacity $\mathfrak{c}_{0}$ in the pressure term. Moreover, there exists $C=C_{T_0}>0$ such that
\begin{align}
&\Nrm{\rho^{\ve}-\rho}{L^{\infty}_t([0,T_{0}])L^{2}_x}^{2}\le C\(
\cM \lrc{\phi_{N}^{\ve},\rho,\bu}(0)+\ve^{2}+\frac{\la(N)}{L}+\eta\lrs{\tfrac{1}{\la(N)}}\)\ ,\label{equ:convergence,uniform in h,hartree-euler,density}\\
&\Nrm{\vect{J}^{\ve}-\rho\, \bu}{L^{\infty}_t([0,T_{0}])L^{1}_x}^{2}\le C\( \cM \lrc{\phi_{N}^{\ve},\rho,\bu}(0)+\ve^{2}+\frac{\la(N)}{L}+
\eta\lrs{\tfrac{1}{\la(N)}}\)\label{equ:convergence,uniform in h,c0, hartree-euler,moment}\ ,\\
&\begin{multlined}[t][.87\textwidth]
    \Nrm{|i\ve\nabla \phi_{N}^{\ve}|^{2}-\rho\n{\bu}^{2}}{L^{\infty}_t([0,T_{0}])L^{1}_x}^{2}\le C\bigg(
\cM \lrc{\phi_{N}^{\ve},\rho,\bu}(0)+\ve^{2}+\frac{\la(N)}{L}+
\eta\lrs{\tfrac{1}{\la(N)}}\bigg)\ .\label{equ:convergence,uniform in h,c0, hartree-euler,kinetic}\end{multlined}
\end{align}
 Here, the convergence rate $\eta(\ve)$, which tends to zero as $\ve\to 0$, depends on the profile near the zero point of the potential $v(x)$.
\end{enumerate}
\end{prop}
\begin{proof}
    It suffices to consider the case $\la(N)\rightarrow \infty$, since the case $\la(N)= 1$ follows in a similar manner.
 The proof follows from \cite{chen2023derivation} with some slight modifications; nevertheless, we present it here for completeness.
    By a direct calculation (cf. \cite[Proposition 3.3]{chen2023derivation}), we have
    \begin{equation}
        \begin{aligned}[t]
            \frac{\d}{\d t}\cM\lrc{\phi_{N}^{\ve},\rho,\bu}(t)
        =&-\sum_{\jj=1}^{3}\sum_{k=1}^{3}\int_{\R^{3}}\pa_{k}u^{\jj}\operatorname{Re}
            \lrs{(i\ve\pa_{k}+u^{k})\phi_{N}^{\ve}\ol{\lrs{i\ve\pa_{\jj}+u^{\jj}}\phi_{N}^{\ve}}}\dd x\\
            &\quad\,-\frac{4\pi \mathfrak{c}_{0}}{2}\intd(\operatorname{div}\bu)(\rho^{\ve}-\rho)^{2}\dd x-\frac{\ve^{2}}{4}\intd
             \rho^{\ve}\Delta \operatorname{div}\bu \dd x+\mathrm{Er}
        \end{aligned}
    \end{equation}
    where $\bu = (u^1, u^2, u^3)$ and the error term is given by
    \begin{align*}
    \mathrm{Er}=\frac{1}{2}\intd (\operatorname{div}\bu) \(4\pi\mathfrak{c}_0\rho^\varepsilon-K\ast\rho^\varepsilon\)\rho^\varepsilon\dd x
    +\intd \bu\cdot \vect{l}^\varepsilon\dd x=:\mathrm{Er}_1+\mathrm{Er}_2\ .
    \end{align*}
    Then, by H\"{o}lder's inequality, we obtain
    \begin{equation}\label{equ:elementary estimate,evolution,modulated energy}
        \begin{aligned}[t]
            \frac{\d}{\d t}\cM\lrc{\phi_{N}^{\ve},\rho,\bu}(t)
            &\le C\nv{\nabla \bu}_{L^{\wq}_x}
            \lrs{\intd \n{(i\ve\nabla+\bu)\phi_{N}^{\ve}}^{2}\dd x+4\pi \mathfrak{c}_{0}\intd (\rho^{\ve}-\rho)^{2}\dd x}\\
            &\quad\, +\frac{\ve^{2}}{4}\nv{\rho^{\ve}}_{L^{1}_x}\nv{\Delta \operatorname{div}\bu}_{L^{\wq}_x}
            +\n{\mathrm{Er}}\ .
        \end{aligned}
    \end{equation}

    It remains to estimate the error term $\mathrm{Er}$. For the $\mathrm{Er}_{1}$ term, by H\"{o}lder's inequality, we obtain
    \begin{align}\label{equ:estimate,Er1}
    \n{\mathrm{Er}_{1}}\le& \Nrm{\operatorname{div} \bu}{L^{\infty}_x}
    \Nrm{\(4\pi\mathfrak{c}_0-\textstyle\int K\dd x\)\rho^\varepsilon}{L^2_x}\Nrm{\rho^\varepsilon}{L^2_x}\\
    &+\Nrm{\operatorname{div} \bu}{L^{\infty}_x}
    \Nrm{\(\textstyle\int K\dd x\, \delta_0-K\)\ast\rho^\varepsilon}{L^{3/2}_x}\Nrm{\rho^\varepsilon}{L^3_x}\notag
    \end{align}
    where $\delta_0$ is the Dirac delta function. On the one hand, by the Property~\eqref{part:approximation_of_scattering_length} of the scattering function, we have
    \begin{align*}
    \intd  K \dd x=\la(N)\intd  v(x)f_{L}^{\mu}(x)\dd x=4\pi \asc_{0}^{\mu}+\mathcal{O}\lrs{\frac{\la(N)}{L}}\ ,
    \end{align*}
    where $\mu=\la(N)^{-1}$ due to that $\be=\ka=1$.
    Using the fact that $\nv{\rho^{\ve}}_{L^{2}_x}\lesssim 1$, Lemma~\ref{lem:phi4_bound},  and
    \begin{align*}
        \n{\asc_{0}^{\mu}-\mathfrak{c}_{0}}\le \eta(\mu)\ ,
    \end{align*}
    which is given by Proposition \ref{prop:quantitative_scattering_length}, we get
    \begin{align}\label{equ:estimate,Er1,error,a0}
    \n{4\pi \mathfrak{c}_{0}-\textstyle\intd K\dd x}\Nrm{\rho^{\ve}}{L^{2}_x}\le C \(\frac{\la(N)}{L}+\eta(\mu)\)\ .
    \end{align}
    On the other hand, the identity approximation \eqref{equ:quantative estimate for identity approximation} in Lemma \ref{lem:quantative estimate for identity approximation} implies that
    \begin{align}\label{equ:estimate,Er1,error,identity}
    \bbn{\lrs{K-\delta_0\textstyle\intd K\dd x}*\rho^{\ve}}_{L^{\frac{3}{2}}_x}\lesssim&\, \frac{\la(N)}{N\ve^{3}}\Nrm{\lra{x} vf^{\mu}_L}{L^1_x}\Nrm{\lra{\varepsilon\nabla}\rho^{\ve}}{L^{\frac{3}{2}}_x}
    \lesssim  \frac{\la(N)}{N\ve^{4}}\ .
    \end{align}
    Combining estimates \eqref{equ:estimate,Er1}--\eqref{equ:estimate,Er1,error,identity}, we arrive at
    \begin{align}\label{equ:estimate,Er1,final}
    \n{\mathrm{Er}_{1}}\lesssim& \frac{\la(N)}{L}+\eta(\mu)+\frac{\la(N)}{N\ve^{6}}\lesssim \frac{\la(N)}{L}+\eta(\mu)\ .
    \end{align}

    For the $\mathrm{Er}_{2}$ term, we start by writing
    \begin{align*}
        \mathrm{Er}_2 = \sum^3_{\jj=1}\frac12 \intdd \n{x-y}K(x-y)\frac{u^\jj(x)-u^\jj(y)}{\n{x-y}}\bd_{\jj}\rho^{\varepsilon}_t(y)\rho^{\varepsilon}_t(x)\dd x\d y
    \end{align*}
    which then means
    \begin{align}
        \n{\mathrm{Er}_2}\lesssim \frac{\Nrm{\grad \bu}{L^\infty}}{N\varepsilon^3} \lambda(N)\Nrm{\n{x} v f^\mu_{L}}{L^1_x}\Nrm{\varepsilon\grad\rho^\varepsilon}{L^\frac32_x}\Nrm{\rho^\varepsilon}{L^3_x}\lesssim  \frac{\lambda(N)}{N\varepsilon^6}\ .
    \end{align}
    Hence we obtain the desired bound
    \begin{align}\label{equ:error estimate,Er}
        \n{\mathrm{Er}}\lesssim \frac{\la(N)}{L}+\eta(\mu)\ .
    \end{align}

    In a similar manner as the analysis on the error term $\mathrm{Er}_{1}$, we have the coercivity estimate for the modulated energy, that is,
    \begin{align}
        \cM\lrc{\phi_{N}^{\ve},\rho,\bu}(t)=&\, \frac{1}{2}\int_{\R^{3}}|(i\ve\nabla+\bu)\phi_{N}^{\ve}|^{2}\dd x
        +\frac{4\pi \mathfrak{c}_{0}}{2}\nv{\rho^{\ve}-\rho}_{L^{2}_x}^{2}\label{equ:coercive estimate,modulated energy}\\
        &+\frac{1}{2}\intd \((K-\mathfrak{c}_{0}\delta)*\rho^{\ve}\)
        \rho^{\ve}\dd x\notag\\
        \gtrsim&\,  \frac{1}{2}\int_{\R^{3}}|(i\ve\nabla+\bu)\phi_N^{\ve}|^{2}\dd x
        +\frac{4\pi \mathfrak{c}_{0}}{2}\Nrm{\rho^{\ve}-\rho}{L^{2}_x}^{2} -\frac{\la(N)}{L}-\eta(\mu)\ .\notag
    \end{align}
    Substituting the error term Estimate~\eqref{equ:error estimate,Er} on $\mathrm{Er}$ and the coercivity Estimate~\eqref{equ:coercive estimate,modulated energy} into Estimate~\eqref{equ:elementary estimate,evolution,modulated energy} on the evolution of the modulated energy, we obtain the bound
    \begin{align}
        \frac{\d}{\d t}\cM\lrc{\phi_{N}^{\ve},\rho,\bu}(t)\lesssim\cM\lrc{\phi_{N}^{\ve},\rho,\bu}(t)+ \ve^{2}+\frac{\la(N)}{L}+\eta(\mu)\ .
    \end{align}
    Then, by Gronwall's lemma, we get that
    \begin{align}\label{equ:upper bound for modulated energy}
        \cM\lrc{\phi_{N}^{\ve},\rho,\bu}(t)+\frac{\la(N)}{L}
        \leq \exp(CT_{0})
        \lrs{\cM\lrc{\phi_{N}^{\ve},\rho,\bu}(0)+\ve^{2}+\frac{\la(N)}{L}+\eta(\mu)}\ .
    \end{align}

    Finally, by \eqref{equ:upper bound for modulated energy} and the coercivity estimate~\eqref{equ:coercive estimate,modulated energy}, we immediately have
    \begin{align}\label{equ:quantitative estimate,moment}
    \Nrm{(i\ve\nabla+\bu)\phi_{N}^{\ve}}{L^{\infty}_t([0,T_{0}])L^{2}_x}^{2}\lesssim_{T_0} \cM\lrc{\phi_{N}^{\ve},\rho,\bu}(0)+\ve^{2}+\frac{\la(N)}{L}+\eta(\mu)\ ,
    \end{align}
    and the convergence of mass density
    \begin{align}\label{equ:quantitative estimate,mass}
    &\Nrm{\rho^{\ve}-\rho}{L^{\infty}_t([0,T_{0}])L^{2}_x}^{2}\lesssim_{T_0}\cM\lrc{\phi_{N}^{\ve},\rho,\bu}(0)+\ve^{2}+\frac{\la(N)}{L}+\eta(\mu)\ .
    \end{align}

    For the convergence of momentum density, by the triangle inequality, H\"{o}lder's inequality, and Estimates~\eqref{equ:quantitative estimate,moment}--\eqref{equ:quantitative estimate,mass}, we get
\begin{align*}
\Nrm{\vect{J}^{\ve}-\rho\,\bu}{L^{1}_x}\leq &\Nrm{\vect{J}^{\ve}-\rho^{\ve}\bu}{L^{1}_x}+\Nrm{\rho^{\ve}\bu-\rho\, \bu}{L^{1}_x}\\
\leq& \Nrm{\phi_{N}^{\ve}}{L^{2}_x}\Nrm{(i\ve\nabla+\bu)\phi_{N}^{\ve}}{L^{2}_x}+\Nrm{\bu}{L^{2}_x}
\nv{\rho^{\ve}-\rho}_{L^{2}_x}\\
\leq &C(T_{0})\lrs{\mathcal{M}\lrc{\phi_{N}^{\ve},\rho,\bu}(0)+\ve^{2}+\frac{\la(N)}{L}+\eta(\mu)}^{\frac{1}{2}}.
\end{align*}
This completes the proof of Inequality~\eqref{equ:convergence,uniform in h,c0, hartree-euler,moment}.

Lastly, for the convergence of kinetic energy density, we also have
\begin{align*}
&\Nrm{|i\varepsilon\nabla \phi_{N}^{\ve}|^{2}-\rho\n{\bu}^{2}}{L^{1}_x}\\
\leq& \Nrm{(i\ve\nabla+\bu)\phi_{N}^{\ve}}{L^{2}_x}\Nrm{(i\ve\nabla-\bu)\phi_{N}^{\ve}}{L^{2}_x}+\Nrm{\rho^{\ve}-\rho}{L^{2}_x}
\Nrm{\bu}{L^{4}_x}^2\\
\leq& C(T_{0})\lrs{\cM\lrc{\phi_{N}^{\ve},\rho,\bu}(0)+\ve^{2}+\frac{\la(N)}{L}+\eta(\mu)}^{\frac{1}{2}},
\end{align*}
which completes the proof of Inequality~\eqref{equ:convergence,uniform in h,c0, hartree-euler,kinetic}.
\end{proof}

\subsection{The method of the WKB analysis}\label{section:WKB Analysis Method}
In the section, we handle the cases $\ka=\frac{1}{2},0$ by a unified approach, the WKB analysis method, which has been widely used in the study of the asymptotic behavior of the wave function. See also the monograph \cite{carles2021semi} for a systematic study of the NLS equation.
Here, following \cite{carles2007wkb,carles2021semi}, we deal with the modified Gross--Pitaevskii equation \eqref{equ:N-hartree equation}, which requires a subtle analysis of the scattering function in different physical regimes.

Define
\begin{align}
    a_{N}^{\ve}(t, x):=\phi_{N}^{\ve}(t, x) e^{-i \varphi_{\mathrm{eik}}(t, x) / \ve}\ ,
\end{align}
where $\varphi_{\mathrm{eik}}(t)$ satisfies the so called eikonal equation:
\begin{equation}\label{equ:eikonal,phase}
\left\{
\begin{aligned}
&\pa_{t}\varphi_{\mathrm{eik}}+\tfrac{1}{2}|\nabla\varphi_{\mathrm{eik}}|^{2}=0\\
&\varphi_{\mathrm{eik}}(0)=\varphi_{\mathrm{eik}}^{\init}
\end{aligned}
\right. \ ,
\end{equation}
which is a Hamilton--Jacobi equation.
Then $\phi_{N}^{\ve}$ solves the modified Gross--Pitaevskii equation \eqref{equ:N-hartree equation} if and only if $a_{N}^{\ve}$ solves the modified amplitude equation:
\begin{equation}\label{equ:eikonal equation,h}
\left\{
\begin{aligned}
\partial_t a_{N}^{\ve}+\nabla \varphi_{\mathrm{eik}} \cdot \nabla a_{N}^{\ve}+\tfrac{1}{2} a_{N}^{\ve}\, \Delta \varphi_{\mathrm{eik}} & =i \tfrac{\ve}{2} \Delta a_{N}^{\ve}-i \ve^{-1} F(a_{N}^{\ve}) a_{N}^{\ve} \\
a_{N}^{\ve}|_{t=0} & =a^{\ve, \init}
\end{aligned}
\right. \ .
\end{equation}

The main result of this section is the following proposition.

\begin{prop}\label{prop:semiclassical limit,wkb}
Let $s>\frac{3}{2}$, and $p\in[1,2]$.
Assume the initial data satisfy
\begin{align*}
&\phi_{N}^{\ve}(0)=a^{\ve,\mathrm{in}}e^{i\varphi_{\mathrm{eik}}^{\mathrm{in}}/\ve},\quad
 \varphi_{\mathrm{eik}}^{\mathrm{in}}\in C^{\infty}(\R^3),\quad \pa_{x}^{\sigma}\varphi_{\mathrm{eik}}^{\mathrm{in}}\in L^{\infty}(\R^3) \quad \forall \sigma \in \N_0^3,\ |\sigma|\geq 2\ ,\\
& a^{\mathrm{in}}\in \bigcap_{k\geq 0}H^{k}(\R^3),\quad
\quad \nv{a^{\ve,\mathrm{in}}}_{H^{s}_x}\leq C\ .
\end{align*}
\begin{enumerate}[$(i)$]
\item For the BGP regime in the critical case:
$$\beta=1,\ \kappa=\tfrac{1}{2},\ \la(N)=(\ln N)^{\upalpha}\text{ with }\upalpha\in[0,1)\ ,$$
we have
\begin{equation}\label{equ:semiclassical limit,convergence rate,k=1/2,precise}
\left\{
\begin{aligned}
&\Nrm{\rho^{\ve}-|a|^{2}}{L^{\infty}_t([0, T_{0}])L^{p}_x}\lesssim \nv{a^{\ve,\mathrm{in}}-a^{\mathrm{in}}}_{H^{s}_x}+\ve+\eta(\ve)\ ,\\[3pt]
&\Nrm{\vect{J}^{\ve}-|a|^{2}\nabla \varphi_{\mathrm{eik}}}{L^{\infty}_t([0, T_{0}])L^{p}_x}\lesssim \nv{a^{\ve,\mathrm{in}}-a^{\mathrm{in}}}_{H^{s}_x}+\ve+\eta(\ve)\ ,\\[3pt]
&\Nrm{|i\ve\nabla \phi_{N}^{\ve}|^{2}-|a|^{2}|\nabla \varphi_{\mathrm{eik}}|^{2}}{L^{\infty}_t([0,T_{0}])L^{1}_x}
\lesssim \nv{a^{\ve,\mathrm{in}}-a^{\mathrm{in}}}_{H^{s}_x}+\ve+\eta(\ve)\ ,
\end{aligned}
\right.
\end{equation}
where $(a,\varphi_{\mathrm{eik}})$ satisfies the eikonal system
 \begin{equation}\label{equ:eikonal equation,theorem, precise}
 \left\{
 \begin{aligned}
&\pa_{t}a+\nabla \varphi_{\mathrm{eik}}\cdot \nabla a+\tfrac{1}{2}a\,\Delta \varphi_{\mathrm{eik}}=-i 4\pi \mathfrak{c}_{0}|a|^{2}a,\\
&\pa_{t}\varphi_{\mathrm{eik}}+\tfrac{1}{2}|\nabla \varphi_{\mathrm{eik}}|^{2}=0,
 \end{aligned}
 \right.
 \end{equation}
 with $\mathfrak{c}_{0}$ being the capacity of $v$.
 \item For the SGP regime:
 $$\beta=1,\ \kappa=0,\ \la(N)=(\ln N)^{\upalpha} \text{ with } \upalpha\in[0,1)\ ,$$
 we have
 \begin{equation}\label{equ:semiclassical limit,convergence rate,k=0, precise}
\left\{
    \begin{aligned}
        &\Nrm{\rho^{\ve}-|a|^{2}}{L^{\infty}_t([0, T_{0}])L^{p}_x}\lesssim \nv{a^{\ve,\mathrm{in}}-a^{\mathrm{in}}}_{H^{s}_x}+\ve\ ,\\[3pt]
        &\Nrm{\vect{J}^{\ve}-|a|^{2}\nabla \varphi_{\mathrm{eik}}}{L^{\infty}_t([0, T_{0}])L^{p}_x}\lesssim \nv{a^{\ve,\mathrm{in}}-a^{\mathrm{in}}}_{H^{s}_x}+\ve\ ,\\[3pt]
        &\Nrm{|i\ve\nabla \phi_{N}^{\ve}|^{2}-|a|^{2}|\nabla \varphi_{\mathrm{eik}}|^{2}}{L^{\infty}_t([0,T_{0}])L^{1}_x}
        \lesssim \nv{a^{\ve,\mathrm{in}}-a^{\mathrm{in}}}_{H^{s}_x}+\ve\ ,
    \end{aligned}
\right.
\end{equation}
where $(a,\varphi_{\mathrm{eik}})$ satisfies the trivial eikonal system
 \begin{equation}\label{equ:eikonal equation, precise}
 \left\{
 \begin{aligned}
&\pa_{t}a+\nabla \varphi_{\mathrm{eik}}\cdot \nabla a+\tfrac{1}{2}a\,\Delta \varphi_{\mathrm{eik}}=0\\
&\pa_{t}\varphi_{\mathrm{eik}}+\tfrac{1}{2}|\nabla \varphi_{\mathrm{eik}}|^{2}=0
 \end{aligned}
 \right. \ .
 \end{equation}
\end{enumerate}
\end{prop}

Proposition~\ref{prop:semiclassical limit,wkb} follows from the following two lemmas, and we shall present the proof of Proposition~\ref{prop:semiclassical limit,wkb} at the end of the section.
\begin{lem}\label{lem:well-posedness,uniform estimate,a}
Let $s>\frac{3}{2}$ and $s\in \mathbb{N}$.
Assume the initial data satisfy
\begin{align}\label{equ:eikonal equation,initial condition}
 \varphi_{\mathrm{eik}}^{\mathrm{in}}\in C^{\infty}(\R^3),\quad \pa_{x}^{\sigma}\varphi_{\mathrm{eik}}^{\mathrm{in}}\in L^{\infty}(\R^3) \quad \text{for $|\sigma|\geq 2$,}
\quad \nv{a^{\ve,\mathrm{in}}}_{H^{s}_x}\leq C\ .
\end{align}
Then
there exists $T_{0}>0$ independent of $\ve\in (0,1]$ such that Equation~\eqref{equ:eikonal,phase} has a unique solution $\varphi_{\mathrm{eik}}(t)$ satisfying
\begin{align}\label{equ:eikonal equation,property}
\varphi_{\mathrm{eik}}\in C^{\infty}([0,T_{0}]\times \R^{3}), \quad \nv{\pa_{x}^{\sigma}\varphi_{\mathrm{eik}}}_{L^{\infty}_t([0,T_{0}])L^{\infty}_x}<\infty\quad \forall \sigma\in \N_0^3,\ |\sigma|\geq 2\ ,
\end{align}
and Equation~\eqref{equ:eikonal equation,h} admits a unique solution $a_{N}^{\ve}\in C([0,T_{0}];H^{s}(\R^3))$.
Moreover,
$a_{N}^{\ve}$ is bounded in $L^{\infty}([0,T_{0}]; H^{s}(\R^3))$ uniformly in
$N$ and $\ve$.
\end{lem}

\begin{proof}
    The local well-posedness of Equation~\eqref{equ:eikonal,phase} follows from the standard characteristic method. By applying $\nabla_{x}$ to Equation~\eqref{equ:eikonal,phase}, we get
    \begin{equation}\label{equ:eikonal,phase,derivative}
    \left\{
    \begin{aligned}
    &\pa_{t}\nabla_{x}\varphi_{\mathrm{eik}}+(\nabla_{x}\varphi_{\mathrm{eik}}\cdot \nabla_{x})\nabla_{x}\varphi_{\mathrm{eik}}=0\\
    &\nabla_{x} \varphi_{\mathrm{eik}}(0)=\nabla\varphi_{\mathrm{eik}}^{\mathrm{in}}
    \end{aligned}
    \right. \ .
    \end{equation}
    Then the solution has the form
    \begin{align}
    \nabla_{x}\varphi_{\mathrm{eik}}(t,x)=(\nabla\varphi_{\mathrm{eik}}^{\mathrm{in}})(X^{-1}(t,x))\ ,
    \end{align}
    where $X(t, x)$ is the flow map given by
    \begin{align*}
    X(t,x)=x+t\nabla\varphi_{\mathrm{eik}}^{\mathrm{in}}(x)\ .
    \end{align*}
    Due to the initial condition that $\Nrm{\nabla^{2}\varphi_{\mathrm{eik}}^{\mathrm{in}}}{L^{\infty}}<\infty$,
    for the local time, the flow map is a bilinear smooth map.
    Therefore, the properties \eqref{equ:eikonal equation,property} follow from the initial conditions \eqref{equ:eikonal equation,initial condition}.

    To construct a solution to Equation~\eqref{equ:eikonal equation,h}, we solve the iterative scheme
    \begin{equation}\label{equ:eikonal,iteration equation}
    \left\{
    \begin{aligned}
    \partial_t a_{N,j+1}^{\ve}+\nabla \varphi_{\mathrm{eik}} \cdot \nabla a_{N,j+1}^{\ve}+\tfrac{1}{2} a_{N,j+1}^{\ve} \Delta \varphi_{\mathrm{eik}} & =i \tfrac{\ve}{2} \Delta  a_{N,j+1}^{\ve}-i\ve^{-1}F( a_{N,j}^{\ve}) a_{N,j}^{\ve} \\
     a_{N,j}^{\ve}(0) & =a^{\ve,\mathrm{in}}
    \end{aligned}
    \right. \ .
    \end{equation}
    By the standard argument of hyperbolic PDE theory (see, e.g., \cite[Proposition 1]{carles2007wkb}), the problem boils down to obtaining energy estimates for Equation~\eqref{equ:eikonal,iteration equation} in $H^{s}(\R^3)$.
    Let $\sigma \in \N_0^{3}$ with $|\sigma|=s$. Applying $\partial_x^{\sigma}$ to \eqref{equ:eikonal,iteration equation}, we have
    \begin{align}\label{equ:eikonal,iteration equation,higher order}
    \partial_{t} \partial_{x}^{\sigma} a_{N}^{\ve}+\nabla \varphi_{\mathrm{eik}} \cdot \nabla \partial_{x}^{\sigma} a_{N}^{\ve}=i \tfrac{\ve}{2}\, \Delta \partial_{x}^{\sigma} a_{N}^{\ve}-i \ve^{-1} \partial_{x}^{\sigma}\left(F(a_{N}^{\ve})\right)+R_{N}^{\ve,\sigma}\ ,
    \end{align}
    where
    \begin{align*}
    R_{N}^{\ve,\sigma}=&\com{\nabla \varphi_{\mathrm{eik}} \cdot \nabla, \partial_{x}^{\sigma}} a_{N}^{\ve}-\tfrac{1}{2} \partial_{x}^{\sigma}\left(a_{N}^{\ve} \Delta \varphi_{\mathrm{eik}}\right)\\
    =&\pa_{x}^{\sigma}\nabla \varphi_{\mathrm{eik}}\cdot \nabla a_{N}^{\ve}-\tfrac{1}{2}\pa_{x}^{\sigma}\,a_{N}^{\ve}\Delta\,\varphi_{\mathrm{eik}}-
    \tfrac{1}{2}a_{N}^{\ve}\Delta\,\varphi_{\mathrm{eik}}\ .
    \end{align*}
    We take the inner product of \eqref{equ:eikonal,iteration equation,higher order} with $\partial_{x}^{\sigma} a_{N}^{\ve}$, and consider the real part. We then have
    \begin{align*}
    &\frac{1}{2} \frac{\d}{\d t}\Nrm{\partial_{x}^{\sigma} a_{N}^{\ve}}{L^2_x}^2+\operatorname{Re} \int_{\mathbb{R}^{3}} \overline{\partial_{x}^{\sigma} a_{N}^{\ve}} \nabla \vp_{\mathrm{eik }} \cdot \nabla \partial_{x}^{\sigma} a_{N}^{\ve}\dd x \\
    &\leq \ve^{-1}\left\|F(a_{N}^{\ve}) \right\|_{H^{s}_x}\Nrm{a_{N}^{\ve}}{H^{s}_x}
    +\Nrm{R_{N,\ve}^{\sigma}}{L^2_x}\Nrm{a_{N}^{\ve}}{H^{s}_x}\ ,
    \end{align*}
    where the first term of the right-hand side of \eqref{equ:eikonal,iteration equation,higher order} vanishes.
    On the one hand, applying Young's inequality and fractional Leibniz rule yields
    \begin{align*}
    \Nrm{F(a_{N}^{\ve})}{H^{s}_x} \leq&\, \ve^{-1}\la(N)\Nrm{vf^\mu_L}{L^1_x} \Nrm{|a_{N}^{\ve}|^{2}}{H^{s}_x}
    \lesssim \ve^{1-2\ka}\Nrm{a_{N}^{\ve}}{L^{\infty}_x}\Nrm{a_{N}^{\ve}}{H^{s}_x}
    \end{align*}
    where in the last inequality we have used Part~\eqref{part:approximation_of_scattering_length} of Lemma~\ref{lem:appendix_version_correlation_structures}.
    On the other hand, we have
    $$
    \begin{aligned}
    \left|\operatorname{Re} \int_{\mathbb{R}^{3}} \overline{\partial_{x}^{\sigma} a_{N}^{\ve}} \nabla \varphi_{\mathrm{eik}} \cdot \nabla \partial_{x}^{\sigma} a_{N}^{\ve}\dd x\right|
    & =\frac{1}{2}\int_{\mathbb{R}^{3}}|\partial_{x}^{\sigma} a_{N}^{\ve}|^2 \Delta \varphi_{\mathrm{eik}}\dd x \lesssim \Nrm{a_{N}^{\ve}}{H^{s}_x}^2\ ,
    \end{aligned}
    $$
    where in the last inequality we have used that $\Delta \varphi_{\mathrm{eik}} \in L^{\infty}\left([0,T_{0}] \times \mathbb{R}^{3}\right)$.
    In addition, by the boundedness of $\varphi_{\mathrm{eik}}$, we also have
    \begin{align*}
    \Nrm{R_{N,\ve}^{\sigma}}{L^{2}_x}\lesssim&\, \Nrm{\pa_{x}^{\sigma}\nabla \varphi_{\mathrm{eik}}}{L^{\infty}_x}\Nrm{\nabla a_{N}^{\ve}}{L^{2}_x}+\Nrm{\pa_{x}^{\sigma}a_{N}^{\ve}}{L^{2}_x}\Nrm{\Delta\varphi_{\mathrm{eik}}}{L^{\infty}_x} +
    \Nrm{a_{N}^{\ve}}{L^{2}_x}\nv{\Delta\varphi_{\mathrm{eik}}}_{L^{\infty}_x}
    \lesssim \Nrm{a_{N}^{\ve}}{H^{s}_x}.
    \end{align*}
    Therefore, for $\ka\in [0,\frac{1}{2}]$, by summing over $\sigma$ such that $|\sigma|=s$, we arrive at
    \begin{align*}
    \frac{\d}{\d t}\Nrm{a_{N}^{\ve}}{H^{s}_x}^{2} \lesssim&\, (\Nrm{a_{N}^{\ve}}{L^{\infty}_x}+1)\Nrm{a_{N}^{\ve}}{H^{s}_x}^{2}
    \lesssim (\Nrm{a_{N}^{\ve}}{H^{s}_x}+1)\Nrm{a_{N}^{\ve}}{H^{s}_x}^{2}\ .
    \end{align*}
    By Gronwall's lemma, we conclude that there exists $T_{0}>0$ such that
    $$
    \Nrm{a_{N}^{\ve}}{{L^{\infty}_t([0,T_{0}])H^{s}_x}} \leq C\left(s,\left\|a^{\ve,\mathrm{in}}\right\|_{H^{s}_x}\right)\ .
    $$
    This yields uniform boundedness in the large norm. Since the convergence in the small norm (that is, contraction in $L^2$) follows from a similar way, we omit it for simplicity.
\end{proof}

\begin{lem}\label{lem:convergence,aNh-a}
    Under the assumption of Proposition~\ref{prop:semiclassical limit,wkb}, the following holds:
    \begin{enumerate}[$(i)$]
     \item For the critical case that $\ka=\frac{1}{2}$, we have
    \begin{align}\label{equ:convergence,a,ka=1}
    \nv{a_{N}^{\ve}-a}_{L^{\infty}_t([0,T_{0}])H^{s}_x}\lesssim \nv{a^{\ve,\mathrm{in}}-a^{\mathrm{in}}}_{H^{s}_x}+\ve+\eta(\ve)\ ,
    \end{align}
    where $a(t)$ satisfies
    \begin{align}\label{equ:eikonal equation,amplitude}
    \pa_{t}a+\nabla \varphi_{\mathrm{eik}}\cdot \nabla a+\tfrac{1}{2}a\,\Delta \varphi_{\mathrm{eik}}=-i 4\pi \mathfrak{c}_{0}|a|^{2}a\ ,
    \end{align}
    with $\mathfrak{c}_{0}$ being the capacity of the potential $v$.
    \item For the subcritical case that $\ka=0$, we have
    \begin{align}\label{equ:convergence,a,ka=0}
    \nv{a_{N}^{\ve}-a}_{L^{\infty}_t([0,T_{0}])H^{s}_x}\lesssim \nv{a^{\ve,\mathrm{in}}-a^{\mathrm{in}}}_{H^{s}_x}+\ve\ ,
    \end{align}
    where $a(t)$ satisfies the transport equation
    \begin{align}\label{equ:transport equation,amplitude}
    \pa_{t}a+\nabla \varphi_{\mathrm{eik}}\cdot \nabla a+\tfrac{1}{2}a\,\Delta \varphi_{\mathrm{eik}}=0\ .
    \end{align}
    \end{enumerate}
\end{lem}

\begin{proof}
    We can also obtain the local well-posedness of Equations~\eqref{equ:eikonal equation,amplitude} and \eqref{equ:transport equation,amplitude} by the standard argument of hyperbolic PDE theory as shown in Lemma~\ref{lem:well-posedness,uniform estimate,a}. Thus, $a(t)$ is well-defined on $[0,T_{0}]$. Next, we only need to provide the proof of the $\ka=\frac{1}{2}$ case, as the $\ka=0$ case follows from a similar way.

    \bigskip
    \noindent\textbf{The $\ka=\frac{1}{2}$ case.}  Let $\mathrm{d}_{N}^{\ve}=a_{N}^{\ve}-a$.
    Taking a difference between Equations~\eqref{equ:eikonal equation,h} and \eqref{equ:eikonal equation,amplitude},
     we have
    \begin{align*}
    \pa_{t}\mathrm{d}_{N}^{\ve}+\nabla \varphi_{\mathrm{eik}}\cdot \nabla \mathrm{d}_{N}^{\ve}+\tfrac{1}{2}\mathrm{d}_{N}^{\ve}\,\Delta \varphi_{\mathrm{eik}}
    =i\frac{\ve}{2}\Delta \mathrm{d}_{N}^{\ve}+i\frac{\ve}{2}\Delta a-
    i\ve^{-1}F(a_{N}^{\ve})a_{N}^{\ve}-i4\pi \mathfrak{c}_{0}|a|^{2}a\ .
    \end{align*}
    Repeating the proof of the $H^{s}$ energy estimate in Lemma~\ref{lem:well-posedness,uniform estimate,a}, by Young's inequality and the boundedness of $a(t)$, we arrive at
    \begin{equation}\label{equ:difference,error term}
        \begin{aligned}
            \frac{\d}{\d t}\nv{\mathrm{d}_{N}^{\ve}}_{H^{s}_x}^{2}\lesssim&\, \nv{\mathrm{d}_{N}^{\ve}}_{H^{s}_x}^{2}+\ve\nv{\Delta a}_{H^{s}_x}\nv{\mathrm{d}_{N}^{\ve}}_{H^{s}_x}
            +\Nrm{\ve^{-1}F(a_{N}^{\ve})a_{N}^{\ve}-4\pi \mathfrak{c}_{0}|a|^{2}a}{H^{s}_x}\nv{\mathrm{d}_{N}^{\ve}}_{H^{s}_x}\\
            \lesssim&\, \nv{\mathrm{d}_{N}^{\ve}}_{H^{s}_x}^{2}+\ve^{2}
            +\Nrm{\ve^{-1}F(a_{N}^{\ve})a_{N}^{\ve}-4\pi \mathfrak{c}_{0}|a|^{2}a}{H^{s}_x}^{2}
        \end{aligned}
    \end{equation}
    We are left to control the last term on the right hand side of Estimate~\eqref{equ:difference,error term}.
    By the triangle inequality, we have that
    \begin{align*}
    &\Nrm{\ve^{-1}F(a_{N}^{\ve})a_{N}^{\ve}-4\pi \mathfrak{c}_{0}|a|^{2}a}{H^{s}_x}\\
    \leq& \Nrm{\ve^{-1}F(a_{N}^{\ve})a_{N}^{\ve}-\ve^{-1}F(a)a}{H^{s}_x}
    +\Nrm{\ve^{-1}F(a)a-4\pi \mathfrak{c}_{0}|a|^{2}a}{H^{s}_x}=:\, I+II\ .
    \end{align*}

    For the term $I$, using the fact that $H^{s}(\R^3)$ is a Banach algebra for $s>\frac{3}{2}$, a simple calculation yields that
\begin{align*}
I\leq&\, C(\Nrm{a_{N}^{\ve}}{H^{s}_x},\Nrm{a}{H^{s}_x})
\Nrm{\ve^{-1}K}{L^{1}_x}\Nrm{\mathrm{d}_{N}^{\ve}}{H^{s}_x}\lesssim\Nrm{\mathrm{d}_{N}^{\ve}}{H_x^{s}}\ ,
\end{align*}
where in the last inequality we have used that $\ve^{-1}\Nrm{K}{L^{1}_x}\lesssim 1$, which
 follows from Identity~\eqref{eq:approximation_of_scattering_length} for $\ka=\frac{1}{2}$.

For the term $II$, by the fractional Leibniz rule, we get
\begin{align*}
\Nrm{\ve^{-1}F(a)a-4\pi \mathfrak{c}_{0}|a|^{2}a}{H^{s}_x}\lesssim& \Nrm{(\ve^{-1}K-4\pi \mathfrak{c}_{0}\delta)*\lra{\nabla}^{s}|a|^{2}}{L^{2}_x}\Nrm{a}{L^{\infty}_x}\\
& +
\Nrm{(\ve^{-1}K-4\pi \mathfrak{c}_{0}\delta)*|a|^{2}}{L^{2}_x}\Nrm{\lra{\nabla}^{s} a}{L^{\infty}_x}\ .
\end{align*}
Due to the boundedness of $\lra{\nabla}^{s}a$, we only need to deal with the term $\ve^{-1}K-4\pi \mathfrak{c}_{0}\delta_0$ which provides the smallness.
By triangle inequality, we obtain
    \begin{align}\label{equ:difference,error term,II}
       & \Nrm{(\ve^{-1}K-4\pi \mathfrak{c}_{0}\delta)*\lra{\nabla}^{s}|a|^{2}}{L^{2}_x}\\
        \leq& \bbabs{\textstyle\intd \ve^{-1} K\dd x-4\pi \mathfrak{c}_{0}}\Nrm{\weight{\nabla}^{s}|a|^{2}}{L^{2}_x} +\bbn{\lrs{\textstyle \ve^{-1}K-\delta\intd \ve^{-1} K\dd x}*\lra{\nabla}^{s}|a|^{2}}_{L^{2}_x}\ . \notag
    \end{align}
For the first term on the right hand side of \eqref{equ:difference,error term,II}, we use the identity approximation \eqref{equ:quantative estimate for identity approximation} in Lemma~\ref{lem:quantative estimate for identity approximation} to get
\begin{multline}\label{equ:difference,error term,II,identity}
\bbn{\lrs{\textstyle \ve^{-1}K-\delta\intd \ve^{-1} K\dd x}*\lra{\nabla}^{s}|a|^{2}}_{L^{2}_x}
\lesssim \frac{\la(N)}{N\ve^2}\Nrm{\lra{x}vf_{L}^{\mu}}{L^{1}_x}\Nrm{\lra{\nabla}^{s+1}|a|^{2}}{L^{2}_x}
\lesssim \frac{\la(N)}{N\ve^{2}}
\end{multline}
where in the last inequality we have used that $0\leq f_{L}^{\mu}\leq 1$.

For the second term on the right hand side of Inequality~\eqref{equ:difference,error term,II},
by Identity~\eqref{eq:approximation_of_scattering_length} and Proposition \ref{prop:quantitative_scattering_length}, we have
\begin{align}\label{equ:difference,error term,c0}
\bbabs{\intd \ve^{-1} K\dd x-4\pi \mathfrak{c}_{0}}
\le \bbabs{\intd \ve^{-1} K\dd x-4\pi \asc_{0}^{\mu}}+4\pi \n{\asc_{0}^{\mu}-\mathfrak{c}_{0}}  \lesssim \displaystyle\frac{\la(N)}{L}+\eta(\ve)\ ,
\end{align}
where $\eta(\ve)$ is the rate of convergence of $\asc_{0}^{\mu}$.

Combining Inequalities~\eqref{equ:difference,error term,II}--\eqref{equ:difference,error term,c0}, we arrive at the estimate
\begin{align*}
II\lesssim \frac{\la(N)}{N\ve^{2}}+\frac{\la(N)}{L}+\eta(\ve)\ .
\end{align*}
Plugging the estimates for $I$ and $II$ into Inequality~\eqref{equ:difference,error term}, we arrive at
\begin{align*}
\frac{\d}{\d t}\Nrm{\mathrm{d}_{N}^{\ve}}{H^{s}_x}^{2}\lesssim \Nrm{\mathrm{d}_{N}^{\ve}}{H^{s}_x}^{2}+\ve^{2}+\lrs{\frac{\la(N)}{N\ve^{2}}+\frac{\la(N)}{L}+\eta(\ve)}^{2}.
\end{align*}
Therefore, by Gronwall's lemma, we conclude the proof of Inequality~\eqref{equ:convergence,a,ka=1}.

\end{proof}

We are now ready to prove the main proposition.

\begin{proof}[\textbf{Proof of Proposition~$\ref{prop:semiclassical limit,wkb}$}]
    For the convergence of mass density, we use the uniform bound of $a_{N}^{\ve}$ proved in Lemma~\ref{lem:well-posedness,uniform estimate,a} to get
    \begin{align*}
    \Nrm{\rho^{\ve}-|a|^{2}}{L^{p}_x}\leq&\Nrm{a_{N}^{\ve}-a}{L^{2}_x}\lrs{\Nrm{a_{N}^{\ve}}{L^{\frac{2p}{p-2}}_x}+\Nrm{a}{L^{\frac{2p}{p-2}}_x}}
    \lesssim \Nrm{a_{N}^{\ve}-a}{L^{2}_x}\ .
    \end{align*}

    For the convergence of momentum density, noting that
    \begin{align*}
        i\ve \nabla \phi_{N}^{\ve}=i\ve(\nabla a_{N}^{\ve})e^{i\varphi_{\mathrm{eik}}/\ve}-(\nabla \varphi_{\mathrm{eik}})a_{N}^{\ve}e^{i\varphi_{\mathrm{eik}}/\ve}\ ,
    \end{align*}
    which implies
    \begin{align*}
    \vect{J}^{\ve}=\im\lrs{\ve\nabla \phi_{N}^{\ve}\,\ol{\phi_{N}^{\ve}}}=
    \im\lrs{\ve\nabla a_{N}^{\ve}\,\ol{a_{N}^{\ve}}}+|a_{N}^{\ve}|^{2}\nabla \varphi_{\mathrm{eik}}\ ,
    \end{align*}
    then we  have
    \begin{align*}
    &\Nrm{\vect{J}^{\ve}-|a|^{2}\nabla \varphi_{\mathrm{eik}}}{L^{p}_x}\\
    \leq& \ve\Nrm{a_{N}^{\ve}}{L^{\frac{2p}{2-p}}_x}\nv{\nabla a_{N}^{\ve}}_{L^{2}_x}+\Nrm{a_{N}^{\ve}-a}{L^{2}_x}\lrs{\Nrm{a_{N}^{\ve}}{L^{\frac{2p}{2-p}}_x}+\Nrm{a}{L^{\frac{2p}{2-p}}_x}}
    \Nrm{\nabla \varphi_{\mathrm{eik}}}{L^{\infty}_x}\\
    \lesssim& \ve+\Nrm{a_{N}^{\ve}-a}{L^{2}_x}\ .
    \end{align*}

    Lastly, for the convergence of energy, note that
    \begin{align*}
    &\Nrm{|i\varepsilon\nabla \phi_{N}^{\ve}|^{2}-|a|^{2}|\nabla \varphi_{\mathrm{eik}}|^{2}}{L^{1}_x}\\
    \leq& \Nrm{\lrs{|a_{N}^{\ve}|^{2}-|a|^{2}}|\nabla \varphi_{\mathrm{eik}}|^{2}}{L^{1}_x}+2\ve\Nrm{(\nabla a_{N}^{\ve}\cdot \nabla \varphi_{\mathrm{eik}}) a_{N}^{\ve}}{L^{1}_x}+\ve^{2}\Nrm{\nabla a_{N}^{\ve}}{L^{2}_x}^{2}\\
    \lesssim& \Nrm{a_{N}^{\ve}-a}{L^{2}_x}+\ve\ .
    \end{align*}

    Therefore, combining Estimates~\eqref{equ:convergence,a,ka=1} and \eqref{equ:convergence,a,ka=0}, we arrive at the desired Estimates~\eqref{equ:semiclassical limit,convergence rate,k=1} and \eqref{equ:semiclassical limit,convergence rate,k=0}.
\end{proof}

\section{Proof of the Main Results}\label{sect:proof_of_main_results}

In this section, we prove the main results based on the quantitative estimates established in Sections~\ref{section:bogoliubov_approximation}--\ref{section:Semiclassical Limit of the Modified Gross--Pitaevskii Equation}.
As shown in Section~\ref{section:Semiclassical Limit of the Modified Gross--Pitaevskii Equation}, for the $\beta=1$ situation there are three typical cases when studying the limiting behavior of the modified Gross--Pitaevskii equation, $\ka\in \{0, 1/2, 1\}$, which lead to different effective limiting equations. To simplify the notations in the proof, we shall only provide the proofs for the three typical cases since the proof for more general $\ka\in(0,1)$ is similar.

\begin{proof}[Proof of Theorem \ref{thm:main_result_GP-HC}]
By the triangle inequality, we have
\begin{align*}
\nrm{\rho_{N:1}^{\varepsilon}-\rho}_{L_{x}^{2}}\leq&\, \nrm{\rho_{N:1}^{\varepsilon}-\rho^{\varepsilon}}_{L_{x}^{2}}+
 \nrm{\rho^{\varepsilon}-\rho}_{L_{x}^{2}}\ ,\\
\nrm{\vect{J}_{N:1}^{\varepsilon}-\vect{J}}_{L_{x}^{1}}\leq&\, \nrm{\vect{J}_{N:1}^{\varepsilon}-\vect{J}^{\varepsilon}}_{L_{x}^{1}}+
 \| \vect{J}^{\varepsilon}-\vect{J}\|_{L_{x}^{1}}\ .
\end{align*}
Hence, by Proposition~\ref{prop:approximation,density_and_momentum} and Proposition~\ref{prop:semiclassical limit and the modulated energy}, we complete the proof of the convergence of mass and momentum densities for the GP/HC regimes.
Next, we deal with the convergence of the energy densities.

For the GP Regime, the parameters are given by
\begin{align*}
\textstyle \beta=1,\ \kappa=1,\ \la=1,\ \mu=1,\ \mathfrak{b}_{0}=\mathfrak{a}_{0}-\frac{1}{4\pi}
\int_{\mathbb{R}^{3}}v(f_{0})^{2}\dd x\ .
\end{align*}
By the triangle inequality, we have
\begin{align*}
\Nrm{E_{N,\mathrm{kin.}}^{\varepsilon}-\tfrac{1}{2}\rho|\bu|^{2}-4\pi \mathfrak{b}_{0}\rho^{2} }{L_{x}^{1}}
\leq&\, \Nrm{E_{N,\mathrm{kin.}}^{\varepsilon}-\tfrac{1}{2}|\ve\nabla \phi_{N}^{\ve}|^{2}-4\pi \mathfrak{b}_{0}|\phi_{N}^{\ve}|^{4}}{L_{x}^{1}}\\
&+\frac{1}{2}\Nrm{|\ve\nabla \phi_{N}^{\ve}|^{2}-\rho|\bu|^{2}}{L_{x}^{1}}+4\pi \mathfrak{b}_{0}\Nrm{|\phi_{N}^{\ve}|^{4}-\rho^{2}}{L_{x}^{1}}\ ,
\end{align*}
and
\begin{multline*}
	\Nrm{E_{N,\mathrm{int.}}^{\ve}-4\pi (\mathfrak{a}_{0}-\mathfrak{b}_{0})\rho^{2}}{L_{x}^{1}}
	\leq\Nrm{E_{N,\mathrm{int.}}^{\ve}-4\pi (\mathfrak{a}_{0}-\mathfrak{b}_{0})|\phi_{N}^{\ve}|^{4}}{L_{x}^{1}}+4\pi |\mathfrak{a}_{0}-\mathfrak{b}_{0}|\Nrm{|\phi_{N}^{\ve}|^{4}-\rho^{2}}{L_{x}^{1}}.
\end{multline*}
By Proposition \ref{prop:many-body_energies_mean-field_approximation} and Proposition \ref{prop:semiclassical limit and the modulated energy}, we obtain
   \begin{equation}
            \left\{
            \begin{aligned}
                &\Nrm{E_{N,\mathrm{kin.}}^{\varepsilon}-\tfrac{1}{2}\rho\n{\bu}^2-4\pi \mathfrak{b}_0 \rho^{2}}{L^{\infty}_t([0,T_{0}])L^{1}_x}^{2}
                \lesssim \cM^\init+\varepsilon^{2}+\tfrac{1}{\ln N}\ ,\\
                &\Nrm{E_{N,\mathrm{int.}}^{\varepsilon}-4\pi (\mathfrak{a}_{0}-\mathfrak{b}_{0})\rho^{2}}{L_{t}^{\infty}([0, T_{0}])L_{x}^{1}}^{2}\lesssim
                \cM^\init+\varepsilon^{2}+\tfrac{1}{\ln N}\ ,
            \end{aligned}
            \right.
         \end{equation}
and hence conclude the convergence of the energy densities.

For the HC Regime, the parameters are given by
 \begin{align*}
 \textstyle \beta=1,\ \kappa=1,\ \la=(\ln N)^{\upalpha},\ \mu=\frac{1}{\la},\ \mathfrak{b}_{0}^{\mu}=\mathfrak{a}_{0}^{\mu}-\frac{1}{4\pi\mu} \int_{\mathbb{R}^{3}}v(f_{0}^{\mu})^{2}\dd x\ .
 \end{align*}
By the triangle inequality, we obtain
\begin{align*}
\Nrm{E_{N,\mathrm{kin.}}^{\varepsilon}-\tfrac{1}{2}\rho|\bu|^{2}-4\pi \mathfrak{c}_{0}\rho^{2} }{L_{x}^{1}} \leq&\, \Nrm{E_{N,\mathrm{kin.}}^{\varepsilon}-\tfrac{1}{2}|\ve\nabla \phi_{N}^{\ve}|^{2}-4\pi \mathfrak{b}_{0}^{\mu}|\phi_{N}^{\ve}|^{4}}{L_{x}^{1}}\\
&+\frac{1}{2}\Nrm{|\ve\nabla \phi_{N}^{\ve}|^{2}-\rho|\bu|^{2}}{L_{x}^{1}}+4\pi \mathfrak{b}_{0}^{\mu}\Nrm{|\phi_{N}^{\ve}|^{4}-\rho^{2}}{L_{x}^{1}}\\
&+4\pi\n{\mathfrak{a}_{0}^{\mu}-\mathfrak{c}_{0}}
\| \rho^{2}\|_{L_{x}^{1}}+4\pi |\asc_0^{\mu}-\mathfrak{b}_{0}^{\mu}|\| \rho^{2}\|_{L_{x}^{1}}\ .
\end{align*}
By Proposition \ref{prop:many-body_energies_mean-field_approximation}, Proposition \ref{prop:semiclassical limit and the modulated energy}, Proposition \ref{prop:quantitative_scattering_length}, and Inequality~\eqref{equ:semiclassical estimates,b0} in Lemma \ref{lem,semiclassical estimates,scattering function,appendix}, we complete the proof of the kinetic energy part.

For the interaction energy part, by the triangle inequality, we have
\begin{align*}
	\|E_{N,\mathrm{int.}}^{\ve}\|_{L_{x}^{1}}
	\leq\|E_{N,\mathrm{int.}}^{\ve}-4\pi (\mathfrak{a}_{0}^{\mu}-\mathfrak{b}_{0}^{\mu})|\phi_{N}^{\ve}|^{4}\|_{L_{x}^{1}}+4\pi |\mathfrak{a}_{0}^{\mu}-\mathfrak{b}_{0}^{\mu}|\||\phi_{N}^{\ve}|^{4}\|_{L_{x}^{1}}\ .
\end{align*}
Using Proposition \ref{prop:many-body_energies_mean-field_approximation}, \eqref{equ:semiclassical estimates,b0} in Lemma \ref{lem,semiclassical estimates,scattering function,appendix}, and the $L^{4}$ uniform bound for $\phi_{N}^{\ve}$ in Lemma \ref{lem:phi4_bound}, we arrive at
\begin{align}
&\|E_{N,\mathrm{int.}}^{\ve}\|_{L_{x}^{1}}\lesssim \frac{1}{\ln N}+\eta(\frac{1}{\la(N)})\ .
\end{align}
Hence, we have completed the proof of Theorem \ref{thm:main_result_GP-HC}.
\end{proof}

\begin{proof}[Proof of Theorem \ref{thm:main_result_HD}]
Using again the triangle inequality, we have
\begin{align*}
	\nrm{\rho_{N:1}^{\varepsilon}-\rho}_{L_{x}^{p}}\leq&\, \nrm{\rho_{N:1}^{\varepsilon}-\rho^{\varepsilon}}_{L_{x}^{p}}+
	\nrm{\rho^{\varepsilon}-\rho}_{L_{x}^{p}}\ ,\\
	\nrm{\vect{J}_{N:1}^{\varepsilon}-\vect{J}}_{L_{x}^{1}}\leq&\, \nrm{\vect{J}_{N:1}^{\varepsilon}-\vect{J}^{\varepsilon}}_{L_{x}^{1}}+
	\| \vect{J}^{\varepsilon}-\vect{J}\|_{L_{x}^{1}}\ .
\end{align*}
Hence, by Proposition~\ref{prop:approximation,density_and_momentum} and Propositions~\ref{prop:semiclassical limit,wkb}, we complete the proof of the convergence of mass and momentum densities for the BGP/SGP regimes.
Next, we deal with the convergence of the energy densities. It suffices to consider the BGP regime, as the others follow similarly.
For the BGP Regime, the parameters are given by
\begin{align*}
\textstyle	\beta=1,\ \kappa=\frac{1}{2},\ \la\in\{1,(\ln N)^{\upalpha}\},\ \mu=\frac{\ve}{\la},\ \mathfrak{b}_{0}^{\mu}=\mathfrak{a}_{0}^{\mu}-\frac{1}{4\pi\mu}
	\int_{\mathbb{R}^{3}}v(f_{0}^{\mu})^{2}\dd x\ .
\end{align*}
By the triangle inequality, we have
\begin{align*}
	\|E_{N,\mathrm{kin.}}^{\varepsilon}-\frac{1}{2}|a|^{2}|\nabla \varphi_{\mathrm{eik}}|^{2}  \|_{L_{x}^{1}} \leq&\, \|E_{N,\mathrm{kin.}}^{\varepsilon}-\tfrac{1}{2}|\ve\nabla \phi_{N}^{\ve}|^{2}-4\pi\mu \la\mathfrak{b}_{0}^{\mu}|\phi_{N}^{\ve}|^{4}\|_{L_{x}^{1}}\\
	&+\frac{1}{2}\||\ve\nabla \phi_{N}^{\ve}|^{2}-|a|^{2}|\nabla \varphi_{\mathrm{eik}}|^{2}\|_{L_{x}^{1}}+4\pi \ve\mathfrak{b}_{0}^{\mu}\||\phi_{N}^{\ve}|^{4}\|_{L_{x}^{1}}\ ,
\end{align*}
and
\begin{align*}
	\|E_{N,\mathrm{int.}}^{\ve}\|_{L_{x}^{1}}
	\leq\|E_{N,\mathrm{int.}}^{\ve}-4\pi \mu \la(\mathfrak{a}_{0}^{\mu}-\mathfrak{b}_{0}^{\mu})|\phi_{N}^{\ve}|^{4}\|_{L_{x}^{1}}+4\pi\ve |\mathfrak{a}_{0}^{\mu}-\mathfrak{b}_{0}^{\mu}|\||\phi_{N}^{\ve}|^{4}\|_{L_{x}^{1}}\ .
\end{align*}
By Proposition \ref{prop:many-body_energies_mean-field_approximation}, Propositions~\ref{prop:semiclassical limit,wkb},
and the $L^{4}$ uniform bound for $\phi_{N}^{\ve}$ in Lemma \ref{lem:well-posedness,uniform estimate,a}, we obtain
\begin{align*}
		&\Nrm{E_{N,\mathrm{kin.}}^{\varepsilon}-\tfrac{1}{2}|a|^{2}|\nabla \varphi_{\mathrm{eik}}|^{2}}{L_{x}^{1}} \lesssim
		\frac{1}{\ln N}+\nv{a^{\ve,\mathrm{in}}-a^{\mathrm{in}}}_{H^{s}_x}+\ve+\eta(\ve)\ ,\\
		&\Nrm{E_{N,\mathrm{int.}}^{\ve}}{L_{x}^{1}}\lesssim
		\frac{1}{\ln N}+\ve\ .
\end{align*}
Hence, we conclude the convergence of the energy densities.
\end{proof}

\noindent \textbf{Acknowledgements} The authors would like to express gratitude to Xuwen Chen for helpful discussions and to Arnaud Triay for pointing out the connection between our problem's setting and the beyond Gross--Pitaevskii scaling regime.
\appendix

\section{Semiclassical Analysis of the Two-Body Scattering Problem}\label{appendix:semiclassical_two-body_problem}

In this appendix, we collection some results regarding the semiclassical analysis of the two-body scattering problem.

\subsection{Proof of Proposition \ref{prop:quantitative_scattering_length}}

\begin{proof}
    Notice, by rescaling Problem \eqref{def:zero-energy_scattering_problem} by some fixed constants, we could without loss of generality choose $\Nrm{v}{L^{\infty}_x}=1$ and $\mathfrak{c}_{0}=R_{0}=1$. Also, we could also assume $v_{1}(x)\equiv 1$ when $|x|=1$.

    Setting $m(r)=rf_{0}^{\mu}(r)$, we recast Problem \eqref{def:zero-energy_scattering_problem}
to the Sturm--Liouville problem
\begin{align}\label{eq:Sturm-Liouville_problem}
 -\mu\,m''+v\, m = 0 \quad \text{ with } \quad  m(0)=0,\quad \lim_{r\rightarrow \infty} \frac{m(r)}{r}= 1\ .
\end{align}
Note that the last boundary condition is equivalent to having $m'(R_0)=1$.
We employ the comparison method which we illustrate in Figure \ref{figure:comparison_method}.
    \begin{figure}[!htb]
        \begin{tikzpicture}
            \node[draw,rectangle,inner sep=0.1cm,outer sep=0.1cm] (1) at (8,3) {\makecell[l]{{\color{red}\textbf{------}} $\ol{v}(r)$\\ \textbf{------} $v(r)$\\
            {\color{blue}\textbf{------}} $\underline{v}(r)$} };
            \draw[->] (-1,0)--(6,0) node [right]{$r$};
            \draw[->] (0,-1)--(0,3.8);
            \draw (4.5,1)parabola [bend at end](5,0);
            \draw (0,2).. controls (0.5,1) and (0.7,3) .. (1,3).. controls (1.5,3) and (2.5,0.8).. (3,1.5)
            ..controls (3.8,3)..(4.5,1);
            \draw[very thick,blue] (0,1)--(4.5,1);
            \draw[very thick,blue] (4.5,0)--(5,0);
            \draw[dashed,blue] (4.5,1)--(4.5,0);
            \draw[very thick,red] (0,3)--(4.5,3);
            \draw[very thick,red] (4.5,1)--(5,1);
            \draw[dashed,red] (4.5,3)--(4.5,1);
            \draw[dashed,red] (5,1)--(5,0);
            \draw[decorate,decoration={brace,raise=2pt}] (5,0)--(4.5,0);
            \draw[decorate,decoration={brace,raise=2pt}] (4.5,0)--(4.5,1);
            \node at (4.75,-0.5) {$\theta$};
            \node at (4,0.5) {$\theta^{n}$};
            \node[left] at(0,3) {1};
            \node[below left] at(0,0) {0};
            \node[below right] at(5,0) {1};
        \end{tikzpicture}
        \caption{}\label{figure:comparison_method}
    \end{figure}
    In particular, we have the bound
    \begin{align}
         \theta^{n}\id_{[0,1-\theta)}(r)=:\underline{v}(r)\leq v(r)\leq \ol{v}(r):=\id_{[0,1-\theta)}(r)+\theta^{n}\id_{[1-\theta,1]}(r)\ ,
    \end{align}
    where the parameter $\theta$, which depends on $\mu$, is to be determined. Here, we have used the vanishing condition of $v$. Hence we consider the corresponding second-order equations
    \begin{equation}\label{eq:second order odes}
    \begin{cases}
        -\mu\,\ol{m}'' +\ol{v}\,\ol{m} =0 & \text{ with }\quad  \ol{m}(0)=0,\ \ol{m}'(1)=1\ ,\\
        -\mu\, m''+v\,m =0 \quad & \text{ with } \quad  m(0)=0,\ m'(1)=1\ ,\\
        -\mu\, \underline{m}''+\underline{v}\,\underline{m}=0& \text{ with } \quad  \underline{m}(0)=0,\ \underline{m}'(1-\theta)=1\ .
    \end{cases}
    \end{equation}

    Notice, by direct computation, we have that
    \begin{equation}\label{eq:lower bound solution}
    \underline{m}(r)=
        \begin{cases}
            \displaystyle \sqrt{\frac{\mu}{\theta^{n}}}\frac{\sinh\(\sqrt{\frac{\theta^{n}}{\mu}}r\)}{
            \cosh\(\sqrt{\frac{\theta^{n}}{\mu}}(1-\theta)\)}& \text{ if }  r\in (0,1-\theta],\\[10pt]
            r-\underline{\mathfrak{a}}_{0}(\mu,\theta) & \text{ if } r\geq [1-\theta,\infty),
        \end{cases}
    \end{equation}
    and
    \begin{equation}\label{eq:upper bound solution}
        \ol{m}(r)=
        \begin{cases}
        A\sinh\(\frac{r}{\sqrt{\mu}}\)    &\text{ if } r\in(0,1-\theta)\ ,\\[10pt]
        \sqrt{\frac{\mu}{\theta^{n}}}\sinh\(\sqrt{\frac{\theta^{n}}{\mu}}(1-r)\)
        +B\cosh\(\sqrt{\frac{\theta^{n}}{\mu}}(1-r)\)   &\text{ if } r\in(1-\theta,1)\ ,\\[10pt]
        r-\ol{\mathfrak{\mathfrak{a}}}_{0}(\mu,\theta)  &\text{ if } r\geq 1\ ,
        \end{cases}
    \end{equation}
    where $\ol{\asc}_0$ and $\underline{\asc}_0$ are the corresponding scattering length.
    The coefficients $A$ and $B$ are uniquely determined by the compatibility conditions
    \begin{align}\label{eq:compatibility_conditions}
        \lim_{r\nearrow (1-\theta)}\ol{m}(r)=\lim_{r\searrow (1-\theta)}\ol{m}(r)\quad\text{ and }\quad \lim_{r\nearrow (1-\theta)}\ol{m}'(r)=\lim_{r\searrow (1-\theta)}\ol{m}'(r).
    \end{align}
    Now, applying the standard comparison theorem for second-order linear equations (see, for instance, \cite[Chapter 8]{coddington1955theory}), we have that
    \begin{align*}
        0\le \ol{m}(r)\le m(r)\le \underline{m}(r)
    \end{align*}
    for all $r\ge 0$. In particular, by Formulae \eqref{eq:lower bound solution} and \eqref{eq:upper bound solution} and Conditions \eqref{eq:compatibility_conditions}, it follows that
    \begin{align*}
        1-\ol{\mathfrak{\mathfrak{a}}}_{0}(\mu,\theta) =\ol{m}(1)\le 1-\asc_{0}^{\mu}\le \underline{m}(1) =1-\underline{\mathfrak{\mathfrak{a}}}_{0}(\mu,\theta)
    \end{align*}
    where
    \begin{align*}
        &1-\underline{\mathfrak{\mathfrak{a}}}_{0}(\mu,\theta)=
         \sqrt{\tfrac{\mu}{\theta^{n}}}\tanh\(\sqrt{\tfrac{\theta^{n}}{\mu}}(1-\theta)\)+\theta\ ,\\
        &1-\ol{\mathfrak{\mathfrak{a}}}_{0}(\mu,\theta)=\frac{\sqrt{\frac{\mu}{\theta^{n}}}\tanh\(\sqrt{\frac{\theta^{n+2}}{\mu}}\)-\sqrt{\mu} \tanh\(\frac{1-\theta}{\sqrt{\mu}}\)}{1-\sqrt{\frac{\mu^{2}}{\theta^{n}}}
        \tanh\(\frac{1-\theta}{\sqrt{\mu}}\)\tanh\(\sqrt{\frac{\theta^{n+2}}{\mu}}\)}\ .
    \end{align*}
    By choosing $\theta=\mu^{\frac{1}{n+2}}$, we deduce that
    \begin{align}
    \ol{m}(1)=\mathcal{O}(\mu^{\frac{1}{n+2}})\quad \text{ and }\quad \underline{m}(1)\sim \sqrt{\frac{\mu}{\theta^{n}}}=\mathcal{O}(\mu^{\frac{1}{n+2}})
    \end{align}
    when $\mu$ is sufficiently small.
\end{proof}

\subsection{Semiclassical Estimates of Scattering Function}\label{section:dirichlet problem,appendix}
Let us summarize some of the properties of the scattering solution $f^\mu_0$ in the next lemma. For readers familiar with the literature on the Gross--Pitaevskii regime, the following lemma is nothing more than the semiclassical extension of result in \cite[Lemma 5.1]{erdos2010derivation}.
\begin{lem}\label{lem,semiclassical estimates,scattering function,appendix}
    Suppose $v$ is a nonnegative bounded function  supported in $\{\n{x}\le R_0\}$ and $\asc_{0}^{\mu}$ be the corresponding scattering length. Let $f_0^\mu$ be the solution to the Dirichlet Problem \eqref{def:zero-energy_scattering_problem}. Then we have the following:
    \begin{enumerate}[$(i)$]
        \item\label{part:zero-energy_scattering_function_pointwise_bound}
        We have the pointwise estimate
        \begin{align}\label{equ:zero-energy_scattering_function_pointwise_bound}
            \exp\(-\zeta_0^\frac12/\mu^{\frac12}\) \le c_{1}^{\mu}:=f_{0}^{\mu}(0)\le f_{0}^{\mu}(r) \le 1, \quad \text{ for all } r\ge 0\ ,
        \end{align}
        where $\zeta_0:= R_0^2  \Nrm{v}{L^\infty}$.

        \item\label{part:zero-energy_scattering_function_gradient_pointwise_bound} We have the following gradient estimates
        \begin{align}
            \n{\grad f_{0}^{\mu}(x)} \le  \frac{\asc_0^{\mu}}{\n{x}^2} \quad \text{ and } \quad \n{\grad f_{0}^{\mu}(x)} \le \frac{1-c_{1}^{\mu}}{\n{x}}\ .
        \end{align}
        \item\label{part:zero-energy_scattering_function_uniform_bound} We have the following uniform bound on the gradient
        \begin{align}\label{est:unscaled_uniform_grad_w_bound}
            \n{\grad f_{0}^{\mu}(x)} \le \frac{2}{\n{x}^2+R_0^2}\(\asc_0^{\mu}+\frac{\mathfrak{c}_0-\asc_0^{\mu}}{\mu}\zeta_0 \)\ .
        \end{align}
        \item Lastly, we have the bound
        \begin{align}\label{equ:semiclassical estimates,b0}
            0\le \asc_0^{\mu}-\frac{1}{4\pi}\intd \n{\grad f_{0}^{\mu}(x)}^2\dd x=\frac{1}{4\pi\mu}\intd v(x) (f_{0}^{\mu}(x))^2\dd x  \le \frac{\asc_0^{\mu}}{R_0}(\mathfrak{c}_0-\asc_0^{\mu})\ ,
        \end{align}
        which tends to zero as $\mu\rightarrow 0$.
    \end{enumerate}
\end{lem}

\subsection{Proof of Lemma \ref{lem:correlation_structure}}\label{appendix:neumann}
Noticing that
\begin{align*}
   &f^\varepsilon_{N, \ell}(x):=f^\mu_L(N^\beta\varepsilon^{2\kappa}x),\quad E_{N,l}^{\ve}=E_{L}^{\mu}\frac{\la(N)(N^{\be}\ve^{2\ka})^{3}}{N}\ ,\\
   &\mu=\frac{N^{1-\beta}\varepsilon^{2(1-\kappa)}}{\la(N)},
\quad L=N^{\beta}\ve^{2\kappa}l\ ,
\end{align*}
by scaling, we need only to consider the Neumann boundary problem
\begin{equation}\label{eq:neumann_boundary_problem}
\left\{
\begin{aligned}
&\lrs{-\mu\, \lapl+ v}f_{L}^{\mu}=E_{L}^{\mu}f_{L}^{\mu},\quad \text{for $|x|<L$}\\
&\text{$f_{L}^{\mu}(L)=1$ and $\pa_{r}f_{L}^{\mu}(L)=0$}
\end{aligned}
\right. \ ,
\end{equation}
with $L\ge R_0$. Using the estimates obtained in the previous section \ref{section:dirichlet problem,appendix}, we extend the $\mu=1$ results given in \cite[Appendix A]{erdos2006derivation} to the semiclassical regime $\mu\to 0$. The technical difficulty lies in the nonlinear dependence between the scattering length $a_{0}^{\mu}$ and the parameter $\mu$, which means that a linear scaling analysis would not work. Instead, the proof is based on the comparison principle and techniques of the construction of auxiliary functions.

\begin{lem}\label{lem:appendix_version_correlation_structures}
    Let $v$ be a nonnegative bounded function supported in $\{\n{x}\le R_0\}$ and $\asc_{0}^{\mu}$ be the corresponding scattering length.  Let $f_0^\mu$ be the solution to the Dirichlet Problem \eqref{def:zero-energy_scattering_problem}, $f^{\mu}_{L}$ be the ground state of the Neumann problem \eqref{eq:neumann_boundary_problem} and $E_{\mathrm{gs}}:=E_{L}^{\mu}$ the corresponding ground state energy. Then, for $L$ sufficiently large, we have:
    \begin{enumerate}[$(i)$]
        \item\label{part:groundstate_asymptotic_expansion_neumann} ground state energy $E_{\mathrm{gs}}$ has the asymptotic upper bound expression
        \begin{align}\label{eq:groundstate_asymptotic_expansion_neumann}
            E_{\mathrm{gs}} \le \frac{3\mu\asc_0^{\mu}}{L^3}\(1+\mathcal{O}\(\frac{R_0}{L}\)\)\quad \text{ as } \quad R_0/L\rightarrow 0\ .
        \end{align}

        \item\label{part:neumann_scattering_function_pointwise_bound} For all $\n{x}\le L$, we have
        \begin{align}\label{est:pointwise_for_gstate}
            \exp\(-\zeta_0^\frac12/\mu^{\frac12}\) \le f_{L}^{\mu}(x)\le 1\ ,
        \end{align}
        which also yields
        \begin{align}\label{est:pointwise_for_gstate2}
            1-f_{L}^{\mu}(x)\le\min\(1-\exp\(-\zeta_0^\frac12/\mu^{\frac12}\), \frac{\asc_0^{\mu}}{\n{x}}\)\ ,
        \end{align}
    where $\zeta_0:= R_0^2  \Nrm{v}{L^\infty}$.
In fact, we have the uniform bound
        \begin{align}
            \n{f_{0}^{\mu}(x)-f_{L}^{\mu}(x)}\le C \frac{\asc_0^{\mu}}{L}\ .
        \end{align}
        \item\label{part:approximation_of_scattering_length} We have that
        \begin{align}
            \intd v f_{L}^{\mu}\dd x  =& 4\pi \mu\,\asc_0^{\mu} + \mathcal{O}\(\frac{\asc_0^{\mu}}{L}\),\label{eq:approximation_of_scattering_length}\\
            \intd v(f_{L}^{\mu})^2\dd x  = &\intd v(f_{0}^{\mu})^2\dd x + \mathcal{O}\(\frac{\mu\,\asc_0^{\mu}}{L}\)\label{equ:Dirichlet,Neumann}\ .
        \end{align}
        \item\label{part:neumann_scattering_function_pointwise_gradient_bound} There exists a universal constant $C>0$ such that we have the following gradient estimates
        \begin{align}
            \n{\grad f_{L}^{\mu}(x)} \le  \frac{\asc_{0}^{\mu}+\mathcal{O}(R_0/L)}{\n{x}^2} \quad \text{ and } \quad \n{\grad f_{L}^{\mu}(x)} \le \frac{2+\mathcal{O}\(R_0/L\)}{\n{x}}\ .
        \end{align}
        \item\label{part:neumann_scattering_function_uniform_bound} Moreover,  there exists $C>0$, independent of $\mu$, such that we have the following bound
        \begin{align}\label{est:uniform_grad_w_bound_neumann}
            \n{\grad f_{L}^{\mu}(x)} \le  \frac{C}{\n{x}^2+R_0^2}\(\asc_0^{\mu}+\frac{\mathfrak{c}_0-\asc_0^{\mu}}{\mu}\zeta_0 \)
        \end{align}
        when $L$ is sufficiently large.
    \end{enumerate}
\end{lem}

\begin{proof}
	For simplicity, we omit the parameters and use the shorthands
	\begin{align*}
	\asc_{0}=\asc_{0}^{\mu},\quad \sfc_{1}=\sfc_{1}^{\mu},\quad f_{\mathrm{gs}}=f^{\mu}_{L},\quad f_{0}=f_{0}^{\mu},\quad w_{0}=1-f_{0}^{\mu}\ .
	\end{align*}
	
    To prove Part \eqref{part:groundstate_asymptotic_expansion_neumann}, we recall that the ground state energy of the Neumann problem is
    \begin{align*}
        E_{\mathrm{gs}} = \inf_{\varphi \in H^1(\n{x}\le L)} \frac{\int_{\n{x}\le L}\mu\n{\grad\varphi}^2+v\varphi^2\dd x}{\int_{\n{x}\le L}\varphi^2\dd x}\ .
    \end{align*}
    With an appropriate choice of trial function, let us show an upper bound for the ground state energy.  Noting that $f_{0}(r)$ is the solution to Dirichlet Problem \eqref{eq:zero-energy_scattering} and $w_0=1-f_{0}$, we see that
    \begin{align*}
        &\int_{\n{x}\le L} \mu \n{\grad f_{0}}^2+v\, (f_{0})^2\dd x\\
        &=\mu\int_{\n{x}=L} w_0 \frac{\bd w_0}{\bd n}\d S+\int_{\n{x}\le L} w_0 (-\mu\,\lapl w_0)+vw_0(w_0-1)+v(1-w_0)\dd x\\
        &\le \int_{\n{x}\le L} w_0 (-\mu\,\lapl w_0)+vw_0(w_0-1)+v(1-w_0)\dd x
        = 4\pi\mu\, \asc_0
    \end{align*}
    where we have used the fact that $w_0(r)$ is radially decreasing.
    Next, notice that
    \begin{align*}
        \int_{\n{x}\le L} (f_{0}(x))^2\d x =&\, 4\pi \int^L_0 m_0(r)^2\d r
         \ge\, 4\pi \int^L_{R_0} (r-\asc_0)^2\d r \ge \frac{4\pi}{3}\(L-R_0\)^3 .
    \end{align*}
    Hence, for $L$ much larger than $R_0$, we see that
    \begin{align}
        E_{\mathrm{gs}}\le \frac{3\mu \asc_0}{L^3\(1-\frac{R_0}{L}\)^3} \le \frac{3\mu\, \asc_0}{L^3}\(1+\mathcal{O}\(\frac{R_0}{L}\)\)\ .
    \end{align}

    To prove Part \eqref{part:neumann_scattering_function_pointwise_bound}, we start by making the observation that  $f_{\mathrm{gs}}$ satisfies
    \begin{align*}
    (-\mu\,\lapl+v)(f_{0}-f_{\mathrm{gs}})=-E_{\mathrm{gs}}f_{\mathrm{gs}}\leq 0
    \end{align*}
    on $\{\n{x}\le L\}$. Applying the maximum principle for elliptic operators (see, e.g., \cite[Chapter 3]{gilbarg1977elliptic}), we get that
    \begin{align}\label{equ:comparion,fd,fn}
  f_{0}(x)\leq f_{\mathrm{gs}}(x)\ ,
    \end{align}
    which yields the lower bound by using \eqref{equ:zero-energy_scattering_function_pointwise_bound}.
    Similarly, observe that
    \begin{align*}
        (-\mu\,\lapl+v)(f_{\mathrm{gs}}-f_{0}^{\mu}-\tfrac{2\asc_0}{L}+\tfrac{E_{\mathrm{gs}}}{6\mu}\n{x}^2)=E_{\mathrm{gs}}(f_{\mathrm{gs}}-1)-\(\tfrac{2\asc_0}{L}-\tfrac{E_{\mathrm{gs}}}{6\mu}\n{x}^2\)v \le 0
    \end{align*}
    when $L$ is sufficiently large. Again, by the maximum principle, we have that
    \begin{align*}
        f_{\mathrm{gs}}(x)\le f_{0}^{\mu}(x)+\frac{2\asc_0}{L}-\tfrac{E_{\mathrm{gs}}}{6\mu}\n{x}^2 \le f_{0}^{\mu}(x)+ C\frac{\asc_0}{L}
    \end{align*}
    for all $\n{x}\le L$.
    This completes the proof of Part~\eqref{part:neumann_scattering_function_pointwise_bound}.

    To prove Part~\eqref{part:approximation_of_scattering_length}, by the uniform bound from Part~\eqref{part:neumann_scattering_function_pointwise_bound}, we have that
    \begin{align*}
    \intd vf_{\mathrm{gs}}\dd x=\intd vf_{0}\dd x+\intd v(f_{\mathrm{gs}}-f_{0})=4\pi\mu\,\asc_{0}+\mathcal{O}\lrs{\frac{\asc_{0}}{L}}
    \end{align*}
    and
    \begin{align*}
    \intd vf_{\mathrm{gs}}^{2}\dd x=&\, \intd vf_{0}^{2}\dd x+\intd v(f_{\mathrm{gs}}-f_{0})(f_{\mathrm{gs}}+f_{0})\dd x\\
    \le &\, \intd vf_{0}^{2}\dd x+\frac{C\mu\,\asc_0}{L}\frac{1}{4\pi\mu}\intd v(f_{\mathrm{gs}}+f_{0})\dd x\ .
    \end{align*}

    For Part \eqref{part:neumann_scattering_function_pointwise_gradient_bound}, let $m_1(r):= r f_{\mathrm{gs}}(r)$, then we see that $m_1$ satisfies the equation
    \begin{align}\label{eq:second_order_diffeq_m_neumann}
    &-\mu\, m_1'' + \(v(r)-E_{\mathrm{gs}}\) m_1 = 0\ .
    \end{align}

    In the radiation region that $r\in[R_{0},L]$, we can solve explicitly for $m_1$ to get
    \begin{align*}
        m_1(r) = L \cos\(\sqrt{\tfrac{E_{\mathrm{gs}}}{\mu}}(r-L)\)+\sqrt{\frac{\mu}{E_{\mathrm{gs}}}}\sin\(\sqrt{\tfrac{E_{\mathrm{gs}}}{\mu}}(r-L)\)\ .
    \end{align*}
    Furthermore, since $\sqrt{E_{\mathrm{gs}}/\mu}\ll 1$, we could expand $m_1(r)$ and $m_1'(r)$ to the fourth order in $\sqrt{E_{\mathrm{gs}}/\mu}$ to find
    \begin{align*}
    m_1(r)=r-\frac{E_{\mathrm{gs}} L^3}{3\mu}+\mathcal{O}\(\frac{\asc_0}{L}\) \quad \text{ and } \quad m_1'(r)=1+\mathcal{O}\(\frac{\asc_0}{L}\)
    \end{align*}
    on $[R_0, L]$ when $L$ is sufficiently large. In particular, there exists some univerisal constant $c>0$ such that
    \begin{align*}
    \bd_r f_{\mathrm{gs}}(r)=&\frac{rm_1'(r)-m_1(r)}{r^{2}}\le  c\frac{\asc_0}{r^2}\\
    \bd_r f_{\mathrm{gs}}(r)=&\frac{m_1'(r)-f_{\mathrm{gs}}(r)}{r} =\frac{1-f_{\mathrm{gs}}(r)+\mathcal{O}(\asc_0/L)}{r}\ .
    \end{align*}

    In the scattering region that $r\in [0,R_{0}]$, consider the auxiliary function
    \begin{align}
    g(r)=rm_1'(r)-m_1(r)+\frac{E_{\mathrm{gs}}}{\mu}\int_{0}^{r}s^2 f_{\mathrm{gs}}(s)\dd s
    = \frac{1}{\mu}\int_{0}^{r}s^2 v(s)f_{\mathrm{gs}}(s)\dd s.
    \end{align}
    Then, for $r\in [0,R_{0}]$, we have that $g'(r)=rv(r)m_1(r)/\mu\geq 0$
    which implies $g$ is monotone nondecreasing and
    \begin{align*}
    g(r)\le&\, g(R_0)= R_0m_1'(R_0)-m_1(R_0)+\frac{E_{\mathrm{gs}}}{\mu}\int_{0}^{R_0}s^2 f_{\mathrm{gs}}(s)\dd s
    \le\, \frac{E_{\mathrm{gs}}\(L^3+R_0^3\)}{3\mu}+c\, \asc_0\ .
    \end{align*}
    Hence, by Part \eqref{part:groundstate_asymptotic_expansion_neumann}, we see that
    \begin{align*}
    \n{\bd_r f_{\mathrm{gs}}(r)} \le \frac{g(R_0)+\frac{E_{\mathrm{gs}}}{\mu}\int_{0}^{R_0}s^2 f_{\mathrm{gs}}(s)\dd s}{r^2} \le \frac{\asc_0+\mathcal{O}\(\frac{R_0}{L}\)}{r^2}\ .
    \end{align*}
    Lastly, we write
    \begin{align*}
    \bd_r f_{\mathrm{gs}}(r) =\frac{m_1'(r)+\frac{E_{\mathrm{gs}}}{\mu}\int_{0}^{r} m_1(s)\dd s-\frac{E_{\mathrm{gs}}}{\mu}\int_{0}^{r} m_1(s)\dd s-f_{\mathrm{gs}}(r)}{r}\ .
    \end{align*}
    Using a similar auxiliary function argument, one obtains the estimate
    \begin{align*}
    \n{\bd_r f_{\mathrm{gs}}(r)} \le  \frac{2+\mathcal{O}(\asc_0/L)}{r}\ .
    \end{align*}

    To prove Part \eqref{part:neumann_scattering_function_uniform_bound}, notice that in the radiation region, we have the estimate
    \begin{align*}
        (R_0^2+r^2)\n{\grad f_{\mathrm{gs}}(r)} \le \asc_0\(\frac{R_0^2}{r^2}+1\)\le C\asc_0\ .
    \end{align*}
    In the scattering region, using the fact that $f_{\mathrm{gs}}(r)\le 1$, we have that
    \begin{align*}
        (R_0^2+r^2)\n{\grad f_{\mathrm{gs}}(r)}\le&\, 2R_0^2\(\frac{E_{\mathrm{gs}}}{\mu r^2}\int_{0}^{r}s^2 f_{\mathrm{gs}}(s)\dd s
        + \frac{1}{\mu r^2}\int_{0}^{r}s^2 v(s)f_{\mathrm{gs}}(s)\dd s\)\\
        \le&\, 2\frac{R_0^3}{\mu}\(E_{\mathrm{gs}}+v_0\)\Nrm{f_{\mathrm{gs}}}{L^\infty(r\le R_0)}\\
        \le&\, 2\frac{R_0^3}{\mu}\(E_{\mathrm{gs}}+v_0\)\(\frac{\mathfrak{c}_0-\asc_0}{R_0}+C\frac{\asc_0}{L}\)
    \end{align*}
    which yields the desired result.
\end{proof}

\section{Tools from Dispersive PDE Theory}\label{appendix:Dispersive PDEs}
In this section, we shall record some useful tools in dispersive PDE theory and prove some properties of the  modified Gross--Pitaevskii equation, which we used throughout the paper. Most of the proofs are standard in the literature, but nevertheless we need to reproduce much of the proofs since we need to bookkeep the nontrivial dependence of the results on $\varepsilon$.
\subsection{Preliminary estimates}
Recall we have the equation
\begin{align}\label{eq:modified_Gross--Pitaevskii2}
    i\varepsilon\,\bd_t\phi = -\tfrac{\varepsilon^2}{2}\lapl \phi + (K\ast |\phi|^2)\phi
\end{align}
with the potential
    $K(x)=\lambda(N) (N^{\beta}\varepsilon^{2\kappa})^3 v(N^\beta \varepsilon^{2\kappa} x)f_{N, \ell}^{\varepsilon}(x)$.

For any fixed $\varepsilon \in (0, 1)$ and $N$, the proof of the global well-posedness and scattering of Equation \eqref{eq:modified_Gross--Pitaevskii2} in energy space is similar to the proof of the cubic nonlinear Schr\"odinger equation, which is standard in the literature (e.g. see \cite{ginibre1985scattering}).
It is also clear that if $\phi_{N, t}^{\varepsilon}$ is an $H^1$-solution of Equation \eqref{eq:modified_Gross--Pitaevskii2} then the solution satisfies the following conversation laws
\begin{align}
    M_t=&\, \Nrm{\phi_{N, t}^{\varepsilon}}{L^2_x}^2 = \Nrm{\phi^{\init}}{L^2_x}^2 = M_0\ ,\\
    E_t=&\, \textstyle\frac12\Nrm{\varepsilon \grad \phi_{N, t}^{\varepsilon}}{L^2_x}^2+ \frac12\intd (K\ast|\phi_{N, t}^{\varepsilon}|^2)\n{\phi_{N, t}^{\varepsilon}}^2\dd x = E_0\ .
\end{align}
To simplify the notations, we write $\phi_t = \phi^\varepsilon_{N, t}$ in this appendix. Certainly, to obtain uniform bounds, the restriction between $N$ and $\ve$ is needed. The exponential restriction \eqref{equ:restriction,theorem,gp,hc}, if needed, suffices for our goal.

Let us state some useful estimates.

\begin{lem}\label{lem:phi4_bound}
    Let $\phi$ be a solution to Equation~\eqref{eq:modified_Gross--Pitaevskii2}  with  initial data $\weight{\varepsilon\grad}\phi^{\init}\in L^2(\R^3)$. Then there exists $C$, dependent only on $E_0$, such that we have the estimate
    \begin{align}
         \Nrm{\rho^\varepsilon_t}{L^\infty_t L^2_x}\lesssim 1.
    \end{align}
\end{lem}

\begin{lem}[{\hspace{-0.05em}\cite[Lemmas A.5]{chen2022quantitative}}]\label{lem:quantative estimate for identity approximation}
    Let $W_{N}(x)=N^{3\be}h(N^{\be}x)-b_{1}\delta$, where $b_{1}=\intd h\dd x$. For any $0\leq s\leq 1$,
    \begin{align}\label{equ:quantative estimate for identity approximation}
    \Nrm{W_{N}*h}{L^{p}_x}\lesssim N^{-\be s}\nv{\lra{\nabla}^{s}h}_{L^{p}_x}
    \end{align}
    for any $1<p<\wq$. The implicit constant depends only on $\nv{\lra{x}h(x)}_{L^{1}_x}$.
\end{lem}

\subsection{Estimates for the modified GP equation: \texorpdfstring{$\beta=1$}{beta=1} case}
In this section, we consider the GP/BGP/SGP/HC regimes.
The Sobolev norm of the solution to Equation \eqref{eq:modified_Gross--Pitaevskii2} on any fixed interval $[0, T]$ grows polynomially in $\varepsilon^{-1}$. The proof is based on an argument introduced by Bourgain  in \cite{bourgain1996growth} (cf. \cite[Ch. V, Theorem 2.13]{bourgain1999global} ).

\begin{lemma} \label{lem:semiclasscial_propagation_of_regularity}
    Assume $s>1$. Let $\phi$ be a solution to \eqref{eq:modified_Gross--Pitaevskii2} in the GP or HC regime whose initial data satisfy: there exists $C_0>0$, independent of $N$ and $\varepsilon$, such that
    \begin{align*}
        &M_0\le C_0,\quad  E_0 \le C_0, \quad \text{ and }\quad \Nrm{\lra{\ve \nabla}^{s}\phi^{\init}}{L^2_x} \le C_0\ .
    \end{align*}
    Then there holds that
    \begin{align}\label{equ:polynomial growth,sobolev norm}
        \Nrm{\weight{\varepsilon\grad}^s\phi(t)}{L^2_x}\lesssim P(\ve^{-1},t).
    \end{align}
    where $P(x,y)$ is a polynomial function.
\end{lemma}

\section{Estimates for the Pair Excitation Function}\label{appendix:pair_excitation_estimates}
Let us summarize some estimates of the pair excitation function $k_{N,t}^{\ve}(x,y)$ defined as in Equation \eqref{def:pair_excitation_function}. See for example \cite[Appendix B]{benedikter2015quantitative} for more details.

By suppressing some indices, we write the pair excitation function as
\begin{align}\label{def:pair_excitation_function_appendix}
    k_{t}(x, y) =
        -N w_{N, \ell}^{\varepsilon}(x-y)  \phi_t(x)\phi_t(y) \ .
\end{align}
where $w_{N, \ell}^{\varepsilon}(x)$ corresponds to the Neumann Problem \eqref{def:correlation_structure_neumann_problem}, and $\phi_t=\phi_{N, t}^{\varepsilon}$ solves the modified Gross--Pitaevskii equation \eqref{eq:modified_Gross--Pitaevskii}.

\begin{lem}\label{lem:pair_excitation_estimates}
    Suppose the pair excitation function $k_t$ is defined as in Equation \eqref{def:pair_excitation_function_appendix}, then we have the following:
    \begin{enumerate}[$(i)$]
        \item\label{part:pair_excitation_est_for_k} There exists $C>0$, depending only on  the energy $E_0$, and constant $M>0$, such that
            \begin{align}
                &\Nrm{k_t}{\mathrm{HS}} \lesssim 1 ,\\
                &\Nrm{\varepsilon\grad_x k_t}{\mathrm{HS}}+ \Nrm{\varepsilon\grad_y k_t}{\mathrm{HS}}\le C\la^{\frac{1}{2}}\sqrt{N}\ ,\label{equ:estimate on k,H1}\\
                &\sup_{x \in \R^3}\Nrm{k_t(x, \cdot)}{L^2_y}=\sup_{x \in \R^3}\Nrm{k_t( \cdot, x)}{L^2_y} \le C\ve^{-M}\lra{t}^{M}\ ,\label{est:L-inftyL-2_of_k}\\
                &\Nrm{\varepsilon \grad_x(k_t\, \conj{k_t})}{\rm HS}+ \Nrm{\varepsilon \grad_y(k_t\, \conj{k_t})}{\rm HS} \lesssim 1\  . \label{est:derivative_of_composition_of_ks}
            \end{align}

            \item \label{part:pair_excitation_est_for_p} The operators $p_{N, t}^\varepsilon, r_{N, t}^\varepsilon, \sh_{N,t}^\varepsilon$ are Hilbert--Schmidt operators and the following bounds hold
            \begin{align}
                &\Nrm{\sh_{N, t}^{\varepsilon}}{\mathrm{HS}}  \le C \Nrm{k_t}{\rm HS}e^{C\, \Nrm{k_t}{\rm HS}}\lesssim 1 \ ,\label{equ:estimate on u,L2}\\
                &\Nrm{p_{N, t}^{\varepsilon}}{\mathrm{HS}}\le C \Nrm{k_t}{\rm HS}^2e^{C\, \Nrm{k_t}{\rm HS}}\lesssim 1\ ,\\
                &\Nrm{r_{N, t}^{\varepsilon}}{\mathrm{HS}}\le C \Nrm{k_t}{\rm HS}^3e^{C\,\Nrm{k_t}{\rm HS}}\lesssim 1\ .
            \end{align}
            Similarly, we see that $\varepsilon \grad_1 p_{N, t}^\varepsilon, \varepsilon \grad_2 p_{N, t}^\varepsilon, \varepsilon \grad_1r_{N, t}^\varepsilon, \varepsilon \grad_2 r_{N, t}^\varepsilon$ are also Hilbert--Schmidt operators with the bounds
            \begin{align}
                &\Nrm{\varepsilon \grad_\ii p_{N, t}^{\varepsilon}}{\mathrm{HS}}\le C \Nrm{\varepsilon \grad_\ii(k_t\, \conj{k_t})}{\rm HS}e^{C\,\Nrm{k_t}{\rm HS}} \lesssim 1 \ , \label{equ:estimate on p,H1}\\
                &\Nrm{\varepsilon \grad_\ii r_{N, t}^{\varepsilon}}{\mathrm{HS}}\le C \Nrm{\varepsilon \grad_\ii(k_t\, \conj{k_t})}{\rm HS}\Nrm{k_t}{\rm HS}e^{C\, \Nrm{k_t}{\rm HS}}\lesssim 1 \ .\label{equ:estimate on r,H1}
            \end{align}
            \item \label{part:pair_excitation_pointwise_est_for_p} From the definition and Lemma \ref{lem:correlation_structure}, we have the following pointwise bounds
            \begin{align}
                &\n{k_t(x, y)} \le C\min\(N, \frac{\asc_{0}^{\mu}}{\varepsilon^{2}\n{x-y}}\) \chi_{\{\n{x-y}\le \ell\}}\n{\phi_t(x)}\n{\phi_t(y)}\ , \label{est:k_pointwise}\\
                & \n{p_{N, t}^\varepsilon(x, y)} \le  \,  C\varepsilon^{-6}\n{\phi_t(x)} \n{\phi_t(y)}e^{C\,\Nrm{k_t}{\rm HS}}\ ,\\
                & \n{r_{N, t}^\varepsilon(x, y)} \le \, C\varepsilon^{-6}\n{\phi_t(x)} \n{\phi_t(y)}e^{C\,\Nrm{k_t}{\rm HS}}\ . \label{est:r_pointwise}
            \end{align}
            Furthermore, by Lemma~\ref{lem:semiclasscial_propagation_of_regularity}, we also have
            \begin{align}
                &\sup_{x \in \R^3}\Nrm{p_{N, t}^\varepsilon(x, \cdot)}{L^2_y} +\sup_{x \in \R^3}\Nrm{r_{N, t}^\varepsilon(x, \cdot)}{L^2_y}+\sup_{x \in \R^3}\Nrm{\sh_{N, t}^\varepsilon(x, \cdot)}{L^2_y}\le  C\ve^{-M}\lra{t}^{M}\ .
            \end{align}

    \end{enumerate}
\end{lem}

\begin{lem}\label{lem:pair_excitation_estimates_time-derivative}
 Suppose the pair excitation function $k_t$ is defined as in Equation \eqref{def:pair_excitation_function_appendix}, then we have the following
    \begin{enumerate}[$(i)$]
        \item \label{Part:time-derivative_est_for_k} There exists $C>0$, depending only on the energy $E_0$, and constant $M>0$ such that
            \begin{align}
                \Nrm{\varepsilon\dot k_t}{\mathrm{HS}}+\Nrm{\varepsilon^2\ddot k_t}{\mathrm{HS}}
                +\Nrm{\varepsilon^2 \grad_\ii(\dot k_t\, \conj{k_t})}{\rm HS}+ \Nrm{\varepsilon^2 \grad_\ii(k_t\, \conj{\dot k_t})}{\rm HS}\le& C\ve^{-M}\lra{t}^{M} \label{est:HS_time-derivative_k}\
                ,\\
                \Nrm{k_t(-\varepsilon^2\lapl)\conj{k_t}}{\rm HS}+\Nrm{\varepsilon\dot k_t(-\varepsilon^2 \lapl) \conj{k_t}}{\rm HS}\le & C\ve^{-M}\lra{t}^{M} \ .\label{est:HS_time-derivative_and_2spatial-derivative_k}
            \end{align}
            Similarly, we also have the bounds
            \begin{align}
                &\sup_{x \in \R^3}\Nrm{\varepsilon\dot k_t(x, \cdot)}{L^2_y} +\sup_{x \in \R^3}\Nrm{\varepsilon^2\ddot k_t(x, \cdot)}{L^2_y}\lesssim \ve^{-M}\lra{t}^{M}\ .
            \end{align}
            \item \label{Part:time-derivative_est_for_p} The operators $\varepsilon\dot p_{N, t}^\varepsilon, \varepsilon\dot r_{N, t}^\varepsilon$ are Hilbert--Schmidt operators and we have the bounds
            \begin{align}
                &\Nrm{\varepsilon\dot p_{N, t}^{\varepsilon}}{\mathrm{HS}}\le C\Nrm{k_t}{\rm HS} \Nrm{\varepsilon\dot k_t}{\rm HS}e^{C\, \Nrm{k_t}{\rm HS}}\lesssim \ve^{-M}\lra{t}^{M}\ ,\label{est:HS_2-time-derivative_p}\\
                &\Nrm{\varepsilon\dot r_{N, t}^{\varepsilon}}{\mathrm{HS}}\le C \Nrm{k_t}{\rm HS}^2\Nrm{\varepsilon\dot k_t}{\rm HS}e^{C\, \Nrm{k_t}{\rm HS}}\lesssim \ve^{-M}\lra{t}^{M}\ . \label{est:HS_2-time-derivative_r}
            \end{align}
            Similarly, for $\varepsilon^2 \grad_1 \dot p_{N, t}^\varepsilon, \varepsilon^2 \grad_2 \dot p_{N, t}^\varepsilon, \varepsilon^2 \grad_1\dot r_{N, t}^\varepsilon, \varepsilon^2 \grad_2 \dot r_{N, t}^\varepsilon$, we have
            \begin{align}
                \Nrm{\varepsilon^2\grad_\ii\, \dot p_{N, t}^\varepsilon}{\rm HS}+ \Nrm{\varepsilon^2\grad_\ii\, \dot r_{N, t}^\varepsilon}{\rm HS}
                \lesssim \ve^{-M}\lra{t}^{M} \label{est:HS_2-time-derivative_and_spatial-derivative_r}\ .
            \end{align}
            \item \label{Part:time-derivative_pointwise_est_for_k} Moreover, we also have the following estimates
            \begin{align}
                \sup_{x, \in \R^3}\Nrm{\varepsilon\dot p_{N, t}^\varepsilon(x, \cdot)}{L^2_y} +
                \sup_{x, \in \R^3}\Nrm{\varepsilon\dot r_{N, t}^\varepsilon(x, \cdot)}{L^2_y} \lesssim\ve^{-M}\lra{t}^{M}. \notag
            \end{align}
        \end{enumerate}
\end{lem}

\renewcommand{\bibname}{\centerline{Bibliography}}
\bibliographystyle{abbrv}
\bibliography{Euler}

\end{document}